\documentclass{article}

\usepackage{amsthm}
\usepackage{amsmath}
\usepackage{amsxtra}
\usepackage{amssymb}
\usepackage{thmtools}
\usepackage[colorlinks=true,linkcolor=blue]{hyperref}
\usepackage{mathrsfs}
\usepackage{undertilde}
\usepackage{turnstile}
\usepackage[retainorgcmds]{IEEEtrantools}
\usepackage{enumitem}

\declaretheorem[name=Definition,style=definition,
numberwithin=section]{dfn}

\declaretheorem[name=Theorem,style=plain,sibling=dfn]{tm}
\declaretheorem[name=Fact,style=plain,sibling=dfn]{factmine}
\declaretheorem[name=Lemma,style=plain,sibling=dfn]{lem}
\declaretheorem[name=Proposition,style=plain,sibling=dfn]{propos}
\declaretheorem[name=Corollary,style=plain,sibling=dfn]{cor}
\declaretheorem[name=Remark,style=definition,sibling=dfn]{rem}

\declaretheorem[name=Conjecture,style=definition,sibling=dfn]{conj}

\declaretheoremstyle[headfont=\scshape]{claimstyle}
\declaretheorem[name=Claim,style=claimstyle]{clm}
\declaretheorem[name=Claim,style=claimstyle,numbered=no]{clm*}

\newcommand{\nt}{\mathrm{nt}}

\newcommand{\illfp}{\mathrm{illfp}}
\newcommand{\Lll}{\mathscr{L}}

\newcommand{\cHull}{{\mathrm{c}\Hull}}
\newcommand{\iso}{\cong}
\newcommand{\RR}{\mathbb R}
\newcommand{\PP}{\mathbb P}
\newcommand{\sub}{\subseteq}
\newcommand{\cross}{\times}
\newcommand{\all}{\forall}
\newcommand{\om}{\omega}
\newcommand{\pow}{\mathcal{P}}
\newcommand{\OR}{\mathrm{OR}}

\newcommand{\Hull}{\mathrm{Hull}}
\newcommand{\cut}{\backslash}

\newcommand{\Ll}{\mathcal{L}}

\newcommand{\rg}{\mathrm{rg}}
\newcommand{\dom}{\mathrm{dom}}
\newcommand{\crit}{\mathrm{crit}}
\newcommand{\rest}{\!\upharpoonright\!}
\newcommand{\com}{\circ}

\newcommand{\Ult}{\mathrm{Ult}}
\newcommand{\sats}{\models}
\newcommand{\elem}{\preccurlyeq}

\newcommand{\AC}{\mathrm{AC}}
\newcommand{\DC}{\mathrm{DC}}
\newcommand{\HOD}{\mathrm{HOD}}
\newcommand{\ZFC}{\mathrm{ZFC}}
\newcommand{\ZF}{\mathrm{ZF}}

\newcommand{\dirlim}{\mathrm{dir lim}}
\newcommand{\id}{\mathrm{id}}
\newcommand{\conc}{\ \widehat{\ }\ }
\newcommand{\trancl}{\mathrm{trancl}}
\newcommand{\Vop}{\mathrm{Vop}}
\newcommand{\OD}{\mathrm{OD}}
\newcommand{\psub}{\subsetneq}

\newcommand{\tu}{\textup}

\DeclareMathOperator{\Th}{Th}

\DeclareMathOperator{\card}{card}
\DeclareMathOperator{\cof}{cof}

\DeclareMathOperator{\rank}{rank}

\author{Farmer Schlutzenberg\\
\ \\
Institute for Discrete Mathematics and Geometry\\
TU Vienna\\
farmer.schlutzenberg@gmail.com\\
}
\title{On the consistency of ZF with an elementary embedding from
$V_{\lambda+2}$ into $V_{\lambda+2}$\footnote{This is the author accepted version of \cite{con_lambda_plus_2_pub_online_first}.
For the published version see \url{https://doi.org/10.1142/S0219061324500132}}
\footnote{\copyright 2024. This work (author accepted version) is openly licensed via CC BY 4.0; see \url{https://creativecommons.org/licenses/by/4.0/}}}

\begin{document}

\maketitle

\begin{abstract}
According to a theorem due to Kenneth Kunen,
under ZFC, 
there is no ordinal $\lambda$ and non-trivial elementary embedding $j:V_{\lambda+2}\to 
V_{\lambda+2}$.
His proof relied on the Axiom of Choice (AC),
and no proof from ZF alone  has been discovered.

$I_{0,\lambda}$ is the assertion, introduced by W.~Hugh Woodin,
that $\lambda$ is an ordinal and there is an elementary embedding $j:L(V_{\lambda+1})\to L(V_{\lambda+1})$ with critical point ${<\lambda}$. And $I_0$ 
asserts that $I_{0,\lambda}$ holds for some $\lambda$.
The axiom $I_0$ is one of the strongest large cardinals not 
known to be inconsistent with AC.
It is usually
studied assuming ZFC in the full universe $V$ (in which case $\lambda$ must be a limit ordinal), but we  assume only ZF. 

We prove, assuming ZF + $I_{0,\lambda}$ +  ``$\lambda$ is an even ordinal'', that  there is a proper 
class transitive inner model $M$ containing $V_{\lambda+1}$
and satisfying
ZF + $I_{0,\lambda}$ + ``there is an elementary embedding 
$k:V_{\lambda+2}\to 
V_{\lambda+2}$'';  in fact we will have $k\sub j$,
where $j$ witnesses $I_{0,\lambda}$ in $M$.
 This result was first proved by the author under the added assumption that $V_{\lambda+1}^\#$ exists; Gabe Goldberg noticed that this extra assumption was unnecessary.
If also $\lambda$ is a limit ordinal
and $\lambda$-$\DC$  holds in $V$,
then the model $M$ will also satisfy $\lambda$-$\DC$.

We show 
that ZFC + ``$\lambda$ is even'' + $I_{0,\lambda}$
implies $A^\#$ exists for every $A\in V_{\lambda+1}$,
but if consistent, this theory does not  imply $V_{\lambda+1}^\#$ exists.
\end{abstract}

\section{The Kunen inconsistency}

A long standing open question in set theory is whether $\ZF$ suffices to prove
that there is no (non-trivial)\footnote{Generally
we will omit the phrase ``non-trivial'' in this context,
where it should obviously be there; sometimes
also ``elementary''.} elementary embedding
\[ j:V\to V.\]
Such embeddings were introduced by William N.~Reinhardt in 
\cite{reinhardt_diss}, 
\cite{reinhardt_remarks},
and the critical point of such an embedding is now known as a Reinhardt 
cardinal. A fundamental theorem is Kunen's result \cite{kunen_no_R} 
that
such an 
embedding is inconsistent with $\ZFC$. His proof involves infinitary 
combinatorics in $V_{\lambda+2}$, and
he in fact shows under $\ZFC$ that there is no elementary
\[ j:V_{\lambda+2}\to V_{\lambda+2}.\]
These discoveries led set theorists to ask various natural questions, such as:
\begin{itemize}
\item[--] Does $\ZF$ suffice to prove there is no elementary $j:V\to V$?
\item[--] Does $\ZF$ suffice to prove there is no elementary $j:V_{\lambda+2}\to 
V_{\lambda+2}$?
\item[--] Does $\ZFC$ prove there is no elementary $j:V_{\lambda+1}\to 
V_{\lambda+1}$?
\end{itemize}
However, if a given large cardinal notion is consistent, then
of course we will 
never know so; the best we can do is to prove
relative consistencies, and to explore the consequences of the large cardinal. 
Moreover, if $\ZF$ is in fact consistent with $k:V_{\lambda+2}\to 
V_{\lambda+2}$, this might suggest that defeating $j:V\to V$ without $\AC$ 
could demand significant modifications to Kunen's $\ZFC$ argument
(and other arguments which have appeared since).

The main aim of this paper is to show that the theory
$\ZF+$``$k:V_{\lambda+2}\to V_{\lambda+2}$ is elementary''
is 
consistent relative to large cardinals which have been extensively studied 
in the context of $\ZFC$.

In \S\ref{sec:elementary_lambda+2}, we give an overview of the main theorem, and in \S\ref{sec:background} we give some background definitions and facts.
In \S\ref{sec:main}, we state  the main result in detail and give its proof.
In the later sections, we make some further related observations.
 In \S\ref{sec:embeddings},
we analyze the existence and uniqueness of extensions of embeddings
related to $I_{0,\lambda}$-embeddings. In \S\ref{sec:HOD}, assuming 
$I_{0,\lambda}$, we prove
some facts on $\HOD^{L(V_{\lambda+1})}_N$ for a certain inner model $N\sub L(V_{\lambda+1})$
which arises in the proof of the main theorem. In \S\ref{sec:sharps},
we prove that $\ZF+I_{0,\lambda}$  implies that $A^\#$ exists for every $A\in 
V_{\lambda+1}$. In \S\ref{sec:consistencies},
we collect some corollaries on relative consistencies. Most of the results in the paper hold for arbitrary 
even ordinals
$\lambda$,
and $I_{0,\lambda}$ in this sense, as discussed in 
\S\ref{sec:elementary_lambda+2}.

 \subsection{Notation}\label{subsec:notation}
 Given a structure $N$ and $X\sub N$,
 $\Hull^N(X)$ denotes the set of all $y\in N$ such that
 $y$ is definable over $N$ from parameters in $X$,
 and  $\cHull^N(X)$ is the transitive collapse of this hull,
 assuming $N$ is transitive
 (with signature including the language of set theory)
and $\Hull^N(X)$ is extensional.
We write $\Th^N(X)$ for the elementary theory of $N$ in parameters in $X$.

Given an embedding $j$, we write $\kappa_0(j)=\crit(j)$ 
for the critical point 
of $j$, and $\kappa_\om(j)=\sup_{n<\om}\kappa_n$ where $\kappa_0=\kappa_0(j)$
and $\kappa_{n+1}=j(\kappa_n)$ (when well defined).

Regarding ultrapowers, we write $\Ult_0(M,\mu)$
for the ultrapower of structure $M$ modulo the $M$-ultrafilter  or $M$-extender $\mu$,
formed using only functions in $M$.\footnote{The notation ``$0$'' might look arbitrary
here, but in certain contexts one also
considers variants $\Ult_n$ with $0<n\leq\om$,
which can modify the class of functions used to form the ultrapower.}
In this context, if $f$ is a function used in forming the ultrapower,
then $[f]^{M,0}_\mu$ denotes the element represented by $f$,
but normally we abbreviate this with just
$[f]^M_\mu$ or just $[f]$, if the meaning is clear.
Note that because we do not assume AC, \L o\'{s}'s Theorem is not
always immediate for our ultrapowers. We say that the ultrapower
``satisfies \L o\'{s}'s Theorem'' to mean that the usual statement of \L o\'{s}'s
Theorem goes through with respect to this ultrapower
(basically that the ultrapower satisfies $\varphi([f])$ iff
$M\sats\varphi(f(k))$ for $\mu$-measure one many $k$).
(This is an abuse of terminology,
as otherwise when we write ``$U$ satisfies $\psi$''
we mean $U\sats\psi$ in the model theoretic sense.)

Let $\lambda$ be a limit ordinal and $X$ a set.
 A \emph{tree of height $\leq\lambda$ on $X$}
 is a set $\mathscr{T}\sub{^{<\lambda}X}$
 which is closed under initial segment.
 That is, each element of $\mathscr{T}$ is a function $f$
 such that $\alpha=\dom(f)\in\lambda$ and $f:\alpha\to X$,
 and $f\rest\beta\in\mathscr{T}$ for each $f\in\mathscr{T}$ and $\beta<\dom(f)$.
 For $\alpha\leq\lambda$ and $f:\alpha\to X$,
 we say that $f$ is a \emph{maximal branch} of $\mathscr{T}$
 iff either (i) $f\in\mathscr{T}$ but there is no $g\in\mathscr{T}$
 with $f\psub g$,
 or (ii) $\alpha$ is a limit ordinal or $\alpha=0$, and $f\rest\beta\in\mathscr{T}$
 for all $\beta<\alpha$ but $f\notin\mathscr{T}$.
 
Recall that for $\lambda$ a limit ordinal,
$\lambda$-$\DC$ is the statement
 that for every set $X$ and every tree $\mathscr{T}$ of height $\leq\lambda$ on 
$X$, there is a maximal branch of $\mathscr{T}$.\footnote{We remark 
that the 
possibly more standard definition of $\lambda$-$\DC$
is superficially slightly different; it asserts that for every non-empty set $X$
and every $R\sub({^{<\lambda}X})\cross X$ such that
\[ \text{ for every }f\in{^{<\lambda}X}\text{
there is some }x\in X\text{ such that }(f,x)\in R,\]
there is $f\in{^\lambda}X$ such that $(f\rest\alpha,f(\alpha))\in R$
for every $\alpha<\lambda$.

The two forms are equivalent modulo $\ZF$. For given a tree 
$\mathscr{T}$
of height $\leq\lambda$ on set $X\neq\emptyset$, define
$R\sub({^{<\lambda}X})\cross X$
by putting $(f,x)\in R$ iff either (i) $f\conc\left<x\right>\in\mathscr{T}$
or (ii) there is no $y\in X$ such that $f\conc\left<y\right>\in\mathscr{T}$.
Clearly $R$ satisfies the requirements for the second form of $\lambda$-$\DC$,
and letting $f:\lambda\to X$ be as there,
note that $f\rest\alpha$ is a maximal branch through $\mathscr{T}$
for some $\alpha\leq\lambda$. Conversely,
let $X,R$ be as in the second form. Let $\mathscr{T}$
be the tree of all functions $f:\alpha\to X$ where $\alpha<\lambda$
and $(f\rest\beta,f(\beta))\in R$ for all $\beta<\alpha$. Let $f$
be a maximal branch through $\mathscr{T}$. Then easily
$\dom(f)=\lambda$, which suffices.}\footnote{Many
authors use the notation $\DC_\lambda$.
We use $\lambda$-$\DC$  to distinguish it from
another notion: $\DC_\RR$ normally denotes
the version of $\om$-$\DC$ in which the set $X$ is only allowed to be $\RR$.}

Given a transitive set $X$, $\Theta_X$
denotes the supremum of all ordinals $\alpha$
such that $\alpha$ is the surjective image of $X$.

We discuss sharps and Silver indiscernibles in Remark \ref{rem:Silver}.

Given a structure $M$ for the language of set theory
(either a set or proper class),
and given $n\leq\om$, we write
we write $\mathscr{E}_n(M)$ for the collection of $\Sigma_n$-elementary embeddings $j:M\to M$,
and $\mathscr{E}(M)=\mathscr{E}_\om(M)$.
(If $M$ is proper class, this could in general
be beset with technical complications,
as the elements of $\mathscr{E}_n(M)$
are then themselves proper classes.
But in the case we will only actually use the notation when simultaneously 
restricting to embeddings $j$ which are determined by some set, avoiding such complications in practice.)
We also write $\mathscr{E}_{\nt}(M)=\mathscr{E}(M)\cut\{\id\}$ for the set of non-trivial (that is, non-identity) elementary embeddings $M\to M$; likewise for $\mathscr{E}_{\nt,n}(M)$.

\section{$\ZF+ j:V_{\lambda+2}\to V_{\lambda+2}$}\label{sec:elementary_lambda+2}

Recall that the large cardinal axiom $I_{0}$,
introduced by Woodin,
asserts that there is  an
ordinal $\lambda$ and an elementary embedding
$j\in\mathscr{E}(L(V_{\lambda+1}))$
with critical point $\crit(j)<\lambda$.

The hypothesis $I_0$ sits just below the boundary
of where Kunen's argument applies. It and strengthenings thereof have 
been 
studied by  Cramer, Dimonte, Shi, Woodin and others, and an extensive 
structure theory has been developed with strong analogues to the Axiom 
of Determinacy; see for example \cite{cramer}, \cite{dimonte}, \cite{shi_higher_degrees} and \cite{woodin_sem2}.
Many others have also studied  large cardinal
notions in this vicinity.
Although $I_{0}$,
witnessed by $j\in\mathscr{E}(L(V_{\lambda+1}))$, implies that the Axiom of Choice $\AC$
fails in $L(V_{\lambda+1})$, the hypothesis is usually considered
assuming $\ZFC$ in the background universe $V$.
However, $\AC$ will be mostly  unimportant here, and we assume in 
general only $\ZF$.  In this more general context, we make the following definition:

\begin{dfn} Assume $\ZF$ and let $\lambda$ be an even ordinal.
An \emph{$I_{0,\lambda}$-embedding}
 is an elementary
embedding $j\in\mathscr{E}(L(V_{\lambda+1}))$
with $\crit(j)<\lambda$ (hence $\crit(j)<\eta$
where $\eta$ is the largest limit ordinal $\leq\lambda$).

And $I_{0,\lambda}$ is the assertion that
there exists an $I_{0,\lambda}$-embedding.
\end{dfn}

\begin{rem}[Canonical extension]\label{rem:standard_extension}
Using the results of \cite{cumulative_periodicity_pub_online_first} or 
\cite{goldberg_even_ordinals_v3},
we will prove most of our results for arbitrary even ordinals
$\lambda$, not just limits.
This actually makes virtually no difference
to the arguments, and moreover, the reader unfamiliar with
\cite{cumulative_periodicity_pub_online_first} can pretty much
just replace every instance of ``even'' in the paper with ``limit'',
without losing much. In \S\ref{sec:sharps} the argument
differs a little more between the limit and successor cases, however.

Assume $\ZF$. A standard observation is that if $j\in\mathscr{E}_{\nt}(V_\lambda)$ 
 where $\lambda$ is a limit, then 
there is a unique
 possible extension of $j$ to some $j^+\in\mathscr{E}(V_{\lambda+1})$; this is defined
\[ j^+(A)=\bigcup_{\alpha<\lambda}j(A\cap V_\alpha). \]
In fact this is the unique possible candidate
for $j^+\in\mathscr{E}_1(V_{\lambda+1})$ extending $j$.

The key result we need from  \cite{cumulative_periodicity_pub_online_first}
is that this fact actually generalizes to all \emph{even} $\lambda$ 
(that is, $\lambda=\eta+2n$ for some
 $\eta\in\mathrm{Lim}\cup\{0\}$ (that is, $\eta$ is either a limit ordinal or $0$) and $n<\om$):
if $j\in\mathscr{E}_1(V_{\lambda+1})$, then  $j$ is also the \emph{canonical extension}
$(j\rest V_\lambda)^+$ of $j\rest V_\lambda$,
but if $\lambda$ is not a limit,  the definition of the canonical
extension is of course
 not so obvious. Moreover,
$j$ is definable over $V_{\lambda+1}$
from the parameter $j\rest V_\lambda$, and uniformly so in even 
ordinals $\lambda$ and 
$j$.
These facts allow one to lift much of the standard theory
of $I_{0,\lambda}$ to arbitrary even $\lambda$
(a project on which Gabe Goldberg has begun work; this is the topic of 
\cite{goldberg_even_ordinals_v3}).

 Actually the argument in \cite{cumulative_periodicity_pub_online_first}
 shows something a little more general:
 Let $\mathscr{L}$ be a transitive inner model of ZF and $\lambda\in\mathscr{L}$
  be an even ordinal. Let $j\in\mathscr{E}_1(V_{\lambda+1}^{\mathscr{L}})$ with
  $k=j\rest V_{\lambda}^{\mathscr{L}}\in \mathscr{L}$.
 Then $j\in \mathscr{L}$, and hence, 
applying
 \cite{cumulative_periodicity_pub_online_first}
 in $\mathscr{L}$,
 we have that $\mathscr{L}\sats$``$j=k^+$''. Note that
 only in the case that $V_{\lambda+1}^{\mathscr{L}}=V_{\lambda+1}$ would this
conclusion follow literally
 from the results of \cite{cumulative_periodicity_pub_online_first}.
 But without assuming that $j\in\Lll$, we can simply apply the usual (first-order in parameter $k$) definition 
 of $k^+$
over $V_{\lambda+1}^{\Lll}$,
and observe that this results in $j$,
and therefore $j\in\Lll$.

For most of the paper, the previous two paragraphs probably suffice as a black box for the canonical extension,
but in 
\S\ref{sec:sharps}, the reader will probably need
to know the details of the definition
for the successor case to make proper sense.

Note that even  in the case that $\lambda$ is a limit, we have
a more general context than standard $I_{0,\lambda}$,
because we do not assume $\AC$;
 one consequence of this is that $\lambda$
need not be the supremum of the critical sequence
of $j$.\footnote{Recall that under $\ZFC$, $I_{0,\lambda}$ implies
$\lambda=\kappa_{\om}(j)$;
this yields that $\lambda$ is a strong limit cardinal,
$V_\lambda$ has  a wellorder in $V_{\lambda+1}$,
and $L(V_{\lambda+1})\sats\lambda$-$\DC$.}
\end{rem}

The main relative consistency result we prove is the following.
The author first proved a slightly weaker result, in
 which the axiom ``$V_{\lambda+1}^\#$ exists''
 was also incorporated in the theory assumed consistent (but then reached the same 
conclusion).
See Remark \ref{rem:Silver} for a brief discussion of $V_{\lambda+1}^{\#}$. Goldberg then observed that the
assumption of $V_{\lambda+1}^\#$ could be dispensed
with, by using instead some further calculations of Woodin's from 
\cite{woodin_sem2}, establishing the stronger result stated below. Likewise
for the detailed version Theorem \ref{tm:main}.
Although the theorem itself seems to be 
new, the main methods in the proof are in fact old,
and in the most part due to Woodin. There are a couple of places
where one seems to need to adapt the details of standard kinds
of arguments a little, but the main methods
are well known.

\begin{tm}\label{tm:simpler}
If the theory
 \[ \ZF+\text{``there is a limit }\lambda\text{  such that 
}\lambda\text{-}\DC+I_{0,\lambda}\text{''}\]
is consistent then so is 
\[\begin{array}{rcl} \ZF&+&\text{``there is  
a limit }\lambda\text{
such that }\lambda\text{-}\DC+I_{0,\lambda}\\
&&\text{ and there is  an elementary  }
k:V_{\lambda+2}\to 
V_{\lambda+2}\text{''}.\end{array}\]

More generally, let $n$ be an integer of the meta-theory.
 If the theory
 \[ \ZF+\text{``there are even ordinals }\bar{\lambda},\lambda\text{
such that }\lambda=\bar{\lambda}+2n\text{ and }I_{0,\lambda}\text{''}\]
is consistent then so is 
\[\begin{array}{rcl} \ZF&+&\text{``there are even ordinals }\bar{\lambda},\lambda\text{ such that }\lambda=\bar{\lambda}+2n\text{
and }I_{0,\lambda}\\
&&\text{ and there is  an elementary  }
k:V_{\lambda+2}\to 
V_{\lambda+2}\text{''}.\end{array}\]
\end{tm}

We will in fact get much more information than that stated above.
 Let $\lambda$ be even and $\Lll=L(V_{\lambda+1})$
and assume $I_{0,\lambda}$, as witnessed by $j\in\mathscr{E}(\Lll)$. We will show in Proposition \ref{prop:ultrapowers} that $j(\lambda)=\lambda$.
So letting
$j'=j\rest 
V_{\lambda+2}^{\Lll}$,
then
\[ j':V_{\lambda+2}^{\Lll}\to 
V_{\lambda+2}^{\Lll} \]
is elementary. 
Assume also that $j$ is \emph{proper};
this is discussed in \ref{dfn:proper}, and means basically that $j$ is determined by $j'$; in particular, $\mathscr{L}[j]=\mathscr{L}[j']$, where $\mathscr{L}[k]=L(V_{\lambda+1},k)$.
Woodin asked in Remark 26 of \cite{woodin_sem2} (pp.~160--161), 
assuming further that $\ZFC$ holds in $V$ (so $\lambda$ is a limit),
whether 
\begin{equation}\label{eqn:neq} V_{\lambda+2}^{\Lll[j]}\neq 
V_{\lambda+2}^{\Lll} \end{equation}
must hold. This is actually just an 
instance of a more general question he asked.

Theorem \ref{tm:main} will show that  in fact, (\ref{eqn:neq}) \emph{cannot}  hold
(just assuming $\ZF$ in $V$ and $\lambda$ even);
that is, we actually have
\[ V_{\lambda+2}^{\Lll[j]}=V_{\lambda+2}^{\Lll}\]
(though as mentioned above, Woodin's question was more general,
and 
Theorem \ref{tm:main} only literally applies to the one case of his question  stated above).
It will follow that
\[ \Lll[j]\sats\ZF+I_{0,\lambda}+\text{``}j':V_{\lambda+2}\to V_{\lambda+2}\text{ 
is elementary''},\]
and  if $\bar{\lambda}$ is the largest limit ordinal $\leq\lambda$
and 
$\bar{\lambda}$-$\DC$ holds in $V$ then we
will also get that
$\Lll[j]\sats\bar{\lambda}$-$\DC$, giving the main consistency result.

We will simultaneously address another related question raised by Woodin,
which deals with sharps, and in particular, $V_{\lambda+1}^{\#}$:

\begin{rem}\label{rem:Silver}
Let $X$ be a transitive set.
Recall that $X^\#$, if it exists, is the theory of
Silver indiscernibles for $\Lll=L(X\cup\{X\})$; its existence
is equivalent to the existence of a $j\in\mathscr{E}_{\nt}(\Lll)$ with $\crit(j)>\rank(X)$ (another fact proved by Kunen).
Recall the Silver indiscernibles are a club proper class $\mathscr{I}$ of 
ordinals,
which are indiscernibles for the model $\Lll$ with respect to all 
parameters
in $X\cup\{X\}$, and such that every element of $\Lll$
is definable (equivalently, $\Sigma_1$-definable) from finitely many 
elements of $\mathscr{I}\cup X\cup\{X\}$. 
Note $X^\#$ relates to $\Lll$
just as $0^\#$ relates to $L$.\footnote{In fact,
if one forces over $V$ by collapsing $X$ to be come
countable as usual with finite conditions, and $G$ is the resulting generic, 
then $X^\#$ is almost the same as
$x^\#$ where $x$ is a real equivalent to $(X,G)$,
and the theory of $x^\#$ is just the relativization of that of $0^\#$.}
In particular, $X^\#$ is uniquely determined,
the corresponding class $\mathscr{I}$ of Silver indiscernibles is uniquely determined,
$X^\#\sub(X\cross\omega)^{<\om}$ (in the codes),
$X^\#\notin\Lll$, and in fact,
$X^\#$ cannot be 
added by set-forcing to $\Lll$. Moreover,
if $N$ is a proper class transitive model of $\ZF$
and $X\cup\{X\}\sub N$ and $N\sats$``$X^\#$ exists'',
then $X^\#$ exists and is in $N$ (computed correctly there).
We write $\mathscr{I}^{X}$
for the class $\mathscr{I}$  Silver indsicernibles mentioned above,
or usually just write $\mathscr{I}^{\Lll}$, when $X$ is clear from context. (In fact, $X$ will usually be of the form $V_{\lambda+1}^{\Lll}$ where $\lambda$ is the least even ordinal such that $\Lll=L(V_{\lambda+1}^{\Lll})$. The notation $\mathscr{I}^{\Lll}$ is formally ambiguous, because, for example, if  we set $X'=X\cup\{X\}\cup\kappa$,
where $\kappa=\min(\mathscr{I}^X)$,
then $L(X\cup\{X\})=L(X'\cup\{X'\})$,
but $\mathscr{I}^X\neq\mathscr{I}^{X'}$.)\end{rem}

Now Woodin writes in Remark 26 of \cite{woodin_sem2}
(p.~161) that a ``natural 
conjecture'' 
is that if $I_{0,\lambda}$ holds, then $V_{\lambda+1}^\#$ exists
(assuming $\ZFC$, so here $\lambda$ is a limit).
Note that with the model $\Lll[j]$ from before, we will have
\[ \Lll[j]\sats\ZF+I_{0,\lambda}+\text{``}V_{\lambda+1}^\#
 \text{  does 
not exist''}\]
since otherwise $\Lll[j]\sats$``$V_{\lambda+1}^\#\in 
L(V_{\lambda+1})$'', a contradiction.
Now this does not quite disprove the conjecture just stated,
if $\ZFC$ is really the background theory.
But assuming $\ZFC$ in $V$,
then there is an easy forcing argument
to reestablish $\ZFC$ in a generic extension of $\Lll[j]$,
while preserving $V_{\lambda+1}$, hence preserving $I_{0,\lambda}$,
thus obtaining a failure of the $V_{\lambda+1}^\#$ conjecture.

The more general applicability of the canonical extension mentioned in Remark \ref{rem:standard_extension}
(applying to embeddings $j\in\mathscr{E}_1(V_{\lambda+1}^{\Lll})$ of the kind mentioned there) will be useful,
and allows us to establish the main theorem of the paper (Theorem \ref{tm:main}) in a context somewhat more general than that of literal  $I_{0,\lambda}$-embeddings, without changing the work required in any significant way. We will consider embeddings $j:\Lll\to\Lll$ of the following form:

\begin{dfn}\label{dfn:relevant}
We say $(\Lll,\lambda,j)$
is a \emph{relevant triple} iff
$\Lll$ is a transitive proper class model of ZF, $\lambda$ is an even ordinal,
$\Lll=L(V_{\lambda+1}^{\Lll})$,
$j:\Lll\to\Lll$ is elementary,
$\crit(j)<\lambda$ 
and $j\rest V_\lambda^{\Lll}\in\Lll$.
\end{dfn}

Remark \ref{rem:standard_extension} gives:
\begin{lem}
 Let $(\Lll,\lambda,j)$ be relevant. Suppose $j(\lambda)=\lambda$ \tu{(}we will show that this does indeed hold in Proposition \ref{prop:ultrapowers} below\tu{)}. Then  $j\rest V_{\lambda+1}^{\Lll}\in\Lll$ \tu{(}and therefore
 $j\rest V_{\lambda+1}^{\Lll}=((j\rest V_{\lambda})^+)^\Lll$, the canonical extension of $j\rest V_{\lambda}^{\Lll}$ as computed in $\Lll$\tu{)}.
\end{lem}

\begin{dfn}\label{dfn:proper}
 Let $(\Lll,\lambda,j)$ be relevant.
 Corresponding to the terminology of \cite{woodin_sem2},
 we say that $j$ is \emph{proper}
 iff either 
 \begin{enumerate}[label=(\roman*)]
  \item 
 $(V_{\lambda+1}^{\Lll})^\#$ does not exist,
 or 
 \item $(V_{\lambda+1}^{\Lll})^\#$ exists
 and $j\rest\mathscr{I}^{V_{\lambda+1}^{\Lll}}=\id$.
 \end{enumerate}
\end{dfn}

We will see in Proposition \ref{prop:ultrapowers}
that if $j$ is proper then it is determined by a set-sized restriction, and that ``$(\Lll,\lambda,j)$ is relevant with $j$ proper'' is first order.

We now describe the main point of the construction, for those readers familiar with the analysis of $\HOD^{L[x,G]}$. Let us assume for the 
purposes of this sketch that 
$V_{\lambda+1}^\#$ exists
where $\lambda$ is even, and write
$\Lll=L(V_{\lambda+1})$.
Let $j:\Lll\to\Lll$ witness $I_{0,\lambda}$.
By Proposition \ref{prop:ultrapowers} below, we may
and do assume that $j$ is proper
(that is, $j\rest\mathscr{I}^{\Lll}=\id$).
We will show later that $j$ is iterable
(this will be defined in \ref{dfn:iterates} below, but it is analogous
to linearly iterating $V$ with a measure).
Let $\left<\Lll_n,j_n\right>_{n\leq\om}$ be the length $(\om+1)$ 
iteration
of $(\Lll_0,j_0)=(\Lll,j)$. (Again, this will be described fully in \ref{dfn:iterates}, but it implies in particular that $\Lll_n=\Lll$ for each $n<\om$, $j_0=j$ and, roughly, ``$j_{n+1}=j_n(j_n)=j(j_n)$'',
and $\Lll_\om$ is the direct limit of $\Lll$ under these maps.) Let $j_{0\om}:\Lll_0\to\Lll_\om$ be the 
resulting direct limit map. 
Let $k_{0\om}=j_{0\om}\rest V_{\lambda+2}^{\Lll}$.
We will later see that
$\Lll[k_{0\om}]=\Lll[j]$.
Now there is a strong analogy between the models
\[ \Lll[k_{0\om}]\text{ and }M_1[\Sigma], \]
where $M_1$ is the 
minimal iterable proper class fine structural
mouse with 1 Woodin cardinal $\delta$  and $\Sigma$ an appropriate fragment of its 
iteration strategy. Moreover, under that analogy,
$V_{\lambda+2}^{\Lll[k_{0\om}]}=V_{\lambda+2}^{\Lll}$ corresponds reasonably to
$V_\delta^{M_1[\Sigma]}=V_\delta^{M_1}$.
(A key fact about $M_1[\Sigma]$
is that, indeed, $V_\delta^{M_1[\Sigma]}=V_\delta^{M_1}$ (assuming $\Sigma$ is chosen appropriately).)
Actually aside from some small extra details with \L o\'{s}'s Theorem 
without full $\AC$ (which
however are as in  \cite{woodin_sem2}), 
and the necessity of proving iterability,
things are simpler here than 
for $M_1[\Sigma]$.
 The reader familiar with Woodin's analysis of $M_1[\Sigma]$
(see \cite{HOD_as_core_model})
and possibly some of Woodin's arguments in \cite{woodin_sem2}
and/or the analysis of Silver indiscernibles in 
\cite{Theta_Woodin_in_HOD} or \cite{vm1},
may now wish to use the hint just mentioned to try to work
through the proof
independently, maybe assuming the iterability just mentioned and the existence of $V_{\lambda+1}^{\#}$.

There are, however, a couple of points where some details of the implementation
are not quite standard, as far as the author is aware.
One  is that the proof of iterability
we will give combines two techniques in a way that might be new,
although this is anyway unnecessary in the case that $\ZFC$ holds in
$V$, by Woodin's iterability results under $\ZFC$.\footnote{
See Lemma 21 of \cite{woodin_sem2};
the proof seems to make at least some use of $\DC$.
We haven't tried to adapt the methods there directly; we give a somewhat
different proof of iterability here.}
Another is that in the $M_1[\Sigma]$ context, towards the end of the 
calculation,
one forms the hull $H=\Hull^{M_\infty[\Lambda]}(\rg(i_{M_1\infty}))$ and shows
that $H\cap M_\infty=\rg(i_{M_1\infty})$, and hence that $M_1[\Sigma]$
is the transitive collapse.
In the  mouse case, it is immediate that $H\elem M_\infty[\Lambda]$.
But in our case the analogue seems to require a little more thought.

And of course, a formally new feature is that $\lambda$ 
is allowed 
be an arbitrary even ordinal (and we only assume $\ZF$).
But modulo \cite{cumulative_periodicity_pub_online_first}
(see Remark \ref{rem:standard_extension}), 
the argument in the even successor case is virtually identical
to the limit case.

\section{Background}\label{sec:background}

We begin by collecting some general facts and properties of models of the form $L(X)$,
and also more specifically of relevant triples,
which will be needed later on.

\begin{propos}\label{prop:Theta_properties}
Let $X$ be a transitive set of rank $\geq\om$.
Let $\Theta=\Theta_{X^{<\om}}^{L(X)}$.
Then:
\begin{enumerate}[label=\arabic*.,ref=\arabic*]\item\label{item:Theta_matches_P(X)} $\Theta$ is the least ordinal $\theta$
such that $\pow(X^{<\om})\cap L(X)\sub L_\theta(X)$.
\item\label{item:cofinally_many_pwd} For cofinally many $\gamma<\Theta$, we have
$L_\gamma(X)=\Hull^{L_\gamma(X)}(X\cup\{X\})$.
\item\label{item:Theta_regular} $L(X)\sats$``$\Theta$ is regular'',
and in fact, there is no $f\in L(X)$ and $\eta<\Theta$ such that $f:\eta\cross X^{<\om}\to\Theta$ and $f$ is cofinal in $\Theta$.
\end{enumerate}
\end{propos}
\begin{proof}
 Part \ref{item:Theta_matches_P(X)} is a quite routine consequence of condensation in $L(X)$.
 
 Part \ref{item:cofinally_many_pwd}: This is by a standard condensation 
argument; here it is:  By part \ref{item:Theta_matches_P(X)},
fixing $A\in\pow(X^{<\om})\cap L(X)$, it suffices to find some $\gamma$ as 
advertised
with $A\in L_\gamma(X)$.
Now  \[L(X)=\Hull_{\Sigma_1}^{L(X)}(X\cup\{X\}\cup\OR).\]
So let $\alpha<\beta\in\OR$ with $A$ being $\Sigma_1$-definable 
over $L_\beta(X)$ from parameters in $X\cup\{X,\alpha\}$. Let 
\[ H=\Hull_{\Sigma_{1}}^{L_\beta(X)}(X\cup\{X,\alpha\})\]
and $C$ be the transitive collapse. Then $C=L_\gamma(X)$
for some $\gamma<\Theta$, and $A\in C$, and
\[ C\sats\text{``There is
an ordinal }\alpha'\text{ such that 
}V=\Hull_{\Sigma_1}^V(X\cup\{X\}\cup\{\alpha'\})\text{''.}\]
But then the least such $\alpha'$ is definable over $C$ from the parameter
$X$, and therefore $C=\Hull^C(X\cup\{X\})$.

 For part \ref{item:Theta_regular}, suppose otherwise and fix such $\eta,f$. Let $g:\eta\to\Theta$ be the function $g(\alpha)=\sup\{f(\alpha,\vec{x})\bigm|\vec{x}\in X^{<\om}\}$,
 noting  that $g(\alpha)<\Theta$.
 Then $g$ is also cofinal in $\Theta$.
 Now $L(X)$ can define a set $A\sub\eta\cross\Theta$ which is not in $L_\Theta(X)$, by diagonalization. That is, given $\alpha<\eta$, let $\beta_\alpha$ be the least $\beta>g(\alpha)$ such that $L_\beta(X)=\Hull^{L_\beta(X)}(X\cup\{X\})$,
 and let $A_\alpha$ be the theory $\Th^{L_\beta(X)}(X\cup\{X\})$, coded naturally as a subset of $X^{<\om}$. Now  let $A=\{(\alpha,\vec{x})\bigm|\alpha<\eta\wedge \vec{x}\in A_\alpha\}$.
 Since $\eta<\Theta$, we can now find a set $A'\in L(X)$ with $A'\sub X^{<\om}$ and $A'$ coding $A$, and hence with $A'\in L(X)\cut L_\Theta(X)$, a contradiction.
\end{proof}

\begin{propos}\label{prop:Gamma_generates} Let $X$ be  a transitive set.
Let $\Gamma$ be a proper class of ordinals with $\mathscr{I}^{X}\sub\Gamma$ if $X^{\#}$ exists. Then $L(X)=\Hull^{L(X)}(X\cup\{X\}\cup\Gamma)$.
\end{propos}
\begin{proof} 
Suppose not.
Then $X^{\#}$ does not exist. Let $H=\Hull^{L(X)}(X\cup\{X\}\cup\Gamma)$.
Then $H\preccurlyeq\Lll$,
since $L(X)\sats$``$V=\OD_{X\cup\{X\}}$''.
Note that $L(X)$ is the transitive collapse of $H$,
and letting $k:L(X)\to L(X)$ be the uncollapse map,
then $\rank(X)<\crit(k)$, as $X$ is transitive
and the rank function is definable over $L(X)$. But then $X^\#$ exists,
contradiction.
\end{proof}

We now prove a fact reminiscent of the Zipper Lemma (Theorem 6.10 of \cite{outline}):

\begin{propos}\label{prop:Zipper_analogue}
Let $X$ be a  non-empty transitive set. Let
$\Theta=\Theta^{L(X)}_{X^{<\om}}$.
Let $j_i\in\mathscr{E}_1(L_\Theta(X))$, for $i=0,1$, with $j_i(X)=X$
 and $j_0\rest X=j_1\rest X$, 
 but $j_0\neq j_1$.  Then:
\begin{enumerate}[label=\arabic*.,ref=\arabic*]
 \item\label{item:range_intersection_bounded} 
 $\rg(j_0)\cap\rg(j_1)$ is bounded in $L_\theta(X)$; in fact,
letting $\alpha<\Theta$ be least such that 
 $j_0(\alpha)\neq j_1(\alpha)$, we have 
 $(j_0``\Theta)\cap(j_1``\Theta)=j_0``\alpha=j_1``\alpha$.
 \item\label{item:cof(Theta)=om} If $j_0\rest X\in L(X)$ then 
$\cof(\Theta)=\om$ \tu{(}in $V$\tu{)}.
\end{enumerate}
\end{propos}
\begin{proof}
Part \ref{item:range_intersection_bounded}: We may assume 
$j_0(\alpha)<j_1(\alpha)$. Therefore
$\sup j_1``\alpha=\sup j_0``\alpha\leq j_0(\alpha)<j_1(\alpha)$,
so $j_0(\alpha)\in\rg(j_0)\cut\rg(j_1)$.

Suppose $j_0(\alpha)<\beta<\theta$ and 
$\beta\in\rg(j_0)\cap\rg(j_1)$. Then the least $\xi>\beta$
with
\[ L_\xi(X)=\Hull^{L_\xi(X)}(X\cup\{X\}) \]
is also in $\rg(j_0)\cap\rg(j_1)$. But then note that
\[ \rg(j_0)\cap L_\xi(X)=
\Hull^{L_\xi(X)}((j_i``X)\cup\{X\})=
 \rg(j_1)\cap L_\xi(X),
\]
for $i=0,1$, and therefore $j_0(\alpha)\in\rg(j_1)$, contradiction.

Part \ref{item:cof(Theta)=om}: Suppose $j_i\rest X\in L(X)$. Let 
$\alpha_0=\alpha$, and given $\alpha_n$,
let $\alpha_{n+1}=\max(j_0(\alpha_n),j_1(\alpha_n))$.
Notice by induction that $\alpha_{n+1}>\alpha_n$ for each $n<\om$.
Let $\gamma=\sup_{n<\om}\alpha_n$; we claim that $\gamma=\Theta$.
For if $\gamma<\Theta$ then  because $j_i\rest X\in 
L(X)$, and by Proposition 
\ref{prop:Theta_properties} part \ref{item:cofinally_many_pwd}
and arguing as in the proof of part \ref{item:range_intersection_bounded} of the present lemma,
we get $j_0\rest\gamma\in L(X)$ and 
$j_1\rest\gamma\in L(X)$.
But  $\left<\alpha_n\right>_{n<\om}$ can be computed
from $j_0\rest\gamma$ and $j_1\rest\gamma$, so it is in $L(X)$.
But then $j_0,j_1$ are both continuous at $\gamma$,
so  $j_0(\gamma)=\gamma=j_1(\gamma)$, contradicting
part \ref{item:range_intersection_bounded}.
\end{proof}

We will now state some definitions and facts about  extensions of given embeddings
to $I_{0,\lambda}$-embeddings $j$, in particular on existence and uniqueness of such extensions, and in fact some generalizations thereof in which $j$ will be the embedding of a relevant triple $(\Lll,\lambda,j)$. These things are basically
as in \cite{woodin_sem2} (although recall we only assume $\ZF$);
 also see \cite{cumulative_periodicity_pub_online_first}
for more on the measure defined below:

\begin{dfn}
Let $(\Lll,\lambda,j)$ be relevant.
 Suppose $j(\lambda)=\lambda$ (cf.~Proposition \ref{prop:ultrapowers}).
 Then $\mu_j$ denotes the $\Lll$-ultrafilter on $V_{\lambda+1}^{\Lll}$ derived from $j$ with seed 
$j\rest V_\lambda^{\Lll}$:
\[ \mu_j=\big\{X\in V_{\lambda+2}^{\Lll}\bigm| j\rest V_\lambda^{\Lll}\in 
j(X)\big\}.\]
Recall from \S\ref{subsec:notation} that the ultrapower $\Ult_0(\Lll,\mu_j)$
is formed using 
only functions 
$f\in\Lll$ (such that $f:V_{\lambda+1}^{\Lll}\to\Lll$).
We write $[f]^{\Lll}_{\mu_j}$, or just $[f]$, for the element of the ultrapower
represented by $f$. Given such functions $f,g$, we write $f=_{\mu_j}g$ iff
\[ \big\{k\in V_{\lambda+1}^{\Lll}\bigm|f(k)=g(k)\big\}\in\mu_j,\]
and similarly for $f\in_{\mu_j}g$.

 We say that $\mu_j$ is \emph{weakly amenable} to $\Lll$ iff
 for each $\pi:V_{\lambda+1}^{\Lll}\to V_{\lambda+2}^{\Lll}$
 with $\pi\in\Lll$, we have $\rg(\pi)\cap \mu_j\in\Lll$.
\end{dfn}

The following lemma is just a combination of standard calculations with a shift of calculations of  Woodin from \cite{woodin_sem2} over to the $\ZF$ + ``$\lambda$ even'' context
(using \cite{cumulative_periodicity_pub_online_first})
and allowing arbitrary relevant models:
\begin{propos}\label{prop:ultrapowers}
Let $(\Lll,\lambda,j)$ be relevant. Then:
 \begin{enumerate}[label=\arabic*.,ref=\arabic*]
  \item\label{item:j(lambda)=lambda}$j(\lambda)=\lambda$ is the least even ordinal $\lambda'$ such that $\Lll=L(V_{\lambda'+1}^{\Lll})$.\end{enumerate}
 Let $U=\Ult_0(\Lll,\mu_j)$ and $j_0:\Lll\to U$ be the ultrapower map.
 Then:
 \begin{enumerate}[label=\arabic*.,ref=\arabic*]
   \setcounter{enumi}{1}
  \item\label{item:mu_j_weakly_amenable}$\mu_j$ is weakly amenable to $\Lll$,
  \item $U$ is extensional and isomorphic to $\Lll$;
  we take $U=\Lll$,
  \item $j_0:\Lll\to U=\Lll$ is elementary,
  \item $j_0\rest L_\Theta(V_{\lambda+1}^{\Lll
})=j\rest L_\Theta(V_{\lambda+1}^{\Lll
})$
  where $\Theta=\Theta_{V^{\Lll
}_{\lambda+1}}^{\Lll}$ \tu{(}see \S\ref{subsec:notation}\tu{)},
  \item if $(V_{\lambda+1}^{\Lll
})^\#$ does not exist then $j_0=j$ is proper
  and $[f]^{\Lll}_{\mu_j}=j(f)(j\rest V_\lambda^{\Lll
})$.
 \end{enumerate}
 Now suppose further that $(V_{\lambda+1}^{\Lll
})^\#$ exists.
 Then:
 \begin{enumerate}[label=\arabic*.,ref=\arabic*]
  \setcounter{enumi}{6}
  \item $j_0\rest\mathscr{I}^{\Lll}=\id$, so $j_0$ is proper,
 \item  $j``\mathscr{I}^{\Lll}\sub\mathscr{I}^{\Lll}$,
 \end{enumerate}
 and letting $j_1:\Lll\to\Lll$ be the unique elementary map
with $\crit(j_1)\geq\min(\mathscr{I}^{\Lll})$
 and $j_1\rest\mathscr{I}^{\Lll}\sub j$, we have
 \begin{enumerate}[label=\arabic*.,ref=\arabic*]
   \setcounter{enumi}{8}
  \item\label{item:j_1_com_j_0=j} $j_1\com j_0=j$ and $j_1([f]^{\Lll}_{\mu_j})=j(f)(j\rest V_\lambda^{\Lll
})$.
 \end{enumerate}
 Moreover \tu{(}still assuming $(V_{\lambda+1}^{\Lll
})^\#$ exists\tu{)}, this is the unique pair 
$(j'_0,j'_1)$
 such that
 $j'_i:\Lll\to\Lll$ is elementary for $i=0,1$,
 $j=j'_1\com j'_0$,
$j'_0$ is proper and $\crit(j'_1)\geq\lambda$.
\end{propos}

\begin{proof}We first prove  everything assuming
part \ref{item:j(lambda)=lambda}. So  $j(\lambda)=\lambda$.

First consider weak amenability.
Let $\pi\in\Lll$ with $\pi:V_{\lambda+1}^{\Lll
}\to V_{\lambda+2}^{\Lll}$.
Then $j(\pi):V_{\lambda+1}^{\Lll
}\to V_{\lambda+2}^{\Lll}$,
and since $j\rest V_\lambda^{\Lll
}\in V_{\lambda+1}^{\Lll
}$,
by Remark \ref{rem:standard_extension}, we have $k=j\rest V_{\lambda+1}^{\Lll
}\in 
\Lll
$. But then note that for $x\in V_{\lambda+1}^{\Lll
}$,
we have
\[ \pi(x)\in\mu_j\iff j\rest V_\lambda^{\Lll
}\in 
j(\pi)(j(x))\iff j\rest V_\lambda^{\Lll
}\in j(\pi)(k(x)), \]
and since $j(\pi),k\in \Lll
$, therefore $\rg(\pi)\cap\mu_j\in\Lll$, as desired.

Now write
$\mathscr{F}=\{f\in\Lll\bigm|f:V_{\lambda+1}^{\Lll
}\to\Lll\}$,
the class of functions used to form $U$.
Let
\[ U'=\{j(f)(j\rest V_\lambda^{\Lll
})\bigm|f\in\mathscr{F}\} \]
and
\[ H=\Hull^{\Lll}(\rg(j)\cup\{j\rest V_\lambda^{\Lll
}\}). \]
Then $U\iso U'=H$.
For $U\iso U'$ directly from the definition. Clearly
$U'\sub H$.
Conversely, let $x\in H$, and let $a\in\mathscr{L}$ and $\varphi$ be a formula such that
\[ x=\text{ the unique }x'\in\Lll\text{ such that 
}\Lll\sats\varphi(x',j(a),j\rest V_\lambda^{\Lll
}). \]
Let $f:V^{\Lll
}_{\lambda+1}\to\Lll$ be the function
\[ f(k)=\text{ the unique }x'\text{ such that }\Lll\sats\varphi(x',a,k), \]
if a unique such $x'$ exists, and $f(k)=\emptyset$ otherwise.
Then
$j(f)(j\rest V_\lambda^{\Lll
}) = x\in U'$, as desired.

Now $V_{\lambda+1}^{\Lll
}\sub H$, because given any $y\in V_{\lambda+1}^{\Lll
}$,
we have
\[ y = (j\rest V_\lambda^{\Lll
})^{-1}``j(y).\]
Since $\Lll\sats$``$V=\HOD(V_{\lambda+1})$'',
it follows that $H\elem\Lll$.

Now since $U\iso U'=H\elem\Lll$, $U$ is extensional and wellfounded,
so we identify it with its transitive collapse.
Let $j_0:\Lll\to U$ be the ultrapower map,
and $j_1:U\to\Lll$ be the natural factor map
\[ j_1([f])=j(f)(j\rest V_\lambda^{\Lll
}).\]
Note then that $U$ is just the transitive collapse
of $H$, which is just $\Lll$, and $j_1$ is the uncollapse map.
Since $H\elem\Lll$, $j_1$ is elementary
and if $j_1\neq\id$ then $\crit(j_1)>\lambda$,
in which case $(V^{\Lll
}_{\lambda+1})^\#$ exists
and $\crit(j_1)\geq\min(\mathscr{I})$
and $j_1$ is determined by its action on $\mathscr{I}$,
where $\mathscr{I}=\mathscr{I}^{\Lll}$.
Note that $j=j_1\com j_0$.
So $j_0=j_1^{-1}\com j$ is elementary
with $j_0\rest V_\lambda^{\Lll
}=j\rest V_\lambda^{\Lll
}$.

Now suppose $j_1\neq\id$, so $(V_{\lambda+1}^{\Lll
})^\#$ exists.
Then we claim  $j_0\rest\mathscr{I}=\id$.
This is a standard calculation with Silver indiscernibles:
We have
\[ \Lll=\Hull^{\Lll}(\mathscr{I}\cup V_{\lambda+1}^{\Lll
}),\]
\begin{equation}\label{eqn:U_hull_of_j_0``I_cup_V_lambda+1}\Lll=U=\Hull^{\Lll}(\rg(j_0)\cup\{j\rest V_\lambda^{\Lll
}\})=\Hull^{\Lll}((j_0``\mathscr{I})\cup V_{\lambda+1}^{\Lll
}).\end{equation}
But  $j_0$ is continuous at all points in $\mathscr{I}$ (as the functions used to form the ultrapower are in $\mathscr{L}$), from which it follows that $j_0``\mathscr{I}$ is a club class of model-theoretic indiscernibles relative to parameters in $V_{\lambda+1}^{\Lll
}$,
and therefore that
\begin{equation}\label{eqn:j_0``I_sub_I} j_0``\mathscr{I}\sub\mathscr{I}.\end{equation} But supposing $j_0``\mathscr{I}\psub\mathscr{I}$, let $\gamma$ be least in $\mathscr{I}\cut j_0``\mathscr{I}$. Then by lines (\ref{eqn:U_hull_of_j_0``I_cup_V_lambda+1}) and (\ref{eqn:j_0``I_sub_I}),
we can fix $\vec{\varepsilon}\in\mathscr{I}\cut\{\gamma\}$ and $x\in V_{\lambda+1}^{\Lll
}$ such that $\gamma$ is definable over $\Lll$ from $(\vec{\varepsilon},x)$,
contradicting indiscernibility.

The rest of parts \ref{item:mu_j_weakly_amenable}--\ref{item:j_1_com_j_0=j} is clear, under our assumption that part \ref{item:j(lambda)=lambda} holds.
So now let us verify part \ref{item:j(lambda)=lambda}. Let $\widetilde{\lambda}$ be the least even ordinal such that $\Lll=L(V_{\widetilde{\lambda}+1}^{\Lll})$. Clearly
$j(\widetilde{\lambda})=\widetilde{\lambda}$, so we may assume $\widetilde{\lambda}<\lambda$.
Now $\crit(j)<\widetilde{\lambda}$,
because otherwise,
$\widetilde{\lambda}<\crit(j)<\lambda$, but
since $j\rest V_\lambda^{\mathscr{L}}\in\mathscr{L}$,
this implies that
$\Lll\sats$``$V_{\widetilde{\lambda}+1}^\#$ exists'',
contradicting the fact that $\Lll=L(V_{\widetilde{\lambda}+1}^{\Lll})$.
So note that the hypotheses of the proposition  hold with $\widetilde{\lambda}$ replacing $\lambda$. So by what we established above, the conclusion of the proposition holds with respect to the relevant triple $(\Lll,\widetilde{\lambda},j)$.

Now since $j(\widetilde{\lambda})=\widetilde{\lambda}$ but $j(\lambda)>\lambda$, we have $\widetilde{\lambda}+\om<\lambda$. Therefore $j\rest V^{\Lll
}_{\widetilde{\lambda}+2}\in \Lll$.
Defining the derived ultrafilter $\widetilde{\mu}_j$
as before, but from seed $j\rest V_{\widetilde{\lambda}}^\Lll$, then $\widetilde{\mu}_j$
is defined using only
$j\rest V^{\Lll}_{\widetilde{\lambda}+2}$,
so $\widetilde{\mu}_j\in\Lll$. But it follows from the preceding parts that $\Lll\sats$ ``There is a ($\Sigma_1$-)elementary $k:V\to V$ which is definable from the parameter $\widetilde{\mu}_j$'', which contradicts \cite{suzuki_no_def_j}.
\end{proof}

\begin{dfn}\label{dfn:standard_decomp}
Under the assumptions of the previous lemma,
the factoring $j=j_1\com j_0$ is the \emph{standard decomposition}
of $j$, where if $(V_{\lambda+1}^{\Lll
})^\#$ does not exist then $j=j_0$ and $j_1=\id$.
 And defining $\lambda^{\Lll}$
 to be the least even ordinal $\lambda'$
 such that $\Lll=L(V_{\lambda'+1}^{\Lll})$, we have $\lambda^{\Lll}=\lambda$
 by part \ref{item:j(lambda)=lambda}.
\end{dfn} 

\begin{dfn}\label{dfn:s^-} Let $\Lll$ be a transitive proper class model of ZF and $\lambda$ be even and suppose that $\lambda$ is the least even ordinal $\lambda'$ such that $\Lll=L(V_{\lambda'+1}^{\Lll})$. For $\alpha\in\OR\cut(\lambda+1)$,
letting $\alpha=\lambda+\beta$,
we write $\Lll|\alpha=L_\beta(V_{\lambda+1}^{\Lll})$,
so $\OR(\Lll|\alpha)=\alpha$.
Let $\vec{\OR}=[\OR\cut(\lambda+1)]^{<\om}\cut\{\emptyset\}$. For $s\in\vec{\OR}$
we write $s^-=s\cut\{\max(s)\}$,
and  define the theory \[\Th^{\Lll}_s=\Th^{\Lll|\max(s)}(V_{\lambda+1}^{\Lll}\cup\{s^-\}). \]

Now  suppose  also that $(V_{\lambda+1}^{\Lll})^\#$ exists. 
Let $0<n<\om$ and $s\in[\mathscr{I}^{\Lll}]^{<\om}$ with $\card(s)=n$.
Then let $t_n$
denote the theory $\Th_{s}^{\Lll}$ coded naturally as a subset of $V_{\lambda+1}^{\Lll}$ (parameters in $s^-$ are coded via
some constant symbols in $V_\om$).
Define the ordinal
  \[ \bar{\iota}_n=\OR(\cHull^{\Lll|\max(s)}(V_{\lambda+1}^{\Lll}\cup\{s^-\})) \]
  (see \S\ref{subsec:notation} for notation). (Of course some of this notation is ambiguous as it depends on $\Lll$, but $\Lll$ will be clear from context.)
\end{dfn}

\begin{rem}\label{rem:sharp_and_cof(Theta)}
 Let $\lambda$ be even and least such that $\Lll=L(V_{\lambda+1}^{\Lll})$.
 Let $\Theta=\Theta_{V_{\lambda+1}^{\Lll}}^{\Lll}$. Suppose $(V_{\lambda+1}^{\Lll})^{\#}$ exists. Then $\cof(\Theta)=\om$,
 and in fact, with notation as in Definition \ref{dfn:s^-},
 we have $\bar{\iota}_n<\bar{\iota}_{n+1}$ for all $n\in(0,\om)$,
 and $\Theta=\sup_{n\in(0,\om)}\bar{\iota}_n$.
\end{rem}

\begin{propos}\label{prop:extensions}
 Let $\lambda$ be even and suppose $\lambda$ is the least even ordinal  such that $\Lll=L(V_{\lambda+1}^{\Lll})$.  Let $\Theta=\Theta_{V_{\lambda+1}^{\Lll}}^{\Lll}$.
Then:
 \begin{enumerate}[label=\arabic*.,ref=\arabic*]
  \item\label{item:unique_ext} Let
$\ell:V_\lambda^{\Lll}\to V_\lambda^{\Lll}$ be elementary
with $\ell\in\Lll$. Then $\ell$
extends to at most one proper $j:\Lll\to\Lll$.
\item\label{item:k_extends_to_k'} Let $k:V_{\lambda+2}^{\Lll}\to V_{\lambda+2}^{\Lll}$ be $\Sigma_1$-elementary with $k\rest V_\lambda^{\Lll}\in\Lll$. Then:
\begin{itemize}
\item[--]  $k$ is fully elementary,
\item[--] there is a unique $\Sigma_1$-elementary $k':L_\Theta(V_{\lambda+1}^{\Lll})\to L_\Theta(V_{\lambda+1}^{\Lll})$ with $k\sub k'$.
\end{itemize}
\item\label{item:k'_comes_from_k} Let $k':L_\Theta(V_{\lambda+1}^{\Lll})\to L_\Theta(V_{\lambda+1}^{\Lll})$ be $\Sigma_1$-elementary with $k'\rest V_\lambda^{\Lll}\in\Lll$. Then:
\begin{itemize}
 \item[--] $k'$ is fully elementary,
 \item[--] $k'\rest V_{\lambda+2}^{\Lll}:V_{\lambda+2}^{\Lll}\to V_{\lambda+2}^{\Lll}$ is  elementary.
\end{itemize}
 \item  Suppose $(V_{\lambda+1}^{\Lll})^\#$ exists.
  \label{item:sharps_emb_ext_charac}
  Let $k:L_\Theta(V_{\lambda+1}^{\Lll})\to L_\Theta(V_{\lambda+1}^{\Lll})$ be 
$\Sigma_1$-elementary with $k\rest V_\lambda^{\Lll}\in\Lll$.
 Then the following are equivalent:
 \begin{itemize}\item[--] $k$ extends to an elementary $j:\Lll\to\Lll$,
 \item[--] $k$ extends to a proper elementary $j:\Lll\to\Lll$,
 \item[--] $k(t_n)=t_n$ for each $n\in(0,\om)$,
 \item[--] $k(\bar{\iota}_n)=\bar{\iota}_n$ for 
each $n\in(0,\om)$.
\end{itemize}
 \end{enumerate}
\end{propos}
\begin{proof}
Part \ref{item:unique_ext}: First, the canonical
extension $\ell^+$ is the unique possible elementary
extension of $\ell$ to a map $V_{\lambda+1}^{\Lll}\to V_{\lambda+1}^{\Lll}$, by
\ref{rem:standard_extension}. So if
$(V_{\lambda+1}^{\Lll})^\#$ exists
then since $\Lll=\Hull^{\Lll}(\mathscr{I}^{\Lll}\cup V_{\lambda+1}^{\Lll})$,
the uniqueness of a proper extension of $\ell$ is clear.
Suppose $(V_{\lambda+1}^{\Lll})^\#$ does not exist,
and $\ell_0,\ell_1$ are two proper extensions of $\ell$,
hence extending $\ell^+$.
Then by Proposition \ref{prop:ultrapowers},
$\ell_0,\ell_1$ are the ultrapower maps with respect to 
$\mu_{\ell_0},\mu_{\ell_1}$. Therefore we can find a proper class $\Gamma$
of common fixed points for $\ell_0,\ell_1$.
Since
$\ell_0\rest V_{\lambda+1}^{\Lll}=\ell^+=\ell_1\rest V_{\lambda+1}^{\Lll}$,
therefore $\ell_0\rest H=\ell_1\rest H$ where
$H=\Hull^{\Lll}(\Gamma\cup V_{\lambda+1}^{\Lll})$.
But by Proposition  \ref{prop:Gamma_generates},
$\Lll=H$, which suffices.

Part \ref{item:k_extends_to_k'}: 
Let
$k\in\mathscr{E}_1(V_{\lambda+2}^\Lll)$ with $k\rest V_\lambda^{\Lll}\in\Lll$; hence $k\rest V_{\lambda+1}^{\Lll}\in\Lll$. By Proposition \ref{prop:Theta_properties}, $V_{\lambda+2}^{\Lll}\sub L_\Theta(V_{\lambda+1}^{\Lll})$. Now to see that $k$ extends uniquely to 
a $k'\in\mathscr{E}_1(L_\Theta(V_{\lambda+1}))$,
it is easily enough to see that if ${\leq^*}\in V_{\lambda+2}^{\Lll}$ is a prewellorder
on $V_{\lambda+1}^{\Lll}$ then $k({\leq^*})$ is wellfounded. But since $k({\leq^*})\in V_{\lambda+2}^{\Lll}$,
 note that the wellfounded and illfounded parts of $k({\leq^*})$ are in $\Lll$,
 and hence in $V_{\lambda+2}^{\Lll}$.
 But the existence of a non-empty illfounded part is just a $\Sigma_1^{V_{\lambda+2}^{\Lll}}$
 assertion about $k({\leq^*})$.
So ${\leq^*}$ is illfounded, contradiction.

Now it follows that $k'$ is in fact fully elementary, because
by Proposition \ref{prop:Theta_properties},
$L_\Theta(V_{\lambda+1}^{\Lll})\sats$ Collection.
Since  $V_{\lambda+2}^{\Lll}$ is a definable subclass of $L_\Theta(V_{\lambda+1}^{\Lll})$, it  follows that $k$ is  fully elementary.

Part \ref{item:k'_comes_from_k}: By some of the preceding considerations.

Part \ref{item:sharps_emb_ext_charac}: This is basically routine.
The fact that if $k$ extends to an elementary $j:\Lll\to\Lll$ then $k$ also extends to a proper one, is by Proposition \ref{prop:ultrapowers}.
The fact that if $k$ extends to a proper elementary $j$ then 
$k(\bar{\iota}_n)=\bar{\iota}_n$, is because
$\bar{\iota}_n$ is definable over $\Lll$ from $s$ and $\lambda$, where 
$s\in[\mathscr{I}^{\Lll}]^{n}$.
\end{proof}

We now define iterability, following \cite{woodin_sem2}:

\begin{dfn}\label{dfn:iterates}
 Let $(\Lll,\lambda,j)$ be relevant
 with $j$ proper. We (attempt to) define 
the $\alpha$th 
iterate
 $(\Lll_\alpha,\lambda_\alpha,j_\alpha,\mu_\alpha)$ of 
$(\Lll,\lambda,j,\mu_j)$, with $(\Lll_\alpha,\lambda_\alpha,j_\alpha)$ relevant and $j_\alpha$ proper. We set 
$\Lll_0=\Lll$ and $\lambda_0=\lambda$ and $j_0=j$ and $\mu_0=\mu_j$. Given 
$(\Lll_\alpha,\lambda_\alpha,j_\alpha,\mu_\alpha)$
  such that $(\Lll_\alpha,\lambda_\alpha,j_\alpha)$ is relevant,
 $j_\alpha$ proper  and $\mu_\alpha=\mu_{j_\alpha}$
 (so \ref{prop:ultrapowers} applies, and in particular $\mu_\alpha$ is weakly amenable to $\Lll_\alpha$), set 
$\Lll_{\alpha+1}=\Lll_\alpha$ and
$\lambda_{\alpha+1}=\lambda_\alpha$ and letting
\[ \mu_\alpha^-=\{\rg(\pi)\cap \mu_{\alpha}\bigm|\pi\in \Lll_\alpha\text{ and 
}\pi:V_{\lambda_\alpha+1}^{\Lll_\alpha}\to V_{\lambda_\alpha+2}^{\Lll_\alpha}\}, \]
set
\[ \mu_{\alpha+1}=\bigcup_{X\in \mu_\alpha^-}j_\alpha(X). \]
 By \cite{cumulative_periodicity_pub_online_first},
  $\mu_{\alpha+1}$ is an $\Lll_{\alpha+1}$-ultrafilter.
If $\Ult_0(\Lll_{\alpha+1},\mu_{\alpha+1})$ satisfies \L o\'{s}'s Theorem,
is wellfounded and $=\Lll_{\alpha+1}$, then we set
$j_{\alpha+1}$ to be the ultrapower map, and otherwise stop the construction.
At limit stages $\eta$ we take direct limits to define 
$\Lll_\eta,\lambda_\eta,\mu_\eta$, and then attempt to define $j_\eta$ as before. If we reach some $(\Lll_\gamma,\lambda_\gamma,j_\gamma,\mu_\gamma)$ but $(\Lll_\gamma,\lambda_\gamma,j_\gamma)$ is not relevant, $j_\gamma$ not proper, or $\mu_\gamma\neq \mu_{j_{\gamma}}$, then we stop the construction there.

We say that $(\Lll,j)$ is \emph{$\eta$-iterable} iff $\Lll_\alpha$ is 
defined (and 
wellfounded) for all $\alpha<\eta$. We say that $(\Lll,j)$ is \emph{iterable}
iff $\eta$-iterable for all $\eta\in\OR$.\end{dfn}

\begin{dfn}\label{dfn:external_iterates}
Let $(\Lll,\lambda,j)$ be relevant with $j$ proper.
If $N\sats\ZF$ is a transitive class with $\Lll=L(V_{\lambda+1})^N$
and $V_{\lambda+2}^{\Lll}=V_{\lambda+2}^N$ and $\mu_j\in N$,
then the iterates of $N$ with respect to $\mu_j$
are just defined internally to $N$ (just like iterating $V$
with respect to a measure under $\ZFC$).
In this case, we say that $N$ is a \emph{tight extension}
of $\Lll$ if for every function $f\in N$ with $f:V_{\lambda+1}^N\to\Lll$
there is $g\in\Lll$ such that $f=_{\mu_j}g$.
\end{dfn}

\begin{lem}\label{lem:tight_extension}
 Adopt notation as above, and suppose that $N$ is a tight extension of $\Lll$.
 Suppose that the ultrapower $\Ult_0(N,\mu_j)$ satisfies \L o\'{s}'s Theorem.
Then $\mu_j$ is iterable, both in the sense of \ref{dfn:iterates} and \ref{dfn:external_iterates}.
Moreover, let
$\left<\Lll_\alpha,\lambda_\alpha,\mu_\alpha,j_\alpha\right>_{\alpha\in\OR}$
be the iterates of $(\Lll,\mu_j)$ and $j_{\alpha\beta}:\Lll_\alpha\to\Lll_\beta$
the iteration maps, and
$\left<N_\alpha,\mu^+_\alpha\right>_{\alpha\in\OR}$
be the \tu{(}internal\tu{)} iterates of $(N,\mu_j)$ and $j^+_{\alpha\beta}$ the 
iteration 
maps.
Then $\Lll_\alpha=j^+_{0\alpha}(\Lll)$ and $\mu_\alpha=\mu^+_\alpha$
and $j_{\alpha\beta}\sub j^+_{\alpha\beta}$,
and  the iteration maps are elementary.
\end{lem}
\begin{proof}
 We already know that $\Ult_0(\Lll,\mu_j)$ is wellfounded
 and isomorphic to $\Lll$, by \ref{prop:ultrapowers}.
By \L o\'{s}'s Theorem, the ultrapower $\Ult_0(N,\mu_j)$ is extensional and the 
ultrapower map $i^{N}_{\mu_j}$ is elementary. 
 Because $N$ is a tight extension, note that $\Ult_0(N,\mu_j)$
 is wellfounded and
\[ 
L(V_{\lambda+1})^{\Ult_0(N,\mu_j)}=i^N_{\mu_j}(\Lll)=\Ult_0(\Lll,\mu_j)=\Lll\] 
and $i^N_{\mu_j}\rest\Lll=i^{\Lll}_{\mu_j}$.
It is easy enough to see that also $\mu_1=\mu^+_1$.
By this and  elementarity all of the assumed facts hold also for 
$\Lll_1,j_1,\mu_1=\mu^+_1,N_1$, and note that in this way we also
get \L o\'{s}'s Theorem and extensionality for $\Ult_0(\Lll_1,\mu_1)$.
Therefore we can iterate the process,
and the same facts and relationships hold for the iterates, assuming 
wellfoundedness.
But then the usual proof of linear iterability shows that all iterates of 
$(N,\mu_j)$ are wellfounded, and hence so are the $\Lll_\alpha$.
\end{proof}

\section{The main argument}\label{sec:main}

In this section we will prove the main result,
Theorem \ref{tm:main}.
We first establish the following lemma;
the methods involved in its proof will be extended
to prove \ref{tm:main}:

\begin{lem}\label{lem:omega^2-it}
 Let $(\Lll,\lambda,j)$ be a relevant triple
 with $j$ proper. Then $(\Lll,j)$ is $\om^2$-iterable \tu{(}in the sense of Definition \ref{dfn:iterates}\tu{)}.
 \end{lem}

\begin{proof}
The proof will heavily rely on some kinds of calculations
in \cite{woodin_sem2}. Fix a relevant triple $(\Lll,\lambda,j)$
with $j$ proper. Fix a proper class 
$\Gamma\sub\OR\cut(\lambda+1)$
such that $j\rest\Gamma=\id$, taking $\Gamma=\mathscr{I}^{\Lll}$ if 
$(V_{\lambda+1}^{\Lll})^\#$ exists. Let $\vec{\Gamma}=[\Gamma]^{<\om}\cut\{\emptyset\}$.
For $s\in\vec{\Gamma}$ let $s^-=s\cut\{\max(s)\}$, and let
\[ H^s=\Hull^{\Lll|\max(s)}(V_{\lambda+1}^{\Lll}\cup\{s^-\}).\]
Note this is the uncollapsed hull and $H^s\elem\Lll|\max(s)$,
because
\[ \Lll|\max(s)=\Hull^{\Lll|\max(s)}(V_{\lambda+1}^{\Lll}\cup\max(s)).\]
Given any  $k\in\mathscr{E}(\Lll)$ with $k\rest V_\lambda^{\Lll}\in\Lll$, $\crit(k)<\lambda$ and $k\rest\Gamma=\id$,
let
\[ k^s=k\rest H^s:H^s\to H^s. \]
Note that $k^s$ is the unique $\pi\in\mathscr{E}(H^s)$
with $k\rest V_{\lambda}^{\Lll}\sub\pi$ and $\pi(s^-)=s^-$; moreover, $k^s\in\Lll$.
For letting $\pi$ be such, we have $\pi(\lambda)=\lambda$,
since otherwise $\lambda$ is a limit
and $H^s\sats$``there is no $\xi<\pi(\lambda)$
such that $V=L(V_{\xi+1})$'',
contradicting the fact that $V_{\lambda+1}^{\Lll}\sub\pi(V_\lambda^{\Lll})$.
 So in fact, $\pi\rest V_{\lambda+1}^{\Lll}=k\rest V_{\lambda+1}^{\Lll}$ is the canonical
extension of $k\rest V_{\lambda}^{\Lll}$.
But for each $x\in H^s$ there is a term $t$ and $z\in V_{\lambda+1}^{\Lll
}$ such that $H^s\sats\text{``}x=t(z,s^-)\text{''}$,
so $H^s\sats\text{``}\pi(x)=t(k(z),s^-)\text{''}$. This shows $\pi=k^s\in\Lll$.

\begin{clm}\label{clm:mu_j_rep}
 Let $\mu_0=\mu_j$ and $\mu_1$ be as in Definition \ref{dfn:iterates},
 and $j_{01}=j=i^{\Lll}_{\mu_0}$ and $j_{12}=i^{\Lll}_{\mu_1}$
 \tu{(}the ultrapower maps; we haven't shown yet 
that $\Ult(\Lll,\mu_1)$ is wellfounded etc\tu{)}.\footnote{
Notationally $j_{mn}:\Lll_m\to\Lll_n$.}
 Then:
 \begin{enumerate}[label=\arabic*.,ref=\arabic*]
  \item\label{item:j_01_rest_alpha_in} $j_{01}\rest(\Lll|\alpha)\in\Lll$ for 
each $\alpha<\Theta=\Theta_{V_{\lambda+1}^{\Lll}}^{\Lll}$,
 \end{enumerate}
 and letting $k=\bigcup_{\alpha<\Theta}j_{01}(j_{01}\rest(\Lll|\alpha))$, we 
have
 \begin{enumerate}[label=\arabic*.,ref=\arabic*]
 \setcounter{enumi}{1}
  \item$k:\Lll|\Theta\to\Lll|\Theta$ is elementary,\label{item:k_elem} 
  \item\label{item:mu_1_derived_from_k} $\mu_1$ is the measure derived from $k$
  with seed $k\rest V_\lambda^{\Lll}$,
  \item\label{item:Ult_0(Lll_Theta,mu_1)} 
$\Ult_0(\Lll|\Theta,\mu_1)=\Lll|\Theta$  and $k$ is 
the ultrapower map \tu{(}so $k\rest 
V_\lambda^{\Lll}=i^{\Lll}_{\mu_1}\rest V_\lambda^{\Lll}$\tu{)},
\item\label{item:j_12(Lll_Theta)} $\Lll|\Theta=j_{12}(\Lll|\Theta)$ and 
$k=j_{12}\rest(\Lll|\Theta)$.
 \end{enumerate}
Let $\Gamma$, etc, be as above for $j=j_{01}$.
Then:
   \begin{enumerate}[label=\arabic*.,ref=\arabic*]
   \setcounter{enumi}{5}
  \item  
   \L o\'{s}'s Theorem holds for $\Ult_0(\Lll,\mu_1)$,
and the ultrapower is extensional 
and isomorphic to $\Lll$,
\label{item:Ult_mu_1=Lll}
\item\label{item:ult_map_is_pw_image} for each $s\in\vec{\Gamma}$, we have $j_{01}\rest H^s\in\Lll$ and
$j_{12}\rest H^s=j(j_{01}\rest H^s)$, and moreover,
$j_{12}=\bigcup_{t\in\vec{\Gamma}}j_{01}(j_{01}\rest 
H^t)$,
\item\label{item:still_fixes_Gamma} $j_{12}\rest\Gamma=\id$.
 \end{enumerate}
\end{clm}
\begin{proof}
Part \ref{item:j_01_rest_alpha_in}: Because $\alpha<\Theta$ there is a 
surjection $\pi:V_{\lambda+1}^{\Lll}\to\Lll|\alpha$ with $\pi\in\Lll$.
But then note that
$j\com\pi=j(\pi)\com j$,
which suffices as $j(\pi)\in\Lll$ and $j\rest 
V_{\lambda+1}^{\Lll}\in\Lll$
(the latter is the canonical extension of $j\rest V_\lambda^{\Lll}$).

So let $k=\bigcup_{\alpha<\Theta}j(j\rest(\Lll|\alpha))$.

Part  \ref{item:k_elem}:
We get $\Sigma_0$-elementarity immediately,
hence $\Sigma_1$-elementarity by $\in$-cofinality.
But then because $L_\Theta(V_{\lambda+1}^{\Lll})\sats\ZF^-$,
this implies full elementarity.

Part \ref{item:mu_1_derived_from_k}: Recall that $V_{\lambda+2}^{\Lll}\sub\Lll|\Theta$, by Proposition \ref{prop:Theta_properties}. Let $\alpha\in(\lambda,\Theta)$. Let  
$\bar{\mu}=\mu_0\cap(\Lll|\alpha)$ and $\bar{j}=j\rest(\Lll|\alpha)$.
Then $\bar{\mu},\bar{j}\in\Lll$
and $\bar{j}\rest V_\lambda^{\Lll}=j\rest V_\lambda^{\Lll}$.
Note $\Lll|\Theta\sats$``$\bar{\mu}$ is the measure derived from $\bar{j}$
with seed $\bar{j}\rest V_\lambda^{\Lll}$''.
So $\Lll|\Theta\sats$``$j(\bar{\mu})$
is the measure derived from $j(\bar{j})$
with seed $j(\bar{j}\rest V_\lambda^{\Lll})$''.
But note $j(\bar{\mu})=\mu_1\cap j(\Lll|\alpha)$
and $j(\bar{j})\sub k$
and $j(\bar{j}\rest V_\lambda^{\Lll})=k\rest V_\lambda^{\Lll}$.
Since this holds for all $\alpha\in(\lambda,\Theta)$, we are done.

 Part \ref{item:Ult_0(Lll_Theta,mu_1)}:
like the proper class case (see the proof of Proposition \ref{prop:ultrapowers}), but simpler.

Part \ref{item:j_12(Lll_Theta)}: Let $f\in\Lll$
with $f:V_{\lambda+1}^{\Lll}\to\Lll|\Theta$.
Then $\rg(f)\sub\Lll|\alpha$ for some $\alpha<\Theta$,
so easily $f\in\Lll|\Theta$, which by the previous part suffices.

Parts \ref{item:Ult_mu_1=Lll}, \ref{item:ult_map_is_pw_image}:
For $s\in\vec{\Gamma}$, write $\ell^s=(j_{01})^s\in\Lll$.
By 
elementarity of 
$j_{01}$ and since $j_{01}(s)=s$, the map
$j_{01}(\ell^s):H^s\to H^s$
is the unique $\ell'\in\mathscr{E}(H^s)$
such that $\ell'(s^-)=s^-$ and $\ell'\rest V_\lambda^{\Lll}=j_{01}(j_{01}\rest 
V_\lambda^{\Lll})=k\rest V_\lambda^{\Lll}$.
 By Proposition \ref{prop:extensions},
  $\Lll=\Hull^{\Lll}(V_{\lambda+1}^{\Lll}\cup\Gamma)$.
It follows that setting \[ 
j'=\bigcup_{s\in\vec{\Gamma}}j_{01}(\ell^s), \]
 $j':\Lll\to\Lll$ is cofinal $\Sigma_0$-elementary,
hence (by ZF) fully elementary. Moreover, $j'\rest(\Lll|\Theta)=k$.
But then again as in the proof of Proposition \ref{prop:ultrapowers}, $\Ult_0(\Lll,\mu_1)$
satisfies \L o\'{s}'s Theorem, is extensional, isomorphic to $\Lll$, 
and its ultrapower map $i$ factors through $j'$, and
in fact $i=j'$, so we are done.

Part \ref{item:still_fixes_Gamma}: This follows easily from the characterization
of $j_{12}$ we just gave.
\end{proof}

We have $j=j_0=i^{\Lll}_{\mu_0}:\Lll\to\Lll$ 
and $j_1=i^{\Lll}_{\mu_1}:\Lll\to\Lll$  (and the 
ultrapowers via $\mu_0,\mu_1$ each satisfy \L o\'{s}'s Theorem,
are extensional and isomorphic to $\Lll$) by Claim \ref{clm:mu_j_rep}.
Fix $\Gamma$, etc, as above.
So $j_0\rest\Gamma=\id$, $j^s:H^s\to H^s$
for $s\in\vec{\Gamma}$,
\[ j_0=\bigcup_{s\in\vec{\Gamma}}(j_0)^s\text{ and } 
j_1=\bigcup_{s\in\vec{\Gamma}}j_0((j_0)^s),\]
and $j_1\rest\Gamma=\id$.
But then we have the same circumstances with $(\Lll,j_1,\mu_1,\Gamma)$
as with $(\Lll,j_0,\mu_0,\Gamma)$, so we can iterate the process through all 
$n<\om$.
So $(\Lll,\mu_0)$ is $\om$-iterable, and the ultrapower map given by $\mu_n$
coincides with the embedding $j_n$ given by iterating the process above.

Write $\left<\Lll_m\right>_{m\leq\om}$
for the first $(\om+1)$ iterates (so $\Lll_m=\Lll$ for $m<\om$), and 
$j_{mn}:\Lll_m\to\Lll_n$,
for $0\leq m\leq n\leq\om$, the iteration maps. (So $j_n=j_{n,n+1}$, and $\Lll_\om$ is the direct limit $\dirlim_{m\leq n<\om}(\Lll_m,\Lll_n;j_{mn})$.)
We next show  that $\Lll_\om$ is wellfounded.\footnote{At this point,
Woodin's argument in \cite{woodin_sem2} seems to use some $\DC$.
We avoid this by arguing more like in the standard proof of linear iterability,
adapted to the ``covering system'' setup.}

For $n,\ell<\om$ and  $\alpha\in\OR$, say $\alpha$ is \emph{$(j_n,\ell)$-stable}
iff $j_{n+\ell, m}(\alpha)=\alpha$ for all $m\in[n+\ell,\om)$.
Likewise for $s\in[\OR]^{<\om}$ (where we require $j_{n+\ell,m}(s)=s$).

Let $\vec{\OR'}=[\OR\cut (\lambda+1)]^{<\om}\cut\{\emptyset\}$.
 Given  $k\in\mathscr{E}_{\nt}(V_\lambda^{\Lll})$ with $k\in\Lll$,
 define the  \emph{finite iterates} of $k$, if possible, as follows: 
 Start with $k_0=k$. Given $k_n$, if $(k_n)^+$ (the canonical
 extension) is such that $(k_n)^+\in\mathscr{E}(V_{\lambda+1}^{\Lll})$
 then let 
$k_{n+1}=(k_n)^+(k_n)$; otherwise $k_{n+1}$ is not defined.
If we let $\bar{j}=j_0\rest V_\lambda^{\Lll}$, then by the preceding observations, each $\bar{j}_n$ exists.
Now given
$s\in\vec{\OR'}$,
 write $\pi^{k}_{sm,m+1}$
 for the unique elementary map $\pi:H^s\to H^s$
 such that $\pi\rest V_\lambda^{\Lll}=k_m$ (the $m$th iterate of $k$) and 
$\pi(s^-)=s^-$,
 if such exists; otherwise $\pi^{k}_{sm,m+1}$ is not defined.

\begin{clm}\label{clm:every_ord_ev_stable}For each $\alpha\in\OR$ there is 
$\ell<\om$ such that 
$\alpha$ is $(j_0,\ell)$-stable.
 \end{clm}
\begin{proof} Suppose otherwise and let $\alpha$ be the least counterexample. 
 Let $s\in\vec{\Gamma}$  with $\alpha\in H^s$. Let $k=j_0\rest V_\lambda^{\Lll}$.
 Then  $\Lll\sats$ ``$\alpha$ is the least 
$\beta\in\OR\cap H^s$ such 
that 
 $\pi^{k}_{sm,m+1}(\beta)>\beta$
 for cofinally many $m<\om$''.
 Applying $j_{0n}$ to this statement with some
 $n$ such that $\alpha'=j_{0n}(\alpha)>\alpha$,
 and letting $k'=j_{0n}(k)=j_{n}\rest V_\lambda^{\Lll}$ (and recalling $j_{0n}(s)=s$), we get
 $\Lll_n=\Lll\sats$ ``$\alpha'$ is the least $\beta\in\OR\cap H^s$
 such that 
$\pi^{k'}_{s,m,m+1}(\beta)>\beta$
for cofinally many $m<\om$''.
 But since $k'=k_n$, clearly $\alpha'=\alpha$, a contradiction.\end{proof}

Using the hulls and maps defined above,
$\Lll$ can form a covering system for the direct limit,
just like the standard covering systems considered
in inner model  theory and/or in \cite{woodin_sem2},
with some direct 
limit $\widetilde{\Lll_\om}$, which we will eventually show $=\Lll_\om$.
This proceeds as follows. Given any  $k\in\mathscr{E}_{\nt}(V_\lambda^{\Lll})$  with $k\in\Lll$
and given $s\in\vec{\OR'}$, say that $s$ 
is \emph{$k$-potentially-stable} iff 
 $\pi^{k}_{sm,m+1}$ exists for all $m<\om$ (hence  $\pi^{k}_{sm,m+1}\in\mathscr{E}(H^s)$, etc).
Similarly, for $\ell<\om$, we say that $s$ is 
\emph{$(k,\ell)$-potentially-stable}
iff this holds for all $m\geq\ell$.
Also let $\mathscr{P}^k$
denote the class of pairs $(s,\ell)$
such that $s$ is $(k,\ell)$-potentially-stable.
If $(s,\ell)\in\mathscr{P}^k$ and $\ell\leq m\leq n$ let 
$\pi^{k}_{smn}:H^s\to H^s$
be defined naturally by composition. That is,
$\pi^{k}_{smm}=\id$,
 $\pi^{k}_{sm,m+1}$ is already defined,
and $\pi^{k}_{sm,n+1}=\pi^{k}_{sn,n+1}\com\pi^{k}_{smn}$.
Say that $k$ is \emph{valid} if for every $s\in\vec{\OR'}$
there is $n<\om$ such that $(s,n)\in\mathscr{P}^k$.

Note that if $\alpha\in\OR\cut(\lambda+1)$ is 
$(j_0,\ell)$-stable then
$\{\alpha\}$ is $(j_0\rest V_\lambda^{\Lll},\ell)$-potentially-stable,
or equivalently, $(\{\alpha\},\ell)\in\mathscr{P}^{j_0\rest V_\lambda^{\Lll}}$.
So by Claim \ref{clm:every_ord_ev_stable}, for every $s\in\vec{\OR'}$
there is $\ell<\om$ such that $(s,\ell)\in\mathscr{P}^{j_0\rest V_\lambda^{\Lll}}$;
that is,
 $j_0\rest V_\lambda^{\Lll}$ is valid. 

Let $(s,\ell),(t,m)\in\mathscr{P}^k$.
We write $(s,\ell)\leq(t,m)$
iff $s\sub t$ and $\ell\leq m$. Supposing $(s,\ell)\leq(t,m)$, define
\[ \pi^{k}_{s\ell,t m}:H^s\to H^t
\]
to have the same graph as has $\pi^{k}_{smn}$;
so $\pi^{k}_{s\ell,tm}$ is  $\Sigma_0$-elementary.

Assume now that $k$ is valid.
Note then that $\mathscr{P}^k$ is directed under $\leq$ and the embeddings 
$\pi^{k}_{s\ell,tm}$ commute; that is,
if $(s,\ell)\leq(t,m)\leq(u,n)$ then
\[ \pi^{k}_{s\ell,u n}=\pi^{k}_{tm,un}\com\pi^{k}_{s\ell,t m}.\]

So we define the directed system 
\begin{equation}\label{eqn:cov_sys}\mathscr{D}^k=\left<H^s,H^t;\pi^{k}_{sm,tn}\bigm|
(s,m)\leq(t,n)\in\mathscr{P}^k\right>.\end{equation}
Let $\widetilde{\Lll}^k_{\om}$ be its direct limit,
and
$\pi^{k}_{sm,\om}:H^s\to\widetilde{\Lll}^k_\om$
the direct limit map. Note that $\mathscr{P}^k$, validity,
$\mathscr{D}^k$, $\widetilde{\Lll}^k_\om$
and $\left<\pi^{k}_{sm\om}\right>_{(s,m)\in\mathscr{P}^k}$ are all definable over $\Lll$ from the parameter $k$, and uniformly so in $k$.

Now if $(s,\ell)\in\mathscr{P}^{j_n\rest V_\lambda^{\Lll}}$  and $s$ is 
$(j_n,\ell)$-stable and $\ell\leq m<\om$
then
\[ j_{n+\ell,n+m}\rest H^s=\pi^{j_n\rest V_\lambda^{\Lll},s}_{\ell,m}.\]
So working in $V$, we can define
$\sigma_n:\widetilde{\Lll}^{j_n\rest V_\lambda^{\Lll}}_\om\to \Lll_\om$
by setting
\[ \sigma_n(\pi^{j_n\rest V_\lambda^{\Lll}}_{s\ell,\om}(x))=j_{n+\ell,\om}(x)\]
whenever $(s,\ell)\in\mathscr{P}^{j_n\rest V_\lambda^{\Lll}}$, $x\in H^s$ and $s$ is $(j_n,\ell)$-stable. Clearly $\widetilde{\Lll}^{j_n\rest V_\lambda^{\Lll}}_\om$ and 
$\sigma_n$
are independent of $n$, so just write $\widetilde{\Lll}_\om,\sigma$
(but it seems these might depend on $j$).

Now $\widetilde{\Lll}_\om=\Lll_\om$
and $\sigma=\id$, because for each $x\in \Lll$
there is $s\in\vec{\Gamma}$ such that $x\in H^s$,
so for any $n<\om$,  noting that $s$ is $(j_0\rest V_\lambda^{\Lll},n)$-stable,
we then get $j_{n\om}(x)=\sigma(\pi^{j_0\rest V_\lambda^{\Lll}}_{sn,\om}(x))$,
so $j_{n\om}(x)\in\rg(\sigma)$, as desired.

\begin{clm} $\Lll_\om$ is wellfounded. \end{clm}
\begin{proof}Suppose not.
Then there is a least $\alpha\in\OR$ such that $j_{0\om}(\alpha)$
is in the illfounded part of $\Lll_\om$. Let $s\in\vec{\Gamma}$ with $\alpha\in 
H^s$. 
Using the covering system above, we have a definition of $\Lll_\om$
over $\Lll$ from the parameter $k=j_0\rest V_\lambda^{\Lll}$,
and the illfounded part can be computed by $\Lll$.
So $\Lll\sats$ ``$(s,0)\in\mathscr{P}^{k}$ and $\alpha$ is the 
least $\beta\in\OR\cap H^s$
such that $\pi^{k}_{s0,\om}(\beta)$ is in the illfounded part 
of $\Lll_\om$''.
Now let $m<\om$ be such that there is $\alpha_1<j_{0m}(\alpha)$
and $j_{m\om}(\alpha_1)$ is in the illfounded part.
Let  $t\in\vec{\Gamma}$ with $s\sub t$
and  $\alpha_1\in H^t$.
Note that we may in fact assume that $s=t$, by enlarging $s$
in the first place, because this does not affect the minimality of $\alpha$ 
there. Therefore applying $j_{0m}$ to the first statement
and letting $(\alpha',k')=j_{0m}(\alpha,k)$,
we have $\Lll\sats$``$\alpha'$ is the least $\beta\in\OR\cap H^s$
such that $\pi^{k'}_{s0,\om}(\beta)$ is in the illfounded part 
of $\Lll_\om$'' (noting here that $\Lll_\om$ is defined in the same manner
from $k$ or $k'=(k)_m$).
But this contradicts the choice of $\alpha_1<\alpha'$ and $m$ and $t=s$.
\end{proof}

We now define the usual $*$-map.
That is, for $\alpha\in\OR$ let
$\alpha^*=\pi^{j_0\rest V_\lambda^{\Lll}}_{sm,\om}(\alpha)$
whenever $\alpha\in s^-$ and $m<\om$ is large enough. It is straightforward to verify:

\begin{clm}\label{clm:*_definability} We have:
\begin{enumerate}[label=\arabic*.,ref=\arabic*]\item\label{item:*_definability} The map $\OR\to\OR$ given by $\alpha\mapsto\alpha^*$,
 is definable over $\Lll$ from the parameter $j_n\rest V_\lambda^{\Lll}$,
 uniformly in $n<\om$.
 \item  $\alpha^*=\lim_{n<\om}j_{n\om}(\alpha)$.
 \end{enumerate}
\end{clm}

Part \ref{item:*_definability} of the claim is established by
considering the variants of the definition of $\alpha^*$  given by replacing $j_0$ with $j_n$ for $n<\om$; they all give the same value.

Now let
$\Gamma^*=j_{0\om}``\Gamma=\{\alpha^*\bigm|\alpha\in\Gamma\}$.
Then clearly 
$\Lll_\om=\Hull^{\Lll_\om}(V_{\lambda_\om+1}^{\Lll_\om}\cup\Gamma^*)$,
where $\lambda_\om=j_{0\om}(\lambda)$,
as in Definition \ref{dfn:iterates};
likewise $\mu_\om=j_{0\om}(\mu_0)$, etc.
Define $\widetilde{j}_{\om,\om+1}$ by pushing $j_{01}$ through
as before; that is, for each $s\in\vec{\Gamma}$,
\[ \widetilde{j}_{\om,\om+1}\rest 
(H^{s^*})^{\Lll_\om}=j_{0\om}((j_0)^s).\]
Much as before, in fact $\widetilde{j}_{\om,\om+1}=j_{\om,\om+1}$ (where $j_{\om,\om+1}=j_\om$ is the 
ultrapower 
map
given by $\mu_\om$), and $\Ult_0(\Lll_\om,\mu_\om)=\Lll_\om$ etc.
Moreover, $\widetilde{j}_{\om,\om+1}\rest\Gamma^*=\id$.
So we have the same circumstances at stage $\om$
as those we had at stage $0$. Therefore by the same proof,
we reach $(\om+\om+1)$-iterability,
and in fact $\om^2$-iterability, by iterating the whole process $\om$-many times. This completes the proof of the lemma.\end{proof}

We are now ready to state and prove the main theorem of the paper; recall that \emph{relevant triple} was introduced in Definition \ref{dfn:relevant},
and \emph{proper}
in Definition \ref{dfn:proper}.
\begin{tm}\label{tm:main}
Let $(\Lll,\lambda,j)$ be a relevant triple. 
Let $k=j\rest V_{\lambda+2}^{\Lll}$. Then
 \[ V_{\lambda+2}^{\Lll[k]}=V_{\lambda+2}^{\Lll},\]
 and if $j$ is proper then $\Lll[k]=\Lll[j]$.
\end{tm}

\begin{proof}
By Proposition \ref{prop:ultrapowers},
we may assume $j$ is proper, and by the same proposition, it then follows that $\Lll[k]=\Lll[j]$.

We now continue on with everything as defined in the proof of Lemma \ref{lem:omega^2-it}.
For convenience of reference, we continue the numbering of claims started there,
and continue to refer to those four claims just as ``Claim \ref{clm:mu_j_rep}'', etc, without mentioning that they come from the proof of Lemma \ref{lem:omega^2-it}.

So 
we have the iterates $(\Lll_\alpha,\lambda_\alpha,\mu_\alpha,j_\alpha)$ for 
$\alpha\leq \om+\om$, and the models $\Lll_\alpha$
are
definable over $\Lll$ from the parameter $j_0\rest V_\lambda^{\Lll}$.
Note that $V_{\lambda_\om+1}^{\Lll_\om}$
is just the direct limit of $V_{\lambda+1}^{\Lll}$ under 
the maps $j_n\rest 
V_{\lambda+1}^{\Lll}$, etc.
And $\left<j_n\rest V_{\lambda+1}^{\Lll}\right>\in\Lll$, so
 $\lambda_\om=j_{0\om}(\lambda)<\Theta$
and in fact $\Lll$ has a surjection $\pi:V_{\lambda+1}^{\Lll}\to 
V_{\lambda_\om+1}^{\Lll_\om}$. Since $\Lll_\om\sub \Lll$, it follows that
 $j_{0\om}(\Theta)=\Theta$.
Note that $j_{0\om}$ is continuous at $\Theta$.

Similarly,
$\left<j_{\om+n}\rest 
V_{\lambda_\om+1}^{\Lll_\om}\right>\in\Lll_\om$ and $\Lll_\om$ defines 
$\Lll_{\om+\om}=j_{0\om}(\Lll_\om)$.
In $V$ we have the iteration map
$j_{\om,\om+\om}:\Lll_\om\to \Lll_{\om+\om}$.
We can now prove,
with basically the usual argument:
\begin{clm}\label{clm:*=j_om,om+om}$\alpha^*=j_{\om,\om+\om}
(\alpha)$ for all $\alpha\in\OR$.\end{clm}
\begin{proof}
Fix $\alpha$.
Let $n<\om$ and $\bar{\alpha}$ be such that $j_{n\om}(\bar{\alpha})=\alpha$
and $\alpha$ is $(j_0,n)$-stable.
Let $s\in\vec{\Gamma}$ with $\bar{\alpha}\in H^s$ and
let $k=j_n\rest V_\lambda^{\Lll}$. Then
\[ 
\Lll=\Lll_n\sats\text{``}\pi^{k}_{s0,\om}(\bar {\alpha})=\alpha\text{''}.\]
We have $s^*=j_{n\om}(s)\in\vec{\Gamma^*}$ and
$\alpha^*=j_{n\om}(\alpha)$, so
letting  $k'=j_\om\rest 
V_{\lambda_\om}^{\Lll_\om}=j_{n\om}(k)$,
\begin{equation}\label{eqn:L_om_maps_alpha_to_alpha^*}
\Lll_\om\sats\text{``}\pi^{k'}_{s^*0,\om}(\alpha)=\alpha^*\text{''};\end{equation}
the two superscript ``$*$''s in line (\ref{eqn:L_om_maps_alpha_to_alpha^*}) are as computed in $\Lll$, not $\Lll_\om$.
But 
\[ (\pi^{k'}_{s^*0,\om}(\alpha))^{\Lll_\om}=j_{\om,\om+\om}(\alpha),\]
 since $j_\om(s^*)=s^*$
 (and applying facts about $(\Lll,j,\Gamma)$ to $(\Lll_\om,j_\om,\Gamma^*)$).
\end{proof}

Given $C\sub\Lll_\om$, write $\Lll_\om[C]=L(V_{\lambda_\om+1}^{\Lll_\om},C)$.
We are interested in 
$\Lll_\om[*]$, $\Ll_\om[j_\om]$
and $\Lll_\om[j_{\om,\om+\om}]$.
We are now ready to make a key observation:

\begin{clm}\label{clm:V_lambda_om+2} 
$V_{\lambda_\om+2}^{\Lll_\om}=V_{\lambda_\om+2}^{\Lll_\om[*]}$.
\end{clm}
\begin{proof} Let $X\in\Lll_\om[*]$ with $X\sub 
V_{\lambda_\om+1}^{\Lll_\om}$;
it suffices to see $X\in\Lll_\om$.
Well, $\Lll_\om[*]=L(V_{\lambda_\om+1}^{\Lll_\om},*)$ and $*\sub\OR^2$, so there are $\eta\in\OR$ and $y\in 
V_{\lambda_\om+1}^{\Lll_\om}$
such that for all $x\in V_{\lambda_\om+1}^{\Lll_\om}$, we have
\[ x\in X\iff \Lll_\om[*]\sats\varphi(x,y,\eta,*)\]
(where the reference to $*$ in $\varphi$ is literally via a predicate).  Let 
$n<\om$ be such that $\eta$ is $(j_0,n)$-stable
and $y\in\rg(j_{n\om})$
and let $j_{n\om}(\bar{y})=y$.
Let $k_{n\om}=j_{n\om}\rest V_{\lambda+1}^{\Lll}$, so $k_{n\om}\in\Lll$. Working in $\Lll=\Lll_n$,
define $\bar{X}_n\sub V_{\lambda+1}^{\Lll}$ by setting
\[ z\in \bar{X}_n\iff z\in V_{\lambda+1}^{\Lll}\text{ and } 
\Lll_\om[*]\sats\varphi(k_{n\om}(z),k_{n\om}(\bar{y}),\eta,*). \]
Let $X'=j_{n\om}(\bar{X}_n)$. We claim  that $X'=X$, so $X\in 
\Lll_\om$, as desired. For certainly for each $x\in V_{\lambda_\om+1}^{\Lll_\om}\cap\rg(j_{n\om})$, we have $x\in X'$ iff $x\in X$, essentially directly by definition. So it suffices to
see that for all $\ell\in[n,\om)$,
we have $j_{n\ell}(\bar{X}_n)=\bar{X}_\ell$ (where $\bar{X}_\ell$ is defined in the same manner as $\bar{X}_n$, but with ``$n$'' replaced by ``$\ell$'').
But $j_{n\ell}(\Lll_\om)=\Lll_\om$ and 
$j_{n\ell}(*)=*$, 
so $j_{n\ell}(\Lll_\om[*])=\Lll_\om[*]$;
 also $j_{n\ell}(k_{n\om})=k_{\ell\om}$ and $j_{n\ell}(\eta)=\eta$.
 Therefore $j_{n\ell}(\bar{X}_n)=\bar{X}_\ell$, as desired.
\end{proof}

Given a transitive proper class $N\sats\ZF$,
we say that a class $C\sub N$ is 
\emph{amenable} to $N$ iff $C\cap x\in N$ for each $x\in N$.
For example, $j_{\om,\om+\om}\rest\OR$ is amenable to $\Ll_\om[*]$.
\begin{clm}\label{clm:set_seg_of_*_suffices}
We have: 
\begin{enumerate}[label=\arabic*.,ref=\arabic*] 
\item\label{item:models_match}$\Lll_\om[j_\om]=\Lll_\om[j_{\om,\om+\om}]
=\Lll_\om [ * ] =
\Lll_\om[j_\om\rest\Theta]$.
 \item\label{item:embs_from_j_om_rest_Theta_etc} $j_\om$, $j_{\om,\om+\om}$, and $*$ are definable over 
$\Lll_\om[j_\om]$ from the parameters
 $j_\om\rest\Theta$ and $j_\om\rest V_{\lambda_\om}^{\Lll_\om}$.
 \item\label{item:j_om_rest_Theta_stable} $j_n(j_\om\rest\Theta)=j_\om\rest\Theta$ and $j_{n\om}(j_\om\rest\Theta)=j_{\om+\om}\rest\Theta$ for all $n<\om$.
 
 \end{enumerate}
\end{clm}
\begin{proof}Part \ref{item:models_match}:

\underline{$\Lll_\om[*]=\Lll_\om[j_{\om,\om+\om}]$}:
We have $\Lll_\om[*]\sub\Lll_\om[j_{\om,\om+\om}]$ by
 Claim \ref{clm:*=j_om,om+om}.
 Also by that claim
 and since $j_{\om,\om+\om}\rest V_{\lambda_\om+1}^{\Lll_\om}\in\Lll_\om$, note that
 $j_{\om,\om+\om}$  is definable over $\Lll_\om[*]$ from the predicate $*$ and parameter $j_{\om,\om+\om}\rest V_{\lambda_\om+1`}^{\Lll_\om}$.
Therefore
$\Lll_\om[*]=\Lll_\om[j_{\om,\om+\om}]$.

\underline{$\Lll_\om[j_\om]=\Lll_\om[j_{\om,\om+\om}]$}:
Clearly $j_{\om,\om+\om}$ is definable from 
$j_{\om,\om+1}=j_\om$. 
To define $j_{\om}$ in $\Lll_\om[j_{\om,\om+\om}]$, it 
suffices to define $j_{\om}\rest V_{\lambda_\om+2}^{\Lll_\om}$,
since this gives $\mu_\om$.
But given $X\in V_{\lambda_\om+2}^{\Lll_\om}$,
we have $X'=j_{\om}(X)\sub V_{\lambda_\om+1}^{\Lll_\om}$,
and for $x\in V_{\lambda_\om+1}^{\Lll_\om}$, we have
\[ x\in X'\iff j_{\om+1,\om+\om}(x)\in 
j_{\om+1,\om+\om}(X')=j_{\om,\om+\om}(X). \]
But $j_{\om+1,\om+\om}\rest V_{\lambda_\om+1}^{\Lll_\om}\in \Lll_\om$,
so using $j_{\om,\om+\om}$, this computes $X'$, uniformly in $X$.

\underline{$\Lll_\om[j_\om]=\Lll_\om[j_\om\rest\Theta]$}:
Working in $\Lll_\om[j_\om\rest\Theta]$,
we can define $j_\om\rest L_\Theta(V_{\lambda_\om+1}^{\Lll_\om})$
from the parameters $j_\om\rest\Theta$ and $j_\om\rest 
V_{\lambda_\om}^{\Lll_\om}$ (using the canonical extension
to define $j_\om\rest V_{\lambda_\om+1}^{\Lll_\om}$),
and since $V_{\lambda_\om+2}^{\Lll}\sub L_\Theta(V_{\lambda_\om+1}^{\Lll_\om})$,
this determines $\mu_\om$ and hence  $j_\om$.

Part \ref{item:embs_from_j_om_rest_Theta_etc}: By some of the foregoing calculations.

Part \ref{item:j_om_rest_Theta_stable}:
To see $j_n({j_\om\rest\Theta})={j_\om\rest\Theta}$, just use the foregoing calculations together with the fact that $j_n(\Lll_\om,{*},\Theta)=(\Lll_\om,{*},\Theta)$. The fact that $j_{n\om}({j_\om\rest\Theta})={j_{\om+\om}\rest\Theta}$ is similar;
we have $j_{n\om}(\Lll_\om,*)=(\Lll_{\om+\om},*')$
where $*'$ is the $*$-map associated to the direct limit system $(\mathscr{D}^{j_\om\rest V_{\lambda_\om}})^{\Lll_\om}$ (cf.~line (\ref{eqn:cov_sys})),
and $j_{n\om}(\Theta)=\Theta$,
so $j_{\om+\om}\rest\Theta$ is defined over $\Lll_\om$ from $(\Lll_{\om+\om},*')$
just as $j_{\om}\rest\Theta$ is over $\Lll$ from $(\Lll_\om,*)$.
\end{proof}

By Claims \ref{clm:V_lambda_om+2} and \ref{clm:set_seg_of_*_suffices}, 
$V_{\lambda_\om+2}^{\Lll_\om[j_\om]}\sub \Lll_\om$. This yields the version 
of the theorem given by replacing 
$\Lll[j]$ with
$\Lll_\om[j_\om]$ (and hence already gives the relative consistency of ZF + ``there is $\xi$ such that $I_{0,\xi}$ holds and $V_{\xi+2}\sub L(V_{\xi+1})$, and hence there is an elementary $k:V_{\xi+2}\to V_{\xi+2}$'').

To complete the proof of the theorem, we need to see that we can pull back the properties of $\Lll_\om[j_\om]$ to $\Lll[j]$. We achieve this in the next few claims, which are  
 a variant of arguments from \cite{HOD_as_core_model} and  \cite{Theta_Woodin_in_HOD}.

\begin{clm}\label{clm:in_range_implies_stable}
Let $x\in\Lll_\om$. Then:
\begin{enumerate}[label=\arabic*.,ref=\arabic*]\item\label{item:x_in_some_rg}There is $n<\om$ such that $x\in\rg(j_{n\om})$.
 \item\label{item:x_ev_fixed} For each $n<\om$ with $x\in\rg(j_{n\om})$, we have $j_{nm}(x)=x$ for all $m\in[n,\om)$.
 \item\label{item:x_ev_fixed_covering_sys} For each $s\in\vec{\OR'}$ with $x\in H^s$, and each $n<\om$ such that $s$ is $(j_0,n)$-stable, we have $\pi^{j_n\rest V_\lambda^{\Lll}}_{s0,sm}(x)=x$ for all $m<\om$.
\end{enumerate}
\end{clm}
\begin{proof}Part \ref{item:x_in_some_rg} is immediate.

Part \ref{item:x_ev_fixed}: Fix $\bar{x}\in\Lll$ with $j_{n\om}(\bar{x})=x$ and let 
$s\in\vec{\Gamma}$
be such that $\bar{x}\in H^s$. Then $\Lll=\Lll_n\sats$``$x=\pi^{j_n\rest 
V_\lambda^{\Lll}}_{s0,\om}(\bar{x})$'',
so letting $x'=j_{nm}(x)$ and $\bar{x}'=j_{nm}(\bar{x})$, we have 
$\Lll\sats$``$x'=\pi^{j_m\rest V_\lambda^{\Lll}}_{s0,\om}(\bar{x}')$'',
but $\pi^{j_m\rest V_\lambda^{\Lll}}_{s0,\om}(\bar{x}')=j_{m\om}(\bar{x}')=x$.

Part \ref{item:x_ev_fixed_covering_sys}:
This is by part \ref{item:x_ev_fixed}, since  the hypotheses give $\pi^{j_n\rest V_\lambda^{\Lll}}_{s0,sm}\sub j_{n,n+m}$.
\end{proof}

We have already defined $\alpha^*$ for $\alpha\in\OR$;
using the preceding claim, we now extend this function to domain
$\Lll_\om$. Fix $x\in\Lll_\om$; we define $x^*$
(this will agree with the earlier definition of $x^*$ in case $x\in\OR$).
For  $(s,n)\in\mathscr{P}^{j_0\rest V_\lambda^{\Lll}}$
such that $x\in H^s\cap\Lll_\om$ and $\pi^{j_0\rest V_\lambda^{\Lll}}_{sn,sm}(x)=x$ for all $m\in[n,\om)$, define
\[ x^*=\pi_{sn,\om}^{j_0\rest V_\lambda^{\Lll}}(x). \]
Like in Claim \ref{clm:*_definability}, we have
 $x^*=\lim_{n\to\om}j_{n\om}(x)$, 
and $x\mapsto x^*$, now with domain $\Lll_\om$, is definable over 
$\Lll_\om[j_\om\rest\Theta]$.
The next claim is proved just
like Claim \ref{clm:*=j_om,om+om}:

\begin{clm}\label{clm:j_om,om+om=*_over_Lll_om}
$j_{\om,\om+\om}:\Lll_\om\to\Lll_{\om+\om}$ is precisely the map $x\mapsto x^*$ \tu{(}with domain $\Lll_\om$\tu{)}.\end{clm}

Now let
\[ H=\Hull^{\Lll_\om[j_\om]}(\rg(j_{0\om})\cup\{j_\om\rest\Theta\})=
 \Hull^{\Lll_\om[j_\om]}(\rg(j_{0\om})\cup\{*\})
\]
(the notation for the hull on the right
is supposed to mean that the class $*$ is available as a predicate; the two 
hulls are of course
identical by the previous claims).
 We can now prove the main remaining fact:

\begin{clm}\label{clm:hull_is_conservative} We have:
\begin{enumerate}[label=\arabic*.,ref=\arabic*]\item\label{item:hull_is_conservative_i}
$H\cap\Lll_\om=\rg(j_{0\om})$, and \item\label{item:hull_is_conservative_ii} $H\elem 
\Lll_\om[j_\om]$.
\end{enumerate}
\end{clm}
\begin{proof}
Part \ref{item:hull_is_conservative_i}: Let $z\in H\cap\Lll_\om$.
We first observe that $j_{0n}(z)=z$ for all $n<\om$. For this,
fix $x\in\rg(j_{0\om})$ and  a
formula $\varphi$
such that
\[ \Lll_\om[*]\sats z\text{ is the unique }y\text{ such that 
}\varphi(y,x,*).\]
Applying $j_{0n}$ 
gives the following, where
$z'=j_{0n}(z)$:
\[ \Lll_\om[*]\sats z'\text{ is the unique }y\text{ such that }
\varphi(y,x,*); \]
we use here that $j_{0n}$ fixes $\Lll_\om$, $*$ and $x$, by Claim 
\ref{clm:in_range_implies_stable}.
So $j_{0n}(z)=z'=z$.

Since $j_{0n}(z)=z$ for all $n<\om$, we have
$z^*=j_{0\om}(z)\in\rg(j_{0\om})$.
But we need $z\in\rg(j_{0\om})$, not just $z^*$.
Well, by Claim \ref{clm:j_om,om+om=*_over_Lll_om},
 since $z\in\Lll_\om$, we have $z^*=j_{\om,\om+\om}(z)$.
Let $s\in\vec{\Gamma}$ (so $s^*\in\vec{\Gamma^*}$)
with $z\in(H^{s^*})^{\Lll_\om}$.
Since  $j_{0\om}\rest V_{\lambda+1}^{\Lll}\in\Lll$, we have  $j_{0\om}\rest H^s\in\Lll$ and $j_{0\om}\rest H^s$ is the unique elementary  $\pi:H^s\to(H^{s^*})^{\Lll_\om}$
with $j_{0\om}\rest V_{\lambda+1}^{\Lll}\sub\pi$
and $\pi(s^-)=(s^*)^-$.
Since also $s^*=j_{0\om}(s)\in\vec{\Gamma^*}$,
it easily follows that (and let $\bar{j}$ be defined as)
\[ \bar{j}=j_{\om,\om+\om}\rest(H^{s^*})^{\Lll_\om}=j_{0\om}(j_{0\om}\rest
H^s)\in\rg(j_{0\om}).\]
But $z\in(H^{s^*})^{\Lll_\om}$
and  $z^*=j_{\om,\om+\om}(z)=\bar{j}(z)$
and both $z^*$ and $\bar{j}$
are in $\rg(j_{0\om})$. But then $z=\bar{j}^{-1}(z^*)\in\rg(j_{0\om})$, 
as desired.

Part \ref{item:hull_is_conservative_ii}: 
By Claim \ref{clm:in_range_implies_stable}, $\Lll_\om[*]=\Lll_\om[k]$ where $k=j_\om\rest\Theta$,
and for convenience we deal with the latter form. Let $x\in\rg(j_{0\om})$
and suppose that
\[ \Lll_\om[k]\sats\exists z\ 
\varphi(z,x,k),\]
and $m<\om$ where $\varphi$ is $\Sigma_m$.
We want to find $z\in H$ with 
$\Lll_\om[k]\sats\varphi(z,x,k)$.\footnote{The referee suggested the following alternate argument
for this. We can fix some term $t$
and $n<\om$ 
such that there are $y\in V_{\lambda+1}^{\Lll}$
and  $z\in\Lll_\om[k]$ such that $z=t^{\Lll_\om[k]}(j_{n\om}(y),x,k)$
and $\Lll_\om[k]\sats\varphi(z,x,k)$
(by minimizing on any ordinal parameters, these are not needed to define $z$ here). Let $k_{n\om}=j_{n\om}\rest V_{\lambda+1}^{\Lll}$. Then $\Lll=\Lll_n\sats$``there are $\widetilde{y}\in V_{\lambda+1}$ and  $\widetilde{z}\in\Lll_\om[k]$
such that $\widetilde{z}=t^{\Lll_\om[k]}(k_{n\om}(\widetilde{y}),x,k)$ and $\varphi(\widetilde{z},x,k)$.
By Claim \ref{clm:in_range_implies_stable},
$j_{0n}(x)=x$ for all $n<\om$;
we also have $j_{0n}(k)=k$ 
and $j_{0n}(k_{0\om})=k_{n\om}$.
So we can pull the statement back to $\Lll=\Lll_0$
under $j_{0n}$, 
giving that $\Lll\sats$``there are $\widetilde{y}\in V_{\lambda+1}$ and $\widetilde{z}\in\Lll_\om[k]$
such that $\widetilde{z}=t^{\Lll_\om[k]}(k_{0\om}(\widetilde{y}),x,k)$ and $\varphi(\widetilde{z},x,k)$''. But taking witnesses $\widetilde{y},\widetilde{z}$ to this,
we have $k_{0\om}(\widetilde{y})\in\rg(j_{0\om})$,
so $\widetilde{z}\in H$, which suffices.}
Let $\eta\in\OR$ with $\Theta<\eta$ and 
$x\in L_\eta(V_{\lambda+1}^{\Lll})\elem_{m+10}\Lll$
and  $j_{0\om}(\eta)=\eta$ (we can find such $\eta$
because $j_{0\om}$ arises from a sequence of ultrapowers).
Then because of how $\Lll_\om[k]$ is defined
in $\Lll$, we have
$(\Lll_\om[k])|\eta\elem_{m+1}\Lll_\om[k]$.
Let
\[ 
J=\Hull^{(\Lll_\om[k])|\eta}
(V_{\lambda_\om+1}^{\Lll_\om}\cup\{x,k\}); \]
then $J\elem(\Lll_\om[k])|\eta$
since $V_{\lambda_\om+1}^{\Lll_\om}\cup\{k\}\sub J$.
Let $C$ be the transitive collapse of $J$
and $\pi:C\to J$ the uncollapse map.
Let $\pi(x^C)=x$ and $\pi(k^C)=k$.

Now we claim that $C,x^C,k^C\in\rg(j_{0\om})$ and $\pi\in H$.
For we defined $C,x^C,k^C,\pi$ over $\Lll_\om[k]$ from the parameters
$\eta=j_{0\om}(\eta)$, $x$
and $k$, so they are in $H$. But $C,x^C,k^C$ are also coded
by elements of 
$V_{\lambda_\om+2}^{\Lll_\om[k]}=V_{\lambda_\om+2}^{\Lll_\om}$
(the equality by Claim \ref{clm:V_lambda_om+2}),
so $C,x^C,k^C\in \Lll_\om\cap H$, so $C,x^C,k^C\in\rg(j_{0\om})$ by part \ref{item:hull_is_conservative_i}.

So write $j_{0\om}(C_0,x^{C_0},k^{C_0})=(C,x^C,k^C)$.
Since $\pi\in H$, we have $\pi``(j_{0\om}``C_0)\sub H$.
But
$\pi$ is elementary,
so $C\sats\exists z\ \varphi(z,x^C,k^C)$,
so $C_0\sats\exists z\ \varphi(z,x^{C_0},k^{C_0})$.
So fixing a $z_0\in C_0$ witnessing this,
then $z=\pi(j_{0\om}(z_0))\in H$ is as desired.
\end{proof}

Since $H\elem\Lll_\om[j_\om]$,
$H$ is extensional,
so we can let $N$ be its 
transitive collapse
and $\widehat{j}_{0\om}:N\to\Lll_\om[j_\om]$ the uncollapse map.

\begin{clm}We have:
\begin{enumerate}[label=\arabic*.,ref=\arabic*]
\item\label{item:N_is_L[j]}$\Lll=L(V_{\lambda+1})^N\sub N=\Lll[j]=\Lll[j\rest\Theta]$,
\item\label{item:j-hat_extends_j} $j_{0\om}\sub 
\widehat{j}_{0\om}:\Lll[j]\to\Lll_\om[j_\om]$
is elementary,
\item \label{item:j_def_from_j_rest_V_lambda_cup_Theta}
 $j$ is definable over $\Lll[j]$ from
$j\rest(V_\lambda^{\Lll}\cup\Theta)$.

\end{enumerate}
\end{clm}
\begin{proof}
Parts \ref{item:N_is_L[j]}
and \ref{item:j-hat_extends_j}:
We have $\Lll=L(V_{\lambda+1})^N\sub N$ and $j_{0\om}\sub\widehat{j}_{0\om}$ by Claim \ref{clm:hull_is_conservative}.
Also, $j_{0\om}(j\rest\alpha)=j_\om\rest j_{0\om}(\alpha)$ for all $\alpha<\Theta$, and $j_{0\om}(\Theta)=\Theta$, so  $\widehat{j}_{0\om}(j\rest\Theta)={j_\om\rest\Theta}$. So by Claim \ref{clm:set_seg_of_*_suffices},
$N=\Lll[j]=\Lll[j\rest\Theta]$.
Part \ref{item:j_def_from_j_rest_V_lambda_cup_Theta}: from $j\rest(V_\lambda^{\Lll}\cup\Theta)$
we can recover $j\rest L_\Theta(V_{\lambda+1}^{\Lll})$, hence $\mu_j$, and hence $j$.
\end{proof}

So the first order theory of $(\Lll[j],\Lll)$
is exactly that of $(\Lll_\om[j_\om],\Lll_\om)$,
and the main theorem now follows immediately from
what we know about $\Lll_\om[j_\om]$.
\end{proof}

 \begin{cor}\label{cor:first_of_main}Let $(\Lll,\lambda,j)$ be a relevant triple with $j$ proper.
Let $k=j\rest V_{\lambda+2}^{\Lll}$.
 Then:
 \begin{enumerate}[label=\arabic*.,ref=\arabic*]\item\label{item:theory_in_L[j]} $\Lll[k]=\Lll[j]$ satisfies ZF + ``$V_{\lambda+2}\sub \Lll=L(V_{\lambda+1})$ and $j:\Lll\to\Lll$ is a proper $I_{0,\lambda}$-embedding'',
 and hence satisfies ``$k:V_{\lambda+2}\to V_{\lambda+2}$ is elementary and $k\sub j$,
 $V_{\lambda+1}^{\#}$ does not exist, and
$\Theta_{V_{\lambda+1}}^{\Lll}=\Theta_{V_{\lambda+1}}$ is regular''.

\item\label{item_DC_lambda_goes_in} If
$V_{\lambda+1}^{\Lll}=V_{\lambda+1}$, $\eta\leq\lambda$,  $\eta$ is a limit
and $\eta$-$\DC$  holds then
$\Lll\sats\eta$-$\DC$ and 
$\Lll[j]\sats\eta$-$\DC$.
 \end{enumerate}
\end{cor}

 \begin{proof}
Part \ref{item:theory_in_L[j]}: These statements are mostly
direct consequences of the definitions and Theorem \ref{tm:main}; however, 
for the regularity of $\Theta$ in $\Lll[j]$,
note we have
\[\Theta=\Theta_{V_{\lambda+1}^{\Lll}}^{\Lll}=\Theta_{V_{\lambda+1}^{\Lll[j]}}^{\Lll[j]}=
\Theta_{V_{\lambda_\om+1}^{\Lll_\om}}^{\Lll_\om}=\Theta_{V_{\lambda_\om+1}^{\Lll_\om[j_\om]}}^{\Lll_\om[j_\om]}\]
and $\Lll_\om[j_\om]\sub \Lll$, so $\Theta$ is regular in $\Lll_\om[j_\om]$,
so by elementarity, $\Theta$ is regular in $\Lll[j]$.
Alternatively,
notice that if $f:\alpha\to\Theta$
is cofinal where $\alpha<\Theta$,
then for each $\beta<\alpha$ we can select the least
$\gamma<\Theta$ such that $f(\beta)<\gamma$
and $\Lll|\gamma=\Hull^{\Lll|\gamma}(V_{\lambda+1}^{\Lll})$,
put the resulting sequence of theories $\Th^{\Ll|\gamma}(V_{\lambda+1}^{\Lll})$
together, and thereby form a new subset of $V_{\lambda+1}^{\Lll}$ in $\Lll[j]$, a contradiction.

Part \ref{item_DC_lambda_goes_in}: This is a standard argument.\footnote{Here is the outline of an alternate argument suggested by the referee. For a set $X$, write $\eta$-$\DC_X$ for the restriction of $\eta$-$\DC$ to trees $\mathscr{T}\sub X^{<\OR}$. Observe that $\eta$-$\DC_X$ implies $\eta$-$\DC_{X\times\alpha}$, for every ordinal $\alpha$. Likewise, $\eta$-$\DC_X$ implies $\eta$-$\DC_Y$ whenever $Y$ is the surjective image of $X$. 
So if every set is the surjective image of $X\times\alpha$ for some ordinal $\alpha$,
then $\eta$-$\DC_X$ implies $\eta$-$\DC$. But $V_{\lambda+1}\sub\Lll[j]$ and in $\Lll[j]$, every set is the surjective image of $V_{\lambda+1}\cross\alpha$ for some ordinal $\alpha$. And $\eta$-$\DC_{V_{\lambda+1}}$ holds in $\Lll[j]$; in fact this is downwards absolute to transitive models $M$ of ZF with $V_{\lambda+1}\sub M$.} Note that we assume  $V_{\lambda+1}^{\Lll}=V_{\lambda+1}$ here.
Let $\eta\leq\lambda$ be a limit and suppose that $\eta$-$\DC$ holds
in $V$.
Let $\mathscr{T}\in\Lll[j]$ be a tree of height $\leq\eta$.
Let $k=j\rest\Theta$.
Because $\Lll[j]=\Lll[k]=L(A)$ for some set $A$,
there is a proper class of ordinals $\xi$ such that $\Lll[j]\sats$``$\xi$ is 
regular''. Fix some such $\xi>\Theta$ with $\mathscr{T}\in 
L_\xi(V_{\lambda+1},k)$. Let
\[ 
H=\Hull_{\Sigma_1}^{L_\xi(V_{\lambda+1},k)}(V_{\lambda+1}\cup\{
k, \mathscr{T}\})\]
(this hull is defined as usual (see \S\ref{subsec:notation}),
except that we restrict to allowing only $\Sigma_1$ formulas $\varphi$).

Now $\Lll[j]\sats$``$H$ is closed under $\eta$-sequences''.
For let $\left<x_\alpha\right>_{\alpha<\eta}\in\Lll[j]$
with each $x_\alpha\in H$.
Working in $\Lll[j]$, select for each $\alpha<\eta$
the least $\Sigma_1$ formula $\varphi$
such that for some finite tuple 
$\vec{y}\in(V_{\lambda+1}\cup\{k,\mathscr{T}\})^{<\om}$,
$x_\alpha$ is the unique $x\in L_\xi(V_{\lambda+1},k)$
such that
$L_\xi(V_{\lambda+1},k)\sats\varphi(x,\vec{y})$, and set $\varphi_\alpha=\varphi$. Now working in $V$,
where we have $\eta$-$\DC$,  choose a sequence 
$\left<\vec{y}_\alpha\right>_{\alpha<\eta}$, such that for each $\alpha<\eta$,
$\vec{y}_\alpha$ witnesses the choice of $\varphi_\alpha$.
Note that $\left<\vec{y}_\alpha\right>_{\alpha<\eta}\in\Lll$.
By the regularity of $\xi$ in $\Lll[j]$ (and upward absoluteness of $\Sigma_1$ 
truth)
there is $\gamma<\xi$ such that 
$L_\gamma(V_{\lambda+1},k)\sats\varphi_\alpha(x_\alpha,\vec{y}_\alpha)$ for each $\alpha<\eta$.
But then from the parameter 
$\left<\varphi_\alpha,\vec{y}_\alpha\right>_{\alpha<\eta}$,
we can recover $\left<x_\alpha\right>_{\alpha<\eta}$
in a $\Sigma_1$ fashion over $L_\xi(V_{\lambda+1},k)$,
by looking for such an ordinal $\gamma$ and objects in 
$L_\gamma(V_{\lambda+1},k)$. Therefore $\left<x_\alpha\right>_{\alpha<\eta}\in 
H$.

Now $H\elem_{\Sigma_1} L_\xi(V_{\lambda+1},k)$, since
\[ 
L_\xi(V_{\lambda+1},k)=\Hull_{\Sigma_1}^{L_\xi(V_{\lambda+1},k)}(V_{\lambda+1}
\cup\xi\cup\{k\}).\]
In particular, $H$ is extensional.
Let $C$ be the transitive collapse of $H$ and $\pi:C\to H$ the uncollapse map.
So $\Lll[j]\sats$``$C$ is closed under $\eta$-sequences''.
Let $\pi(\bar{\mathscr{T}})=\mathscr{T}$. Then since $\crit(\pi)>\lambda$,
$\bar{\mathscr{T}}$ is also a tree of height $\leq\eta$.
So by $\eta$-$\DC$ in $V$, there is a maximal branch $b$ through $\mathscr{T}$. But in $\Lll$, there is a surjection $\sigma:V_{\lambda+1}\to C$,
which implies $b\in\Lll\sub\Lll[j]$, and since $\Lll[j]\sats$``$C$ is closed
under $\eta$-sequences'', $b\in C\sats$``$b$ is maximal 
through $\bar{\mathscr{T}}$'', and it follows that $\pi(b)$ is a maximal branch
through $\mathscr{T}$ in $\Lll[j]$, as desired.

The proof for $\Lll$ instead of $\Lll[j]$ is essentially the same, but 
slightly easier.
\end{proof}

\begin{tm}\label{tm:j_iterable}Let $(\Lll,\lambda,j)$ be relevant, with $j$ proper. Then $\Lll[j]$ is a tight extension of $\Lll$, and  $(\Lll,j)$ is iterable in the sense of both \ref{dfn:iterates} and \ref{dfn:external_iterates}, the latter with respect to $N=\Lll[j]$.\footnote{
In \cite{woodin_sem2},
Woodin proves iterability assuming $\ZFC$ in $V$.}
\end{tm}
\begin{proof}
 We may assume $V_{\lambda+2}\subseteq\Lll$, since we may work inside $\Lll[j]$.

Let us first show that the ultrapower $\Ult_0(\Lll[j],\mu_j)$ 
 is
 wellfounded.
Suppose otherwise.
Let
$\eta\in\OR$ be large, in particular with $\Ult(\Lll[j]|\eta,\mu_j)$ illfounded, and consider
$\Hull^{(\Lll[j]|\eta,j\rest\eta)}(V_{\lambda+1})$.
This hull is fully elementary, and (by choice of $\eta$)
$j\rest\eta$ determines $j\rest L_\eta(V_{\lambda+1})$ and hence $\mu_j$.
Let $C$ be the transitive collapse of the hull. Then 
the illfoundedness reflects into $\Ult_0(C,\mu_j\cap C)$.
(Here we don't need to construct a specific $\in$-descending sequence;
the illfounded part of the ultrapower is computed in a first order fashion,
so
by 
elementarity, functions representing objects in the illfounded part
of $\Ult_0(\Lll[j]|\eta,\mu_j)$ also represent objects in the illfounded part
of $\Ult_0(C,\mu_j\cap C)$.)
As $C$  is the surjective image of $V_{\lambda+1}$, 
we get
$C\in\Lll|\Theta$. But then $\Ult_0(C,\mu_j)$ gets absorbed into $j(C)$, which 
is transitive, a contradiction.

We now verify that
the same ultrapower satisfies \L o\'{s}'s Theorem; the proof is rather similar to that of wellfoundedness.
Fix a formula $\varphi$
and $\eta$ as above and an $f\in\Lll[j]|\eta$ such that
\begin{equation}\label{eqn:Los_thm_hypo} \big[\Lll[j]|\eta\sats\exists x\ \varphi(x,f(u))
\big]\text{ for }\mu_j\text{-almost all }u\in V_{\eta+1}.
\end{equation}
 Let $C$ be as before and $\pi:C\to \Lll[j]|\eta$ the uncollapse map,
 which is elementary.
Let $\pi(\bar{f})=f$. We have  $\pi\rest V_{\lambda+1}=\id$,
and $C\in L_\Theta(V_{\lambda+1})$.
Since  $j\rest L_\Theta(V_{\lambda+1})\in\mathscr{E}(L_\Theta(V_{\lambda+1}))$ 
we can fix $y\in j(C)$ such that
\[ j(C)\sats\varphi(y,j(\bar{f})(j\rest V_\lambda))\]
and fix $g\in L_\Theta(V_{\lambda+1})$
with $g:V_{\lambda+1}\to C$
and $[g]^{\Lll}_{\mu_j}=y$, and  so by \L o\'{s}'s Theorem, 
\[ \big[C\sats\varphi(g(u),\bar{f}(u))\big]
\text{ for }\mu_j\text{-almost all }u\in V_{\lambda+1}.\]
But now working in $\Lll[j]$, let
$g':V_{\lambda+1}\to\Lll[j]|\eta$
be $g'(u)=\pi(g(u))$, and note that
$g'$ witnesses \L o\'s's Theorem
with respect to line (\ref{eqn:Los_thm_hypo}).

So $\Ult_0(\Lll[j],\mu_j)$ is wellfounded and satisfies \L o\'{s}'s Theorem. But then 
since $\mu_j\in \Lll[j]$, the usual 
argument (like that for linearly iterating a normal measure under $\ZFC$) now 
shows that all iterates of $(\Ll[j],\mu_j)$ are wellfounded and the ultrapowers 
satisfy \L o\'{s}'s Theorem, so $(\Ll[j],\mu_j)$ is iterable.

Let $\Lll'[j']=\Ult_0(\Lll[j],\mu_j)$ and $\widehat{j}:\Lll[j]\to\Lll'[j']$ be 
the 
ultrapower map.
Note that $\Lll[j]$ and $\Lll$ have the same functions
 $f:V_{\lambda+1}\to L_\Theta(V_{\lambda+1})$.
So the two ultrapowers $\Ult_0(\Lll,\mu_j)$
and $\Ult_0(\Lll[j],\mu_j)$ agree on the part computed 
with such functions,
and in particular 
\[ L_\Theta(V_{\lambda+1})=j(L_\Theta(V_{\lambda+1}))=\widehat{j}(L_\Theta(V_{
\lambda+1}))\]
and $j\rest L_\Theta(V_{\lambda+1})=\widehat{j}\rest L_\Theta(V_{\lambda+1})$.
We have
$\widehat{j}(\Lll)=\bigcup_{\alpha\in\OR}\widehat{j}(L_\alpha(V_{\lambda+1
} ))$.
Note that $\widehat{j}(\Lll)=\Lll$.
Let
$k:\Ult_0(\Lll,\mu_j)\to\widehat{j}(\Lll)$
be the natural factor map; that 
is,
\[ k([f]^{\Lll,0}_{\mu_j})=[f]^{\Lll[j],0}_{\mu_j}=\widehat{j}(f)(j\rest 
V_\lambda).\]
Then (by \L o\'{s}'s Theorem) $k$ is fully elementary,
and by the agreement of the ultrapower constructions regarding
functions mapping into
$L_\Theta(V_{\lambda+1})$,
we have $\crit(k)\geq\Theta$. But $k$ is computed in $\Lll[j]$,
where $V_{\lambda+1}^\#$ does not exist.
Therefore $k=\id$. So all functions $f\in\Lll[j]$
mapping into $\Lll$ are represented
mod $\mu_j$-measure one by functions $g\in\Lll$,
so $\Lll[j]$ is a tight extension of $\Lll$.\footnote{
An analogous thing holds for many direct limit 
systems of mice, a fact due independently to the author
and to Ralf Schindler, via the argument in \cite{Theta_Woodin_in_HOD}
or in \cite{vm1}. In the current context,
we can use this simpler ``non-existence of $V_{\lambda+1}^\#$''
argument instead, which was also observed independently
by Goldberg.} By Lemma \ref{lem:tight_extension},
 $(\Lll,\mu_j)$ is iterable (including \L o\'{s}'s Theorem for the 
ultrapowers determined by the iterates etc),
and the rest of the theorem follows.
\end{proof}

\section{Embedding extensions and the cofinality of 
$\Theta$}\label{sec:embeddings}

Having proved the main theorem, we now establish some more features
of the model $\Lll[j]$. We begin with an examination of existence and 
uniqueness of extensions of 
elementary embeddings, and the relationship of these issues to $\cof(\Theta)$,
in the case that $V_{\lambda+1}\sub\Lll$;
however, we don't assume $I_{0,\lambda}$ in general. 
So in this section we will consider models of form $\Lll=L(V_{\lambda+1})$, sometimes under the added assumption that $\lambda$ is even and $I_{0,\lambda}$ holds.
We carry on using notation as in the previous section. 
First an observation which follows easily from the calculations
with ultrapowers in the previous section. 

\begin{lem}\label{lem:Lll_no_Sigma_1-elem_k}
Let $\lambda$ be an even ordinal, $\Lll=L(V_{\lambda+1})$ and $\Theta=\Theta_{V_{\lambda+1}}^{\Lll}$. Then in $\Lll$, there is  no
$\Sigma_1$-elementary
$k:L_\Theta(V_{\lambda+1})\to L_\Theta(V_{\lambda+1})$.
\end{lem}
\begin{proof}
We may assume $V=\Lll$. Let $\mu_k$ be the measure 
derived from $k$ with seed $k\rest 
V_\lambda$.
Then $\Ult(V,\mu_k)$ is wellfounded and satisfies \L o\'{s}'s Theorem as in the proof  of Theorem \ref{tm:j_iterable}. But then $V=L(V_{\lambda+1})=\Ult(V,\mu_k)$,
and the ultrapower map is definable from the parameter $\mu_k$,
contradicting Suzuki \cite{suzuki_no_def_j}.
\end{proof}

\begin{dfn}
Let $\lambda$ be  even, $\Lll=L(V_{\lambda+1})$ and $\Theta=\Theta_{V_{\lambda+1}}^{\Lll}$. Work in $\Lll$. Let 
$j\in\mathscr{E}_{\nt}(V_\lambda)$.
 Then $T_{j}$ denotes the partial order consisting of bounded approximations to an
 $\ell\in\mathscr{E}_1(L_\Theta(V_{\lambda+1}))$ extending $j$.
The conditions of $T_j$ are elementary maps 
\[ \ell':L_{\alpha+1}(V_{\lambda+1})\to L_{\beta+1}(V_{\lambda+1}) \]
with $j\sub\ell'$,
for some 
ordinals $\alpha,\beta<\Theta$,
ordered by $\ell_1'\leq\ell_0'$ iff $\ell_0'\sub\ell_1'$.

Let $T_{j,\alpha}$ be the $\alpha$th derivative of $T_{j}$ given by 
iteratively removing
nodes $t$ for which there is a bound on the domains of extensions of $t$.
That is, (i) $T_{j,0}=T_j$,
 (ii) given $T_{j,\alpha}$, then
 $T_{j,\alpha+1}$
 is the set of all $t\in T_{j,\alpha}$
 such that for all $\beta<\Theta$
  there is $s\in T_{j,\alpha}$ with $t\sub s$ and
$\beta\sub\dom(s)$,
and (iii) $T_{j,\eta}=\bigcap_{\alpha<\eta}T_{j,\alpha}$ for limit $\eta$.
Let $T_{j,\infty}=\bigcap_{\alpha\in\OR}T_{j,\alpha}$.

Now work in $V$. Say a \emph{cofinal branch through $T_j$}
is a set $b\sub T$ which is linearly ordered by $\leq$, with $\Theta\sub\dom(\bigcup b)$.
And \emph{cofinal branch through $T_{j,\infty}$}
is likewise, but with $b\sub T_{j,\infty}$.
\end{dfn}

\begin{rem}\label{rem:branches_and_embeddings}
 Recall by Proposition \ref{prop:extensions},
 maps $k\in\mathscr{E}_1(V_{\lambda+2}^{\Lll})$ extending $j$ are in one-to-one correspondence with
 maps $\ell\in\mathscr{E}_1(L_\Theta(V_{\lambda+1}^{\Lll}))$, with $k$ corresponding to $\ell$ just when $k\sub \ell$.
 Also by \ref{prop:extensions}, any such $k,\ell$ are  in fact fully elementary.
 
 Note for any $b$, $b$ is a cofinal branch through $T_j$ iff $b$ is  cofinal through $T_{j,\infty}$.
 
Let $b$ be cofinal  through $T_{j,\infty}$
and $\ell=\bigcup b$. Then note that $\ell\in\mathscr{E}_1(L_\Theta(V_{\lambda+1}^{\Lll}))$ 
and $j\sub\ell$. Conversely,
if $j\sub\ell\in\mathscr{E}_1(L_\Theta(V_{\lambda+1}^{\Lll}))$, then there is a unique cofinal branch $b$ with $\bigcup b=\ell$,
and in particular, $T_{j,\infty}\neq\emptyset$.
\end{rem}

\begin{lem}\label{lem:embedding_descriptions}
 Adopt notation as above. Working in $V$, we have:
 \begin{enumerate}[label=\arabic*.,ref=\arabic*]
 \item\label{item:Lll_no_branch_thru_T} $\Lll$ has  no cofinal branch through $T_j$, nor $T_{j,\infty}$.
  \item\label{item:T_j_infty_perfect} If $T_{j,\infty}\neq\emptyset$
 then for every $t\in T_{j,\infty}$ and every $\alpha<\Theta$,
 there are $s_1,s_2\in T_{j,\infty}$ with $t\sub s_1,s_2$ and 
$\alpha\leq\dom(s_1),\dom(s_2)$ and $s_1\not\sub s_2\not\sub s_1$.
 \item\label{item:when_cof(Theta)=om} If $T_{j,\infty}\neq\emptyset$
and $\cof(\Theta)=\om$ then there are at least $\Theta^\om$-many pairwise distinct $k\in\mathscr{E}_1(V_{\lambda+2}^\Lll)$
extending $j$; this holds in particular if $V_{\lambda+1}^\#$ exists.
\end{enumerate}
\end{lem}

\begin{proof}
Part \ref{item:Lll_no_branch_thru_T}:
By Remark \ref{rem:branches_and_embeddings} and
 Lemma \ref{lem:Lll_no_Sigma_1-elem_k}.

Part \ref{item:T_j_infty_perfect}: Fix $t\in T_{j,\infty}$ and $\alpha<\Theta$. The existence of $s_1\in T_{j,\infty}$ with $\alpha\sub\dom(s_1)$
follows immediately from the definitions. The existence of such an  $(s_1,s_2)$ then
follows from the fact that $T_{j,\infty}\in\Lll$, but $\Lll$ has no cofinal branch
through $T_{j,\infty}$.

Part \ref{item:when_cof(Theta)=om}:
Fix a strictly increasing cofinal function $f:\om\to\Theta$
such that
\[ L_{f(n)+1}(V_{\lambda+1}^{\Lll})=\Hull^{L_{f(n)+1}(V_{\lambda+1}^{\Lll})}(V_{\lambda+1}^{\Lll})\]
for each $n<\om$. 
Let $T'$ be the tree of of attempts to build a function
$k:\rg(f)\to\Theta$
such that for each $n<\om$,
$j\cup (k\rest(f``(n+1)))$ extends to an elementary
\[ \ell:L_{f(n)+1}(V_{\lambda+1}^{\Lll})\to L_{k(f(n))+1}(V_{\lambda+1}^{\Lll}). \]
Note that $j\cup (k\rest(f``(n+1)))$  determines $\ell$
(using that $j$ determines $\ell\rest V_{\lambda+1}$ as the canonical extension).
The nodes of $T'$ should be functions with domain $f``m$ for some $m<\om$.
Note that infinite branches through $T'$ are in one-one correspondence with
$\Sigma_1$-elementary embeddings $L_\Theta(V_{\lambda+1}^{\Lll})\to L_\Theta(V_{\lambda+1}^{\Lll})$.

Let $T''$ be the set of nodes $t\in T'$ such that there is $s\in T_{j,\infty}$
with $t=s\rest\dom(t)$. Then note that $T''\neq\emptyset$ and $T''$ is perfect
(for every $t\in T''$ there are $s_1,s_2\in T''$
with $t\sub s_1, s_2$ but $s_1\not\sub s_2\not\sub s_1$);
otherwise,
letting $U=(T_{j,\infty})_t$
(the sub-tree of
 $T_{j,\infty}$ consisting
of all nodes extending $t$),
we have $U\in\Lll$, and it has
a unique cofinal branch $b$,
so $b\in\Lll$, contradicting part \ref{item:Lll_no_branch_thru_T}.

But the set of nodes in $T''$ is wellorderable.
 So $T''$ has $\Theta^\om$-many
infinite branches, because for every $t\in T''$
there are $\Theta$-many $s\in T''$ with $t\sub s$.
\end{proof}

\begin{cor}\label{cor:in_Lll[j]}
Let $(\Lll,\lambda,j)$ be relevant and $\Theta=\Theta^{\Lll}_{V_{\lambda+1}^{\Lll}}$.
Then: 
\begin{enumerate}[label=\arabic*.,ref=\arabic*]
\item\label{item:i_extends_to_at_most_one_k_again} Work in $\Lll[j]$. For every elementary $i:V_{\lambda}\to V_{\lambda}$,
there is at most one extension of $i$ to a $\Sigma_1$-elementary 
$k:V_{\lambda+2}^\Lll\to V_{\lambda+2}^\Lll$.
 \item\label{item:k:V_lambda+2_to_V_lambda+2} Work in $\Lll[j]$. Let $k:V_{\lambda+2}^\Lll\to 
V_{\lambda+2}^\Lll$ be $\Sigma_1$-elementary. Then:
\begin{enumerate}[label=\tu{(}\alph*\tu{)}]\item $k$ extends uniquely to a fully elementary $k':\Lll\to\Lll$,
\item there is a unique pair $(\widehat{k},\widetilde{j})$
such that
$\widehat{k}:\Lll[j]\to\Lll[\widetilde{j}]$ is
elementary and $k\sub\widehat{k}$; fixing this $(\widehat{k},\widetilde{j})$, we have
\item $k'\sub\widehat{k}$,
\item $k',\widehat{k}$ are the ultrapower maps induced by $\mu_k$ \tu{(}note $k'$ is proper\tu{)},
\item 
$\widehat{k}(j)=\bigcup\{k'(\bar{j})\bigm|\bar{j}\in\Lll\wedge\bar{j}\sub 
j\}$.
\end{enumerate}
\item\label{item:when_cof(Theta)=om_tm} Work in $V$ \tu{(}we have $\Lll[j]\sub V$\tu{)}. Then:
\begin{enumerate}[label=\tu{(}\alph*\tu{)}]\item $j$ is the unique extension of $j\rest V_\lambda^{\Lll}$ to a proper embedding $\Lll\to\Lll$,\item  if $\cof(\Theta)=\omega$ then
there are $\Theta^\om$-many distinct extensions of $j\rest V_\lambda^{\Lll}$ to a $\Sigma_1$-elementary $k:V_{\lambda+2}^{\Lll}\to V_{\lambda+2}^{\Lll}$.
\end{enumerate}
\end{enumerate}
\end{cor}
\begin{proof}
Recall that by \ref{tm:main}, $V_{\lambda+2}^{\Lll[j]}=V_{\lambda+2}^{\Lll}$ and $\Lll[j]\sats$``$V_{\lambda+1}^{\#}$ does not exist''.
Part \ref{item:i_extends_to_at_most_one_k_again}:
Since $\Theta$ is regular in $\Lll[j]$, this is by Proposition \ref{prop:Zipper_analogue}.
For part \ref{item:k:V_lambda+2_to_V_lambda+2} argue  as 
in the proof of Theorem \ref{tm:j_iterable}. 
Part \ref{item:when_cof(Theta)=om}: The first clause
is by  Proposition \ref{prop:extensions},
the second by Lemma \ref{lem:embedding_descriptions} part \ref{item:when_cof(Theta)=om_tm}.
\end{proof}

The next and final result in this section contrasts with Corollary \ref{cor:in_Lll[j]} part \ref{item:when_cof(Theta)=om_tm},
considering the situation when $\cof(\Theta)>\om$:

\begin{tm}\label{tm:extensions_when_cof(Theta)>om} Assume ZF, let $\lambda$ be even, $\Lll=L(V_{\lambda+1})$ and $\Theta=\Theta_{V_{\lambda+1}}^{\Lll}$. Suppose $\cof(\Theta)>\om$.
Then:
\begin{enumerate}[label=\arabic*.,ref=\arabic*]
 \item\label{item:i_extends_to_at_most_one_k} Every $i\in\mathscr{E}(V_{\lambda})$ extends to at most one $k\in\mathscr{E}_1(L_\Theta(V_{\lambda+1}))$.
 \item\label{item:k_extends_to_exactly_one_j} 
 Every $k\in\mathscr{E}_1(L_\Theta(V_{\lambda+1}))$
 extends uniquely to some $j\in\mathscr{E}(\Lll)$.
\end{enumerate}
\end{tm}
\begin{proof}
By Remark \ref{rem:sharp_and_cof(Theta)}, $V_{\lambda+1}^{\#}$ does not exist.

Part \ref{item:i_extends_to_at_most_one_k}:
This is just by Theorem \ref{prop:Zipper_analogue}
part \ref{item:cof(Theta)=om}.

Part \ref{item:k_extends_to_exactly_one_j}:
Uniqueness is by Proposition \ref{prop:extensions}.
For existence we use the ultrapower map 
$j:\Lll\to \Ult_0(\Lll,\mu_k)$.
This ultrapower satisfies \L o\'{s}'s Theorem as in the proof of Theorem \ref{tm:j_iterable}, so we just need to see it is 
wellfounded. So suppose not. Let $\eta\in\OR$
be such that $\eta>\theta$ and $U=\Ult_0(M,\mu_k)$ is illfounded, where $M=\Lll|\eta$.
Then in fact, $\mathscr{O}=\Ult_{M}(\eta,\mu_k)$ is illfounded,
where this denotes the ultrapower of $\eta$
formed with functions $f:V_{\lambda+1}\to\eta$ where $f\in M$.
For we can map $U$
into $\mathscr{O}$ by sending $[f]$ to $[\rank\com 
f]$,
noting that if $[f]\in^U[g]$ then $[\rank\com f]<^{\mathscr{O}}[\rank\com g]$.
When necessary, we write $[f]^{\eta,M}_{\mu}$
for the element of $\mathscr{O}$ represented by $f$.
We write $\illfp(\mathscr{O})$ for the illfounded part of $\mathscr{O}$.
Note $\mathscr{O}$ and $\illfp(\mathscr{O})$
are computed in $V$, where we have $\mu_k$,
but we do not assume they are computed in $\mathscr{L}$.

We will now mimic the usual proof of linear iterability,
identifying a reasonable notion of ``least''
element of $\illfp(\mathscr{O})$, use this to identify a canonical descending 
sequence through $\mathscr{O}$, and then show that we can reflect this into the 
action of $k$ for a contradiction. Recall that we do not assume $\AC$ in $V$. The identification of the 
descending sequence will not require the assumption that $\cof(\Theta)>\om$;
this only comes up later in the reflection.

So let $(t_0,\xi_0)$ be the lexicographically
least pair $(t,\xi)$ such that $t$ is a term and $\xi<\eta$ and there is
$y\in V_{\lambda+1}$ such that $t^M(y,\xi)$ is some $f\in M$ such that
$f:V_{\lambda+1}\to\eta$ and $[f]\in\illfp(\mathscr{O})$.
(The minimality of 
$(t_0,\xi_0)$
is computed in $V$.) Let
\[ H_0=\Hull^M(V_{\lambda+1}\cup\{\xi_0\}),\]
so $H_0\elem M$. Let $C_0$ be the transitive collapse of $H_0$ 
and $\pi_0:C_0\to M$ the uncollapse.
So $H_0,C_0,\pi_0\in\Lll$,
and $C_0\in L_\Theta(V_{\lambda+1})$.
Let Let $\bar{\eta}_0=\OR^{C_0}$
and $\pi_0(\bar{\xi}_0)=\xi_0$. 
Note $\crit(\pi_0)=\Theta^{C_0}=\Theta\cap H_0$
and $\pi_0(\Theta^{C_0})=\Theta$.
Let $\mathscr{O}_0=\Ult_{C_0}(\bar{\eta}_0,\mu_k\cap C_0)$
(the notation is like that used in defining $\mathscr{O}$ above).
Since $\mu_k\cap C_0\in\Lll$, we have $\mathscr{O}_0\in\Lll$,
and note that $\mathscr{O}_0$ is wellfounded,
because we get an order-preserving map $\mathscr{O}_0\to\OR(k(C_0))$, defined 
 \[[f]^{\bar{\eta}_0,C_0}_{\mu_k\cap C_0}\mapsto k(f)(k\rest V_\lambda).\]
So we identify
$\mathscr{O}_0$ with its ordertype, so $\mathscr{O}_0\in\OR$. Let 
$\gamma_0$ be the least $\gamma<\mathscr{O}_0$ with
\[ \exists y\in 
V_{\lambda+1}\ \left(\gamma=[t_0^{C_0}(\bar{\xi}_0,y)]^{\bar{\eta}_0,C_0}_{\mu_k\cap C_0}
\text{ and }[t_0^{M}(\xi_0,y)]^{\eta,M}_{\mu_k}\in\illfp(\mathscr{O})\right).\]
Let $Y_0$ be the set of all such $y$.
So for any such parameters $y,y'$,
we have 
\[ [t_0^{C_0}(\bar{\xi}_0,y)]^{\bar{\eta}_0,C_0}_{\mu_k\cap 
C_0}=\gamma_0=[t_0^{C_0}(\bar{\xi}_0,y')]^{\bar{\eta}_0,C_0}_{\mu_k\cap C_0}\]
and
\[ 
[t_0^{M}(\xi_0,y)]^{\eta,M}_{\mu_k}=
[t_0^{M}(\xi_0,y')]^{\eta,M}_{\mu_k}\in\illfp(\mathscr{O}).\]
Now choose $(t_1,\xi_1)$ as for $(t_0,\xi_0)$,
except we also require that 
\[ 
[t^M(\xi,y)]^{\eta,M}_{\mu_k}<^{\mathscr{O}}[t_0^M(\xi_0,y')]^{\eta,M}_{\mu_k}
\text{ for any/all }y'\in Y_0.\]
Then define 
$H_1=\Hull^M(V_{\lambda+1}\cup\{\xi_0,\xi_1\})$,
and define $C_1,\pi_1,\mathscr{O}_1$ like before.
Let $\bar{\eta}_1=\OR^{C_1}$ and $\pi_1(\bar{\xi}^1_0,\bar{\xi}^1_1)=(\xi_0,\xi_1)$.
Let 
$\gamma_1$ be the least $\gamma<\mathscr{O}_1$ such that 
there is $y\in 
V_{\lambda+1}$ 
such that
\begin{itemize}
 \item[--]
$\gamma=[t_1^{C_1}(\bar{\xi}^1_1,y)]^{\bar{\eta}_1,C_1}_{\mu_k\cap 
C_1}$,
\item[--] $[t_1^{M}(\xi_1,y)]^{\eta,M}_{\mu_k}\in\illfp(\mathscr{O})$ and 
\item[--] $[t_1^M(\xi_1,y)]^{\eta,M}_{\mu_k}<[t_0^M(\xi_0,y')]^{\eta,M}_{\mu_k}
\text{ for any/all }y'\in Y_0$.
\end{itemize}
Let $Y_1$ be the set of all such $y$.
Proceed in this manner through all finite stages.
So letting
\[\gamma^{\mathscr{O}}_n= [t_n^M(\xi_n,y)]^{\eta,M}_{\mu_k}
 \text{ for any/all }y\in Y_n,
\]
we have $\gamma_{n+1}^{\mathscr{O}}<^{\mathscr{O}}\gamma_n^{\mathscr{O}}$
for all $n<\om$.

Now let $H=\Hull^M(V_{\lambda+1}\cup\{\xi_n\}_{n<\om})=\bigcup_{n<\om}H_n$.
We don't immediately seem to know that $H\in\Lll$. But still, $H\elem M$ 
and in $V$ we can form
the transitive collapse $C$, and let $\pi:C\to M$ be the uncollapse.
Then $C=L_\alpha(V_{\lambda+1})$
for some ordinal $\alpha$, so $C\in\Lll$.
And note  $\Theta^C=\sup_{n<\om}\Theta^{C_n}$,
and since $\Theta^{C_n}<\Theta$ for each $n<\om$, and $\cof(\Theta)>\om$, we have $\Theta^C<\Theta$
(this is the one key point in the proof of existence of $j$ where we need this assumption).
But
there are cofinally
many $\beta<\Theta$ such that $\Lll|\beta=\Hull^{\Lll|\beta}(V_{\lambda+1})$,  and since 
$C\sats$``$\Theta^C=\Theta_{V_{\lambda+1}}$'',
 there is no such $\beta\in[\Theta^C,\OR^C)$.
Therefore $\OR^C<\Theta$ and $C\in L_\Theta(V_{\lambda+1})$.

Now $\Ult_{C}(\OR^C,\mu_k\cap C)$ is illfounded by construction,
but like before, we can embed this ultrapower into
$\OR(k(C))$ in an order-preserving manner, a contradiction.
\end{proof}

Does $\ZF+I_{0,\lambda}+\cof(\Theta^{L(V_{\lambda+1})}_{V_{\lambda+1}})=\om$
imply  $V_{\lambda+1}^\#$ exists?
Similarly, does this theory together with $\AC$ imply
$V_{\lambda+1}^\#$ exists? (Having $\Theta$ singular does routinely imply  
$\pow(V_{\lambda+1})\not\sub L(V_{\lambda+1})$.) It seems one might arrange that $\cof(\Theta)=\om$
by forcing over $\Lll[j]$, without changing $V_{\lambda+1}$.

\section{$\Lll_\om[j_\om]=\HOD^{\Lll}(\Lll_\om)$ is a ground of 
$\Lll$}\label{sec:HOD}

Much work in inner model theory
has been oriented around the analysis
of the $\HOD$ of determinacy models in terms of hod mice
or strategy mice, such as
$M_\infty[\Sigma]$, for an
appropriate direct limit of mice $M_\infty$ and iteration strategy $\Sigma$.
(A related strategy mouse, $M_1[\Sigma]$,
was already mentioned in the sketch of the main proof, earlier in the paper.)
There has also been analogous analyses of the 
$\HOD$ of certain generic extensions of mice.
Here we will prove an analogue, showing that $\HOD^{\Lll}(M_\om)=M_\om[j_\om]$,
and hence $\HOD^{\Lll}\sub M_\om[j_\om]$
and $M_\om[j_\om]$ is closed under ordinal definability in $\Lll$.
The arguments are mostly straightforward adaptations of the usual ones to our 
context, and there is such an adaptation in
\cite{reinhardt_non-definability}. 

Throughout this section we deal with a relevant triple $(\Lll,\lambda,j)$ with $j$ proper,
and a proper class $\Gamma\sub\OR\cut(\lambda+1)$ of fixed points of $j$,
with $\Gamma=\mathscr{I}^{\Lll}$
if $(V_{\lambda+1}^{\Lll})^{\#}$ exists.
We also adopt the other associated notation introduced in \S\S\ref{sec:elementary_lambda+2},\ref{sec:main}.

Recall Claims \ref{clm:in_range_implies_stable}
and \ref{clm:j_om,om+om=*_over_Lll_om}
of the proof of Theorem \ref{tm:main},
and the map $x\mapsto x^*$ (with domain $\Lll_\om$) given just after
Claim \ref{clm:in_range_implies_stable}. We first extend these ideas a little further.

\begin{dfn}
Let $x\in\Lll$. We say that $x$ is \emph{eventually stable}
iff there is $n<\om$ such that $j_n(x)=x$ (and hence $j_m(x)=x$ for all 
$m\in[n,\om)$). We say that $x$ is \emph{hereditarily eventually
stable} (\emph{hes}) iff every $y\in\trancl(\{x\})$ is eventually stable.

We now extend the earlier $*$-map to all eventually stable $x$:
given such an $x$ and $s\in\vec{\OR}$ with $x\in H^s$, let
\[ x^*=\lim_{n\to\om}\pi^{j\rest 
V_\lambda^{\Lll}}_{sn,\om}(x)=\lim_{m\to\om}j_{m\om}(x),\]
so $x^*\in\Lll_\om$.
\end{dfn}

\begin{lem}
 $\Lll$ can define the notions \emph{eventually stable}, \emph{hes},
 and the map $x\mapsto x^*$ \tu{(}with domain the class of all eventually stable $x$\tu{)}, from the 
parameter $j_n\rest V_\lambda^{\Lll}$, uniformly in $n<\om$.
\end{lem}
\begin{proof}
 A set $x\in\Lll$ is eventually stable iff there is  
$(s,m)\in\mathscr{P}^{j_n\rest V_\lambda^{\Lll}}$
such that $x\in H^s$ and
$\pi^{j_n\rest V_\lambda^{\Lll}}_{sm,s\ell}(x)=x$
 for all $\ell\in[m,\om)$.
The map $x\mapsto x^*$ is then clearly definable.
\end{proof}

Recall $j_{\om,\om+\om}:\Lll_\om\to\Lll_{\om,\om+\om}$ is the iteration map.
\begin{dfn}
Define $\widehat{j}_{\om,\om+\om}:\Lll_\om[j_\om]\to\Lll_{\om+\om}[j_{\om+\om}]$
as the unique elementary extension of $j_{\om,\om+\om}$;
by Theorem \ref{tm:main}, this makes sense and is the iteration map.
\end{dfn}

\begin{lem}
Let $x\in\Lll$. Then $x\in \Lll_\om[j_\om]$ iff $x$ is hes. Moreover, if $x\in \Lll_\om[j_\om]$ then 
$\widehat{j}_{\om,\om+\om}(x)=x^*$.
\end{lem}
\begin{proof}
We first show that every element of $\Lll_\om[j_\om]$ is hes. Each element of $\Lll_\om$ is eventually stable, by Claim \ref{clm:in_range_implies_stable} of the proof of Theorem \ref{tm:main}.
And $j_\om\rest\Theta$ is eventually stable, by Claim \ref{clm:set_seg_of_*_suffices}
 of the same proof.
 Also by that claim, $\Lll_\om[j_\om]=\Lll_\om[j_\om\rest\Theta]$, so each $x\in\Lll_\om[j_\om]$
is definable over $\Lll$ from some $z\in\Lll_\om$ and $j_\om\rest\Theta$ and some $\alpha\in\OR$. It follows that each such
$x$ is eventually stable, and since $\Lll_\om[j_\om]$ is transitive, each such $x$ is in fact hes.

Let $x\in\Lll_\om[j_\om]$.
We now show that $x^*=\widehat{j}_{\om,\om+\om}(x)$.
Let $y\in V_{\lambda_\om+1}^{\Lll_\om}$ and $\eta\in\OR$
and $\varphi$ be a formula such that 
\begin{equation}\label{eqn:x_def} \Lll_\om[j_\om]\sats\text{``}x\text{ is the unique }x'\text{ such that 
}\varphi(x',y,\eta,j_\om\rest\Theta)\text{''}.\end{equation}
Applying $\widehat{j}_{\om,\om+\om}$ to this statement and letting $\widehat{x}=\widehat{j}_{\om,\om+\om}(x)$, we get
\begin{equation}\label{eqn:x-hat_def} \Lll_{\om+\om}[j_{\om+\om}]\sats
 \text{``}\widehat{x}\text{ is the unique }x'\text{ such that 
}\varphi(x',y^*,\eta^*,j_{\om+\om}\rest\Theta)\text{''};\end{equation}
this is because
$\widehat{j}_{\om,\om+\om}$
extends $j_{\om,\om+\om}$ elementarily, and since $j_{\om,\om+\om}(\eta,y)=(\eta^*,y^*)$ by Claim \ref{clm:j_om,om+om=*_over_Lll_om}
of the proof of \ref{tm:main}.

We want to use line (\ref{eqn:x-hat_def}) to see that
$\widehat{x}=x^*$.

Let $n<\om$ be such that $y\in\rg(j_{n\om})$
and $j_n(\eta)=\eta$. Then  $j_n(y)=y$
and $j_n({j_\om\rest\Theta})={j_\om\rest\Theta}$ by Claim \ref{clm:in_range_implies_stable}, so  
by line (\ref{eqn:x_def}), it follows that $j_n(x)=x$. And
letting $k=j_n\rest V_\lambda^{\Lll}$, we have ($\Lll=\Lll_n$ and)
\[ \Lll_n\sats\text{``}x=\text{  unique 
}x'\in(\widetilde{\Lll}_\om[\widetilde{j}_\om])^k\text{ s.t.~} 
(\widetilde{\Lll}_\om[\widetilde{j}_\om])^k\sats\varphi(x',y,\eta, 
(\widetilde{j}_\om)^k\rest\Theta)\text { '' } , \]
where ``s.t.'' means ``such that'' and $(\widetilde{\mathscr{L}}_\om)^k$ means 
the version of $\Lll_\om$ given by
iterating  $k$, and $(\widetilde{\mathscr{L}}_\om[\widetilde{j}_\om])^k$ 
likewise (as computed in $\Lll=\Lll_n$ here; so in this case
it happens that these are just $\Lll_\om$ and $\Lll_\om[j_\om]$,
but the point is that we are referring to the definitions of these classes
from the parameter $k$). So
letting $\ell=j_\om\rest V_{\lambda_\om}^{\Lll_\om}$, 
\[ \Lll_\om\sats\text{``}x^*=\text{  unique 
}x'\in(\widetilde{\Lll}_\om[\widetilde{j}_\om])^{\ell}\text{ s.t.~}
(\widetilde{\Lll}_\om[\widetilde{j}_\om])^{\ell}\sats\varphi(x',y^*,\eta^*,
(\widetilde{j}
_\om)^\ell\rest\Theta)\text{''}\]
(as $\Theta^*=\Theta$). But
$((\widetilde{\Lll_\om})^{\ell})^{\Lll_\om}=\Lll_{\om+\om}$
and $((\widetilde{j_\om})^{\ell})^{\Lll_\om}=j_{\om+\om}$, and so consulting line (\ref{eqn:x-hat_def}), it follows that $x^*=\widehat{x}$, as desired.

We now prove by induction on rank that every hes $x\in\Lll$
is in $\Lll_\om[j_\om]$. So fix a hes $x\in\Lll$ with $x\sub \Lll_\om[j_\om]$.
Then $x^*\in\Lll_\om\sub\Lll_\om[j_\om]$.
But $x\sub\Lll_\om[j_\om]$, and since $*$ agrees
with $k=\widehat{j}_{\om,\om+\om}$ over $\Lll_\om[j_\om]$,
$\Lll_\om[j_\om]$ can compute $x=\{y\bigm|k(y)\in x^*\}$,
so $x\in\Lll_\om[j_\om]$, as desired.
\end{proof}

In the following theorem, we write $\OD_X$ for
the class of all $x$ such that $x$ is definable
 from (finitely many) parameters in $\OR\cup X$;
note that we officially do not allow $X$ as a predicate here,
but in the case below, that will make no difference.
We then define $\HOD_X$ from $\OD_X$ just as $\HOD$ is defined from $\OD$.

\begin{tm}\label{tm:HOD^Lll_Lll_om} Let $(\Lll,\lambda,j)$ be relevant with $j$ proper, and adopt the usual notation. Then $\HOD^{\Lll}_{\Lll_\om}=\HOD^{\Lll}_{\Lll_\om[j_\om]}=
\Lll_\om[j_\om]$.\end{tm}
\begin{proof}
We first show that $\HOD^{\Lll}_{\Lll_\om[j_\om]}\sub\Lll_\om[j_\om]$.
Proceeding 
by induction on rank, let $A\in\HOD^{\Lll}_{\Lll_\om[j_\om]}$
be such that $A\sub\Lll_\om[j_\om]$;
it suffices to see that $A\in\Lll_\om[j_\om]$.
Let $p\in\Lll_\om[j_\om]$ and $\varphi$ be a formula such that
for $x\in\Lll_\om[j_\om]$, we have
\[ x\in A\iff \Lll\sats\varphi(p,x).\]
Then since $p$ and all such $x$ are hes, we get
\[ x\in A\iff 
\Lll_\om\sats\varphi(p^*,x^*)\iff\Lll_\om\sats\varphi(\widehat{j}_{\om,\om+\om}
(p),\widehat{j}_{
\om,\om+\om}(x)), \]
so $A\in\Lll_\om[j_\om]$.

So to prove the theorem, it just suffices
to see that $j_\om\rest\Theta$ is definable over $\Lll$
from the parameter $(V_{\lambda_\om+1}^{\Lll_\om},j_\om\rest 
V^{\Lll_\om}_{\lambda_\om})$.
But in $\Lll$,
$j_\om\rest L_\Theta(V_{\lambda_\om+1}^{\Lll_\om})$ is the unique 
$k\in\mathscr{E}_1(L_\Theta(V_{\lambda_\om+1}^{\Lll_\om}))$
extending $j_\om\rest V_{\lambda_\om+1}^{\Lll_\om}$,
by Proposition \ref{prop:Zipper_analogue}, and in particular, part \ref{item:cof(Theta)=om} of that proposition.
\end{proof}

Using the preceding result, we next observe that $\Lll_\om[j_\om]$
is a (set-)ground for $\Lll$, via Vopenka 
forcing
$\Vop$ for adding $j_{0\om}``V_\lambda^{\Lll}$ generically to $\Lll_\om[j_\om]$.
Moreover, we will analyze aspects of this forcing.

\begin{dfn}\label{dfn:Vop}
 Work in $\Lll$. Let $\widetilde{\Vop}$ be the partial order whose elements are 
the non-empty sets $A\sub\pow(V_{\lambda_\om}^{\Lll_\om})$
which are $\OD_{\Lll_\om}$, with ordering $\widetilde{\leq}$ being $\sub$. (Note here that
 $A\in\Lll$, so $A\sub\Lll$ (as we are working in $\Lll$ for the current definition), but it is not required that  $A\sub\Lll_\om$.)
Let $\Vop'$ be the following partial order.
The conditions are tuples $p=(a,\gamma,\eta,\varphi)$ with 
$a\in V_{\lambda_\om+1}^{\Lll_\om}$, $\gamma<\eta\in\OR$
with $\lambda_\om+1<\eta$,
and $\varphi$ a formula. Given such a tuple $p$, let
\[ A_p=A_{(a,\gamma,\eta,\varphi)}=\{x\sub 
V_{\lambda_\om}^{\Ll_\om}\bigm|\Lll|\eta\sats
\varphi(x,a,\gamma,V_{\lambda_\om+1}^{\Lll_\om})\},\]
so $A_p\in\widetilde{\Vop}$.
The ordering of $\Vop'$
is then $p\leq' q$ iff 
$A_p\sub A_q$.
So if $\Vop'/\equiv$ is isomorphic to $\widetilde{\Vop}$,
where $\equiv$ is the usual equivalence relation
($p\equiv q$ iff $p\leq' 
q\leq' p$).
Finally let $\Vop$ be the restriction of $\Vop'$
to tuples $p\in L_\Theta(V_{\lambda_\om+1}^{\Lll_\om})$,
and $\leq^{\Vop}$ its order.
\end{dfn}

\begin{lem}\label{lem:Vop_props} Work in $\Lll$. We have:
\begin{enumerate}[label=\arabic*.,ref=\arabic*]
 \item\label{item:Vop_suffices} 
 $\Vop/\equiv$ is isomorphic to $\widetilde{\Vop}$.
\item\label{item:Vop_def} $\Vop$ is definable from parameters over 
$(L_\Theta(V_{\lambda_\om+1}^{\Lll_\om}),j_\om\rest\Theta)$.
\item\label{item:conditions_for_y} The set of pairs $(y,p)\in 
V_{\lambda_\om}^{\Lll_\om}\cross\Vop$
such that $y\in x$ for every $x\in A_p$, is definable over 
$(L_\Theta(V_{\lambda_\om+1}^{\Lll_\om}),j_\om\rest\Theta)$.
\item\label{item:Vop_cc} Let  $A$ be an antichain of $\Vop$.
Then $A\in L_\alpha(V_{\lambda+1})$ for some $\alpha<\Theta$,
and if $A\in\Lll_\om[j_\om]$
then 
$A\in L_\alpha(V_{\lambda_\om+1}^{\Lll_\om})$ for some $\alpha<\Theta$.
\end{enumerate}
\end{lem}
\begin{proof}
 Part \ref{item:Vop_suffices}: Given any $p\in\Vop'$, we can take a hull to 
replace
 $p$ with an equivalent condition $q$ below $\Theta$,
 and note then that $q\in L_\Theta(V_{\lambda_\om+1}^{\Lll_\om})$.
 (Note here that each subset of $V_{\lambda_\om}^{\Lll_\om}$
 is coded by an element of $V_{\lambda+1}^{\Lll}$.)
 
Part \ref{item:Vop_def}: Working over 
$L_\Theta(V_{\lambda_\om+1}^{\Lll_\om})$,
 from 
$j_\om\rest V_{\lambda_\om}^{\Lll_\om}$ and the predicate 
$j_\om\rest\Theta$,
 we can compute $j_\om\rest L_\Theta(V_{\lambda_\om+1}^{\Lll_\om})$,
 and hence 
$j_{\om,\om+\om}\rest L_\Theta(V_{\lambda_\om+1}^{\Lll_\om})$,
by iterating the former embedding. Now for $p,p_1\in\Vop$
with $p=(a,\gamma,\eta,\varphi)$
and $p_1=(a_1,\gamma_1,\eta_1,\varphi_1)$, the following are equivalent:
\begin{enumerate}\item[--]
$p\leq^{\Vop}p_1$, \item[--] $A_{p}\sub A_{p_1}$, \item[--]
 $L_\Theta(V_{\lambda+1}^{\Lll})\sats$``for all $x\sub 
V_{\lambda_\om}^{\Lll_\om}$,
\[ L_\eta(V_{\lambda+1})\sats
\varphi(x,a,\gamma,V_{\lambda_\om+1}^{\Lll_\om})\implies
L_{\eta_1}(V_{\lambda+1})\sats
\varphi_1(x,a_1,\gamma_1,V_{\lambda_\om+1}^{\Lll_\om})\text{''},\]
\item[--]  $L_\Theta(V_{\lambda_\om+1}^{\Lll_\om})\sats$``for all $x\sub 
(V_{\lambda_\om}^{\Lll_\om})^*$,
\[L_{\eta^*}(V_{\lambda_\om+1})\sats
\varphi(x,a^*,\gamma^*,(V_{\lambda_\om+1}^{\Lll_\om})^*)\implies
L_{\eta_1^*}(V_{\lambda_\om+1})\sats
\varphi_1(x,a_1^*,\gamma_1^*,(V_{\lambda_\om+1}^{\Lll_\om})^*)\text{''}\]
(recalling that we extended $*$ to $\Lll_\om$, in fact to $\Lll_\om[j_\om]$, 
above).
\end{enumerate}
Since $*$ 
agrees with $j_{\om,\om+\om}$, this gives the claimed definability of 
$\leq^{\Vop}$, assuming the set of conditions of $\Vop$ is likewise definable. But note that $\Vop\sub L_\Theta(V_{\lambda_\om+1}^{\Lll_\om})$, and that for all $p=(a,\gamma,\eta,\varphi)\in L_\Theta(V_{\lambda_\om+1}^{\Lll_\om})$ of the right form,
we have $p\in\Vop$ iff $A_p\neq\emptyset$, and the requirement that $A_p\neq\emptyset$
can be defined in a manner similar to that for $\leq^{\Vop}$ above.

Part \ref{item:conditions_for_y}: This is much like the previous part.

Part \ref{item:Vop_cc}: Note that if $A\in\Lll$ is an antichain of $\Vop$
then for all $p,q\in A$ with $p\neq q$ we have $A_p\cap A_q=\emptyset$.
But then if the elements of $A$ are constructed unboundedly in 
$L_\Theta(V_{\lambda_\om+1}^{\Lll_\om})$,
then $\Lll$ would have a surjection from $V_{\lambda+1}^{\Lll}$ to $\Theta$,
a contradiction. This yields both clauses.
\end{proof}

\begin{tm} Let $(\Lll,\lambda,j)$ be relevant with $j$ proper, and adopt the usual notation. Then
 $\Lll_\om[j_\om]$ is a  ground of $\Lll$,
 via the forcing $\PP=\Vop$ \tu{(}as in \ref{dfn:Vop}\tu{)}.\footnote{Goldberg has also 
communicated to 
the author that $\Lll_\om[j_\om]$
is also a ground of $\Lll$ via a certain kind of Prikry forcing.} Moreover:
 \begin{enumerate}[label=\arabic*.,ref=\arabic*]
  \item\label{item:Vop_def_2}  $\PP$ is definable from parameters over  the structure
$(L_\Theta(V_{\lambda_\om+1}^{\Lll_\om}),j_\om\rest\Theta)$.
\item\label{item:Vop_antichains_in_Lll_om[j_om]} Every antichain of $\PP$ in $\Lll_\om[j_\om]$
is the surjective image of $V_{\lambda_\om+1}^{\Lll_\om}$ in 
$\Lll_\om[j_\om]$.
\item\label{item:Vop_antichains_in_Lll} Every antichain of $\PP$
in $\Lll$ is the surjective image of $V_{\lambda+1}^{\Lll}$ in $\Lll$.
\item\label{item:Vop_every_cond} There is $p_0\in\PP$ such that every  $p\leq p_0$ is contained
in an $\Lll_\om[j_\om]$-generic filter $G$ such that $\Lll_\om[j_\om][G]=\Lll$.
 \end{enumerate} 
\end{tm}
\begin{proof} Parts \ref{item:Vop_def_2},
\ref{item:Vop_antichains_in_Lll_om[j_om]} and \ref{item:Vop_antichains_in_Lll} are directly by Lemma \ref{lem:Vop_props}.

Now let $x\in\Lll$ with $x\sub V_{\lambda_\om}^{\Lll_\om}$. 
We can now proceed with the usual Vopenka argument  to show that $x$ induces 
filter $G_x$ which is $(\Lll_\om[j_\om],\Vop)$-generic,
with $x\in\Lll_\om[j_\om][G_x]$.
That is, let
\[ G_x=\{p\in\Vop\bigm|x\in A_p\}.\]
Then $G_x$ is easily a filter. For genericity, let $D\in\Lll_\om[j_\om]$, 
$D\sub\Vop$, be dense. Then by Theorem \ref{tm:HOD^Lll_Lll_om}, there is  $p\in\Vop$
such that $A_p=\bigcup_{q\in D}A_q$ (that is, by \ref{tm:HOD^Lll_Lll_om},
 $D$ is $\OD_{\Lll_\om}^{\Lll}$,
 so $\bigcup_{q\in D}A_q$ is in $\widetilde{\Vop}$, so we get a $p$ as desired in $\Vop$). So if $A_p\psub 
\pow(V_{\lambda_\om}^{\Lll_\om})\cap\Lll$ then there is also $p'\in\Vop$ with 
$A_{p'}=(\pow(V_{\lambda_\om}^{\Lll_\om})\cap\Lll)\cut A_p$. But then
note that $p'$ is incompatible with every $q\in D$, a contradiction.
So $A_p=\pow(V_{\lambda_\om}^{\Lll_\om})\cap\Lll$, so
there is $q\in D$ with $x\in A_q$, so $G_x\cap D\neq\emptyset$,
as desired.

To see that $x\in\Lll_\om[j_\om][G_x]$,
use part \ref{item:conditions_for_y} of Lemma \ref{lem:Vop_props}.

Now applying this with $x=j_{0\om}``V_\lambda^{\Lll}$,
note that in $\Lll_\om[j_\om][G_x]$,
we can compute $V_\lambda^{\Lll}$ and $j_{0\om}\rest V_\lambda^{\Lll}$,
and from these data, given any $X\sub V_\lambda^{\Lll}$,
since $j_{0\om}(X)\in\Lll_\om$, we can then invert and compute $X$.
So $\Lll_\om[j_\om][G_x]=\Lll$.

It just remains to verify part \ref{item:Vop_every_cond}.
Let $\Vop_{\mathrm{hom}}$ be the restriction of $\Vop$
below the condition $p_0$, where
\[ A_{p_0}=\big\{x\in\pow(V_{\lambda_\om}^{\Lll_\om})\cap\Lll\bigm|V_{\lambda+1}^{\Lll}\sub 
L_\Theta(V_{\lambda_\om+1}^{\Lll_\om},x)\big\}, \]
noting that $A_{p_0}$ is appropriately definable in $\Lll$.
Then by the foregoing discussion,
for every $p\in\Vop_{\mathrm{hom}}$
 there is easily $x\in\pow(V_{\lambda_\om}^{\Lll_\om})\cap\Lll$
 such that $p\in G_x$ and $\Lll_\om[j_\om][G_x]=\Lll$
 (in fact, any $x\in A_p$ works).
\end{proof}

\begin{rem}
It follows that $\Lll_\om[j_\om]$ can  compute
the theory of $\Lll$ in parameters in $\Lll_\om[j_\om]$,
by consulting what is forced  by $\Vop_{\mathrm{hom}}$.
(But we had already observed that 
$\HOD^{\Lll}_{\Lll_\om[j_\om]}=\Lll_\om[j_\om]$.)

Forcing with the full $\Vop$, we might have sets  $x\in\Lll$ such that $x\sub 
V_{\lambda_\om}^{\Lll_\om}$ and $x\notin\Lll_\om[j_\om]$
but $\Lll_\om[j_\om][G_x]\neq\Lll$.
In this case we get 
\[ \Lll_\om[j_\om][G_x]=\HOD^{\Lll}_{\Lll_\om[j_\om]\cup\{x\}},
\]
by basically the usual calculations. (Note that every element
of the model on the right can be coded by a subset of $\Lll_\om[j_\om]$;
these all get into $\Lll_\om[j_\om][G_x]$ by the usual argument.)
\end{rem}
In contrast to the fact that $\Lll_\om[j_\om]=\HOD^{\Lll}_{\Lll_\om[j_\om]}$:

\begin{tm}
We have:
\begin{enumerate}[label=\arabic*.,ref=\arabic*]
 \item\label{item:HOD(L_om[j_om])} $\Lll_\om[j_\om]\psub 
\HOD^{\Lll[j]}_{\Lll_\om[j_\om]}$.
 \item\label{item:pow(Theta)_in_HOD} 
$\pow(\Theta)\cap\HOD^{\Lll[j]}\not\sub\Lll$.
\end{enumerate}

\end{tm}
\begin{proof}

Part \ref{item:HOD(L_om[j_om])}:  $\Lll[j]\sats$``$\Lll=L(V_{\lambda+1})$'', so by Theorem \ref{tm:HOD^Lll_Lll_om},
$\Lll_\om[j_\om]=\HOD^{\Lll}_{\Lll_\om[j_\om]}\sub\HOD^{\Lll[j]}_{\Lll_\om[j_\om]}$.
So suppose $\HOD^{\Lll[j]}_{\Lll_\om[j_\om]}=\Lll_\om[j_\om]$.
 This will contradict part \ref{item:pow(Theta)_in_HOD} once we have it,
 but before we proceed to that part, we sketch an alternate argument for 
contradiction.

Under the contradictory hypothesis,
more Vopenka forcing gives that $\Lll[j]$ is 
a set-generic
 extension of $\Lll_\om[j_\om]$. But $\Lll_\om[j_\om]$
 is the $\om$th iterate of $\Lll[j]$, so  we have the elementary embedding
 \[ j_{0\om}^+:\Lll[j]\to\Lll_\om[j_\om],\]
which is $\Sigma_5$-definable from parameters  over $\Lll[j]$.
But then we can use a standard argument for a contradiction:
Working in $\Lll_\om[j_\om]$, let $\kappa_0\in\OR$ be least such that
for some forcing $\PP$,
$\PP$ forces ``there is a $\Sigma_1$-elementary
embedding $k:V^{\PP}\to V$ which
is $\Sigma_5$-definable from parameters and $\crit(k)=\kappa_0$''.
Then $\kappa_0$ is outright definable over $\Lll_\om[j_\om]$,
but taking a witnessing generic and $k$,
then $k$ is actually fully elementary by Proposition 5.1 of \cite{kanamori},
so $\kappa_0\in\rg(k)$, so $\kappa_0\neq\crit(k)$,  contradiction.

Part \ref{item:pow(Theta)_in_HOD}: Work in $\Lll[j]$, so $V_{\lambda+1}=V_{\lambda+1}^{\Lll}$.
Let $X$ be the set of all $\alpha<\Theta$ such that
$k(\alpha)=\alpha$ for all $I_{0,\lambda}$-embeddings $k$.
We claim that $X$ witnesses the theorem.

First, the definition we have just given (inside $\Lll[j]$)
shows that $X\in\HOD^{\Lll[j]}$.

Now we claim that $X$ is unbounded in $\Theta$.\footnote{The referee suggested the following alternative argument for the unboundedness of $X$ in $\Theta$: note that for each $I_{0,\lambda}$-embedding $k$, the set $X_k$ of all $\alpha<\Theta$
which are fixed by $k$, is an $\om$-club in $\Theta$.
But there are only $V_{\lambda+1}$-many $k$ to consider,
since $k$ is  determined by $k\rest V_\lambda$.
It is then easy enough to see that $X=\bigcap_kX_k$ is also $\om$-club in $\Theta$, and in particular, is unbounded.}
For recall that since $V=\Lll[j]$, $V_{\lambda+1}^{\#}$ does not exist. Let
$\Lambda$ be the class of all ordinals which are fixed by all 
$I_{0,\lambda}$-embeddings $k$. Note that $\Lambda=\OR\cap\Hull^{\Lll}(\Lambda)$
and $X=\Lambda\cap\Theta$.
So we just need to see that $\Hull^{\Lll}(\Lambda)$ is unbounded in $\Theta$, so let $\alpha<\Theta$.
Note $\Hull^{\Lll}(V_{\lambda+1}\cup\Lambda)=\Lll$
by \ref{prop:extensions}, 
and $\lambda\in\Lambda$ by  \ref{prop:ultrapowers}. Let $s\in[\Lambda]^{<\om}$ and $n<\om$
be such that
\[ H=\Hull_{\Sigma_n}^{\Lll}(V_{\lambda+1}\cup\{s\}) \]
includes some $\beta\in(\alpha,\Theta)$.
Then $\beta<\sup(H\cap\Theta)<\Theta$ and $\sup(H\cap\Theta)\in\Hull^{\Lll}(\Lambda)$.
So $X$ is unbounded in $\Theta$ as claimed.

But $X\notin\Lll$, because if we had $X\in\Lll$,
then we easily get $j\rest L_\Theta(V_{\lambda+1})\in\Lll$, which is false.
That is, note there are unboundedly many $\alpha<\Theta$
with $\alpha\in X$ and
\[ L_\alpha(V_{\lambda+1})=\Hull^{L_\alpha(V_{\lambda+1})}(V_{\lambda+1}), \]
and for such $\alpha$, $j\rest L_\alpha(V_{\lambda+1})$ is determined by
$j\rest V_{\lambda+1}\in\Lll$ and  that $j(\alpha)=\alpha$.
\end{proof}

\section{Sharps at even and odd levels}\label{sec:sharps}

Since $\Lll[j]\sats$``$j:L(V_{\lambda+1})\to L(V_{\lambda+1})$
and $V_{\lambda+2}\sub L(V_{\lambda+1})$'' (with $\lambda$ even as before, etc),
we trivially have that $\Lll[j]\sats$``$j:L(V_{\lambda+2})\to 
L(V_{\lambda+2})$'' (but this does not hold with $\lambda+3$
replacing $\lambda+2$, since $j\rest\Theta$ is coded by an element of 
$V_{\lambda+3}$).  So we have the following fact:

\begin{tm}\label{tm:non_sharp_existence} If  
$\ZF+$``$\lambda$ is even and there is $j\in\mathscr{E}(L(V_{\lambda+2}))$  with $\crit(j)<\lambda$'' is consistent, then it
does not imply the existence of $V_{\lambda+1}^\#$.
\end{tm}

One might  ask, however, whether
there are natural classes of
sets $A\sub V_{\lambda+1}$ which do not code all elements of $V_{\lambda+1}$,
and which for which $\ZF+I_{0,\lambda}$ does prove the existence of $A^\#$.

We show now that at odd levels, the picture of Theorem \ref{tm:non_sharp_existence} completely reverses:

\begin{tm}\label{tm:sharp_existence}
Let $(\Lll,\lambda,j)$ be relevant
with $j$ proper. 
 Then for every $A\in V_{\lambda+1}^{\Lll}$,
 $A^\#$ exists and $A^\#\in\Lll$.
\end{tm}

It was pointed out by the anonymous referee that Richard Laver states in \cite{laver_implications} (p.~80) that Woodin proved in ZFC that if there is an elementary $j:V_{\lambda+1}\to V_{\lambda+1}$
(so $\lambda$ is a limit ordinal) then $A^\#$ exists for every $A\in V_{\lambda+1}$ (he cites
handwritten notes by Solovay for the result); this suggests that the hypothesis for the above result may not be optimal.
Goldberg \cite{goldberg_even_ordinals_v3}
has other sharp existence results, including   Corollary 6.11 of  \cite{goldberg_even_ordinals_v3},
by which if $\lambda$ is even and $k:V_{\lambda+3}\to V_{\lambda+3}$
is $\Sigma_1$-elementary then $A^\#$ exists for every $A\in V_{\lambda+2}$. (Note that by \ref{tm:sharp_existence}, if in fact $j:L(V_{\lambda+3})\to L(V_{\lambda+3})$
 then $A^\#$ exists for every $A\in 
V_{\lambda+3}$.)

\begin{proof}
By working in $\Lll[j]$,
we may and do assume that $V_{\lambda+2}\sub\Lll$.
We will first prove
that $j(A)^\#$ exists, and then
deduce from this that $A^\#$ exists.

If $\lambda$ is a limit, let $\mu$ be the $V_\lambda$-extender
derived from $j$, and otherwise let $\mu$ be the measure derived from $j$ with 
seed $j\rest V_{\lambda-2}$;
 so in any case, $\mu\in V_{\lambda+1}$.
 By \cite{cumulative_periodicity_pub_online_first}, we 
have $\Ult_0(V_{\lambda},\mu)=V_{\lambda}$ and $j\rest V_{\lambda}$
is the ultrapower map, and $\Ult_0(V_{\lambda+1},\mu)\psub V_{\lambda+1}$,
but $j\rest V_{\lambda+1}$ is still the ultrapower map
associated to the latter ultrapower.
Therefore
\[ \Ult_0((V_{\lambda},A),\mu)=(V_{\lambda},j(A)).\]

 Let $M=L(V_{\lambda},A)$. We form the internal ultrapower
 $\Ult_0(M,\mu)$; recall that this means that we only use functions in $M$ to form the 
ultrapower.\footnote{This convention is also applying for the ultrapowers in the preceding paragraph.
In case $\lambda$ is a limit, this is needed to ensure that $\Ult_0(V_\lambda,\mu)=V_\lambda$;
otherwise, depending on $\cof(\lambda)$,
we could get $V_\lambda\psub\Ult(V_\lambda,\mu)$
(where all functions are used to form the ultrapower). In case $\lambda$ is a successor, it makes no difference, but note that we use conventions here as in \cite{cumulative_periodicity_pub_online_first},
by which \emph{all} functions $f:V_{\lambda-1}\to V_\lambda$ are (coded by) elements of $V_\lambda$,
so in this case it makes no difference whether we use $\Ult_0$ or $\Ult$.}

\begin{clm*}$\Ult_0(M,\mu)$ satisfies \L o\'{s}'s Theorem.\end{clm*}
\begin{proof}
We consider first
the case that $\lambda$ is a successor. Let $f\in M$
with $f:\mathscr{E}(V_{\lambda-2})\to M$ and suppose
\[ \all^*_\mu k\ \big[M\sats\exists x\ \varphi(x,f(k))\big]. \]
Then by the elementarity of $j$,
we can fix $x\in j(M)$ such that
\[ j(M)=L(V_{\lambda},j(A))\sats\varphi(x,j(f)(j\rest V_{\lambda-2})).\]
Let $y\in V_{\lambda}$ and $t$ be a term be such that
\[ j(M)\sats x=t(y,j(f)(j\rest V_{\lambda-2}),j(A)).\]
The function $f_y:\mathscr{E}(V_{\lambda-2})\to V_{\lambda}$
defined
$f_y(k)=(k^+)^{-1}``y$
(where $k^+$ is the canonical extension of $k$) is such that $f_y\in 
V_{\lambda}$ and $[f_y]^{V_{\lambda}}_\mu=y$.
Now define the function $g\in M$, $g:\mathscr{E}(V_{\lambda-2})\to M$, by
$g(k)=t(f_y(k),f(k),A)$.
Then $j(g)(j\rest V_{\lambda-2})=x$, and therefore
\[ \all^*_\mu k\ \big[M\sats\varphi(g(k),f(k))\big], \]
as desired.

Now suppose instead that $\lambda$
is a limit. We use the theory of
extenders under $\ZF$ developed in \cite{reinhardt_non-definability}.
Let $a\in\left<V_\lambda\right>^{<\om}$
and  $f\in M$ with $f:\left<V_\lambda\right>^{<\om}\to M$ be such that
\[ \all^*_{\mu_a}k\ \big[M\sats\exists x\ \varphi(x,f(k))\big].\]
Then by elementarity and since
 $j(M)\sats$``$V=\HOD(V_\lambda,j(A))$'',
we get again some $x\in j(M)$ such that
\[ j(M)\sats\varphi(x,j(f)(a)) \]
and some
  $y\in V_\lambda$ term $t$ such that
\[ j(M)\sats x=t(y,j(f)(a),j(A)).\]
By enlarging $y$, we can assume that
$a\sub y\in\left<V_\lambda\right>^{<\om}$.
In $M$, let
$g:\left<V_\lambda\right>^{<\om}\to M$ be defined
$g(k)=t(k,f^{ay}(k),A)$.
Then by elementarity, $j(g)(y)=x$ and
\[ \all^*_{\mu_y}k\ \big[M\sats\varphi(g(k),f^{ay}(k))]\big],\]
so we are done.
\end{proof}

In particular, $U=\Ult_0(M,\mu)$ is extensional,
and letting $i=i^M_\mu$, the ultrapower map,
 $i:M\to U$ is elementary.
We also have the factor map
$k:U\to j(M)$,
defined as usual by
$k([f]^{M}_\mu)=j(f)(j\rest V_\lambda)$
in the successor case, and in the limit case by
$k([a,f]^M_\mu)=j(f)(a)$.
By \L o\'{s}'s Theorem, $k$ is elementary.
Therefore $U$ is wellfounded, and moreover,
note that $V_\lambda\sub U$ and $k\rest V_\lambda=\id$,
and $k\com i=j\rest M$; also $k(\lambda)=\lambda=i(\lambda)$ since $j(\lambda)=\lambda$. Therefore
$U=L(V_{\lambda},j(A))=j(M)$.
So  $k:j(M)\to j(M)$ and if $k\neq\id$ then
 $\crit(k)>\lambda$.

Now if $k\neq\id$ then  $(V_{\lambda},j(A))^\#$ exists,
so suppose $k=\id$. Since $k\com i=j\rest M$,
therefore $i\rest\OR=k\com i\rest\OR= j\rest\OR$ by our assumption.
But $i$ is computed inside $L(V_{\lambda+1})$, because $\mu\in V_{\lambda+1}$.
And $j\rest V_{\lambda+1}$ is also computed in $L(V_{\lambda+1})$,
because it is the ultrapower map associated to $\Ult_0(V_{\lambda+1},\mu)$
(see \cite{cumulative_periodicity_pub_online_first}).
But since $L(V_{\lambda+1})\sats$``$V=\HOD_{V_{\lambda+1}}$'',
$j$ is determined by $j\rest(V_{\lambda+1}\cup\OR)$.
Therefore $j$ is definable over $L(V_{\lambda+1})$
from the parameter $\mu$, contradicting Suzuki \cite{suzuki_no_def_j}.

So $k\neq\id$ and $(V_{\lambda},j(A))^\#$ exists. Now $M=L(V_\lambda,A)$, 
$i\rest V_\lambda = j\rest V_\lambda$ and $i(A)=j(A)$.
Let $\Lambda$ be a proper class of
fixed points of $i$ which are Silver indiscernibles for 
$L(V_\lambda,j(A))$. Let $\iota=\min(\Lambda)$
and
\[ H=\Hull^M(V_\lambda\cup\{\lambda,A\}\cup\Lambda\cut\{\iota\}).\]
Note that $H$ is proper class, $H\elem M$ and  $V_\lambda\cup\{\lambda,A\}\sub H$, but $\iota\notin H$. (If $\iota=t^M(y,\lambda,A,\vec{\kappa})$ where $t$ is some term, $y\in V_\lambda$
and $\vec{\kappa}\in(\Lambda\cut\{\iota\})^{<\om}$,
then $\iota=t^{j(M)}(j(y),\lambda,j(A),\vec{\kappa})$,
contradicting the choice of $\Lambda$.)
So $M$ is the transitive collapse of $H$,
and letting $\pi:M\iso H$ be the uncollapse map,
then $\pi:M\to M$ is elementary
and non-identity, with $\lambda<\crit(\pi)$,
so we are done.
\end{proof}

\section{Relative consistencies}\label{sec:consistencies}

In this last section we deduce some relative consistencies, including some equiconsistencies,
from the results (and some of the methods) earlier in the paper. We begin by pointing out that ZFC + $I_{0,\lambda}$ (if consistent) does not imply that $V_{\lambda+1}^{\#}$ exists (this answers the second of Woodin's questions mentioned 
immediately following Remark \ref{rem:Silver}).

\begin{cor}\label{cor:getting_ZFC_I_0_no_sharp}
If the theory
``$\ZFC+\exists\lambda\ \big[\lambda\text{ is a 
limit}\wedge I_{0,\lambda}\big]$''
is consistent then so is the theory \[\ZFC+\exists\lambda\ \big[\lambda\text{ is a limit}\wedge I_{0, 
\lambda } \wedge\text{``}V_{\lambda+1}^\#\text{ does 
not exist''}\big].\]
\end{cor}

\begin{proof}
 Assume $\ZFC+I_{0,\lambda}$ (so $\lambda$ is a 
limit)
and let $j\in\mathscr{E}(\Lll)$ be proper where $\Lll=L(V_{\lambda+1})$. Let $k=j\rest V_{\lambda+2}^{\Lll}$.
So $\lambda=\kappa_\om(j)$ is a limit of measurables and $\cof(\lambda)=\om$,
so there is a wellorder of $V_\lambda$ of ordertype $\lambda$ in $V_{\lambda+1}$,
and $\lambda^+$ is regular. Let $\PP$ be the forcing
for adding a surjection $f:\lambda^+\to V_{\lambda+1}$,
with conditions functions $p:\alpha\to V_{\lambda+1}$
for some $\alpha<\lambda^+$. So $\PP\in\Lll\sub\Lll[k]$,
and $\PP$ is $(\lambda+1)$-closed. Let $G$ be $(V,\PP)$-generic.
Then
\[ V_{\lambda+1}\sub V_{\lambda+1}^{\Lll[k][G]}\sub V_{\lambda+1}^{V[G]}=V_{\lambda+1},\]
and 
$\Lll[k][G]\sats\ZFC+I_{0,\lambda}+\text{``}V_{\lambda+1}^\#\text{ does not 
exist}\text{''}$.
\end{proof}

We now turn to the consistency strengths of theories involving $\mathscr{E}(V_{\lambda+2})$. In the following result,
the consistency implication from theory \ref{item:1st_theory_no_choice} to theories \ref{item:2nd_theory_no_choice}
and \ref{item:3rd_theory_no_choice}
 is a direct consequence of  \ref{tm:main} and
\ref{cor:first_of_main}. The consistency implication from theory \ref{item:3rd_theory_no_choice}
to theory \ref{item:2nd_theory_no_choice} was first observed by Goldberg (Theorem 6.8 of \cite{goldberg_even_ordinals_v3}).
 But we include a proof here, as this implication follows quite  easily from methods discussed earlier in the paper (and which are related to calculations in \cite{reinhardt_non-definability} and \cite{ZF_extenders}).

\begin{tm}\label{tm:equicon_no_choice}
 The following theories are equiconsistent:
 \begin{enumerate}[label=\arabic*.,ref=\arabic*]
  \item\label{item:1st_theory_no_choice} ZF + $I_0$,
  \item\label{item:2nd_theory_no_choice} ZF + ``there is $\lambda\in\OR$ and $j\in\mathscr{E}(L(V_{\lambda+2}))$ with $\crit(j)<\lambda$'',   
  \item\label{item:3rd_theory_no_choice} ZF + ``there is $\lambda\in\OR$  such that $\mathscr{E}_{\nt}(V_{\lambda+2})\neq\emptyset$''.
 \end{enumerate}
\end{tm}
\begin{proof}
If $j\in\mathscr{E}(L(V_{\lambda+2}))$
with $\crit(j)<\lambda$ then
$j\rest L(V_{\lambda+1})\in\mathscr{E}(L(V_{\lambda+1}))$ and
$j\rest V_{\lambda+2}\in\mathscr{E}(V_{\lambda+2})$.
So the equiconsistency of theories  \ref{item:1st_theory_no_choice}
and \ref{item:2nd_theory_no_choice}
is an immediate corollary of Theorem \ref{tm:main},
and the consistency of theory  \ref{item:3rd_theory_no_choice}
relative to theory \ref{item:2nd_theory_no_choice}
is immediate.

So assume ZF and $k\in\mathscr{E}(V_{\lambda+2})$ is non-trivial;
it suffices to see that $I_{0,\lambda}$ holds.  Let $\mu_k$, derived from seed $k\rest V_\lambda$, be  as usual. Let $\Lll=L(V_{\lambda+1})$. It is enough to see \begin{equation}\label{eqn:Ult_0_sats_Los_and_is_wfd}U=\Ult_0(\Lll,\mu_k)\text{ satisfies \L o\'s's Theorem and is wellfounded},\end{equation}
 since then $U$ is isomorphic to $\Lll$,
 and the associated ultrapower map $i:\Lll\to\Lll$ is in $\mathscr{E}(\Lll)$ and $k\sub i$.
 
  The proof of line (\ref{eqn:Ult_0_sats_Los_and_is_wfd}) is
  a very slight variant of the proof of  Theorem \ref{tm:extensions_when_cof(Theta)>om} part \ref{item:k_extends_to_exactly_one_j}.
 Note that $V_{\lambda+2}^{\Lll}$ is definable without parameters over $V_{\lambda+2}$. So $k'\in\mathscr{E}(V_{\lambda+2}^{\Lll})$ where $k'=k\rest V_{\lambda+2}^{\Lll}$. So using Proposition \ref{prop:extensions},
 let $k''$ be the unique extension of $k'$ 
 such that $k''\in\mathscr{E}(L_\Theta(V_{\lambda+1}))$.
 Let $\Theta=\Theta^{\Lll}_{V_{\lambda+1}}$.
 By Theorem \ref{tm:extensions_when_cof(Theta)>om} part \ref{item:k_extends_to_exactly_one_j} applied to $k''$,
 we may assume $\cof(\Theta)=\om$.
 As mentioned in the proof of that result,
 there is only one point in the proof of the existence of $j$ there where the assumption that $\cof(\Theta)>\om$ is used:
 to see that $\Theta^C<\Theta$ (in the notation of that proof). But this is irrelevant here,
 because note that we anyway have a surjection $\pi:V_{\lambda+1}\to C$ in $V$, and hence $C$ is coded by some $\bar{C}\in V_{\lambda+2}$ (as there, DC is not needed, as the construction of $C$ was completely explicit). Since $k''$ agrees with $k$ over $V_{\lambda+2}^{\Lll}$, the rest of the proof is easily adapted, considering $k(\bar{C})$ (instead of $k''(C)$, which doesn't make sense if $\Theta^C=\Theta$).
\end{proof}

Actually  the same proof gives a more general result.
\begin{dfn}
 Given a non-negative integer $n$ (of the metatheory), write $I_{0,\lambda,2n}$ for the assertion (in parameter $\lambda$)
``there is a limit ordinal $\bar{\lambda}$ such that $\lambda=\bar{\lambda}+2n$
and $I_{0,\lambda}$ holds''.
\end{dfn}

\begin{tm}
Fix a non-negative integer $n$ of the meta-theory.
Then the following theories are equiconsistent:
\begin{enumerate}[label=\arabic*.,ref=\arabic*]
\item
ZF + ``there is $\lambda$ such that $I_{0,\lambda,2n}$'',
\item 
$\ZF$ + ``there are $\lambda,\bar{\lambda}$ with $\lambda=\bar{\lambda}+2n$ and $k\in\mathscr{E}_{\nt}(L(V_{\lambda+2}))$ with $\crit(k)<\lambda$'',
\item 
$\ZF$ + ``there are $\lambda,\bar{\lambda}$ with $\lambda=\bar{\lambda}+2n$ and $\mathscr{E}_{\nt}(V_{\lambda+2})\neq\emptyset$''.
\end{enumerate}
\end{tm}
\begin{proof}
  The proof of Theorem \ref{tm:equicon_no_choice} still goes through.
\end{proof}

We now consider a variant of Theorem \ref{tm:equicon_no_choice} for ZFC + $I_0$. Goldberg credited both directions of the following result to the author in \cite{goldberg_even_ordinals_v3}. This was, however, generous, since the author initially made the extra assumption of the existence of $V_{\lambda+1}^{\#}$
for Theorem \ref{tm:main}, and the consistency  implication from theory \ref{item:second_theory}
to theory \ref{item:first_theory} was actually first pointed out by Goldberg himself in \cite{goldberg_even_ordinals_v3}.

\begin{tm}\label{tm:equicon_choice}
 The following theories are equiconsistent:
 \begin{enumerate}[label=\arabic*.,ref=\arabic*] \item\label{item:first_theory} ZFC + $I_0$,
 \item\label{item:second_theory} ZF + ``there is an ordinal $\lambda$ such that $\lambda$-DC holds and $\mathscr{E}_{\nt}(V_{\lambda+2})\neq\emptyset$''.
 \end{enumerate}
\end{tm}

\begin{proof}
If ZFC + $I_0$ is consistent,  then so is
 theory \ref{item:second_theory},
 by  \ref{cor:first_of_main}.
 
 So assume ZF + $\lambda$-DC + $k\in\mathscr{E}_{\nt}(V_{\lambda+2})$. We may assume $\lambda=\kappa_\om(j)$. Let $\Lll=L(V_{\lambda+1})$.
 By the proof of Theorem \ref{tm:equicon_no_choice},
 $I_0$ holds as witnessed by the ultrapower map $i$ considered there.
So by Corollary \ref{cor:first_of_main}, $\Lll[i]\models\lambda$-DC,
  and then as in the proof of Corollary \ref{cor:getting_ZFC_I_0_no_sharp} (the proof of which goes through with just $\lambda$-DC),
  there is a forcing extension of $\Lll[i][G]$ of $\Lll[i]$ which models ZFC + $I_0$.
\end{proof}

For the proof of the next corollary we use the following theorem of Cramer (see Theorem 3.9 of \cite{cramer}):

\begin{factmine}[Cramer]\label{fact:Cramer}
Assume $\ZFC$, $\lambda$ is a limit and $V_{\lambda+1}^\#$ exists.
Suppose  that
$k'\in\mathscr{E}_{\nt}(L_\om(V_{\lambda+1},V_{\lambda+1}^\#))\neq\emptyset$.
Then $I_{0,\bar{\lambda}}$ holds at some $\bar{\lambda}<\lambda$.
\end{factmine}

\begin{cor}\label{cor:strong_sharp_reflection}:
Assume $\ZF+\lambda$-DC
where 
$\lambda$ is a limit, $V_{\lambda+1}^\#$ exists
and
$k\in\mathscr{E}_{\nt}(L_\om(V_{\lambda+1},V_{\lambda+1}^\#))\neq\emptyset$.
Then there is a limit $\bar{\lambda}<\lambda$ such that $I_{0,\bar{\lambda}}$ 
holds,
and there is a proper class inner model $N$ such that:
\begin{itemize}
\item[--] $V_{\bar{\lambda}+1}\sub N$ and 
$V_{\bar{\lambda}+2}^N=V_{\bar{\lambda}+2}\cap L(V_{\bar{\lambda}+1})$, and
\item[--] 
$N\sats\ZF+\bar{\lambda}$-$\DC+I_{0,\bar{\lambda}}+ \mathscr{E}_{\nt}(V_{\bar{\lambda}+2})\neq\emptyset$.
\end{itemize}
\end{cor}
\begin{proof}
Setting $k'=k$,
we have everything required to apply
Fact \ref{fact:Cramer}, except that
we only have $\lambda$-$\DC$, not full
$\AC$. But 
note we may assume
$\lambda=\kappa_\om(k)$. So $\lambda$ is a strong limit cardinal,
$V_\lambda\sats\ZFC$, and $\lambda^+$ is regular.
 Let $\PP,G$ be as in the proof of \ref{cor:getting_ZFC_I_0_no_sharp}, and now work in 
$M=L(V_{\lambda+1},V_{\lambda+1}^\#,k',G)$.  Note that $M\sats\ZFC$,
 so we can apply Fact 
\ref{fact:Cramer} in $M$.
 
 So fix $\bar{\lambda}<\lambda$ such that  $M\sats I_{0,\bar{\lambda}}$,
 as witnessed by $\bar{j}$, with $\bar{j}$ proper.
 Then by Theorem \ref{tm:main}, $N=L(V_{\bar{\lambda}+1})[\bar{j}]$
 has the right properties, and because $V_{\lambda+1}^M=V_{\lambda+1}^V$,
 and $\bar{j}$ is determined by $\mu_{\bar{j}}$, we actually
 have $N\sub V$, so we are done. 
 \end{proof}

\section{Questions}

\begin{enumerate}[label=\arabic*.,ref=\arabic*]
\item Is there a ``natural'' theory extending ZFC which is equiconsistent with $\ZF$ + ``there is a non-trivial
elementary $j:V_{\lambda+2}\to V_{\lambda+2}$''?
Note here that theory \ref{item:second_theory}
in Theorem \ref{tm:equicon_choice}
incorporates $\lambda$-DC.
\item What is the consistency strength of $\ZF+$``there is an elementary 
$j:V_{\lambda+3}\to V_{\lambda+3}$''? Are there $\ZFC$-compatible large 
cardinals beyond this level?\footnote{There is of course a 
trivial kind of answer to ``beyond this level'': 
$\ZFC+\mathrm{Con}(\ZF+j:V_{\lambda+3}\to V_{\lambda+3})$.
But there might be a more interesting answer. Maybe equiconsistency
is the better question.} Or equiconsistent?
\end{enumerate}

We finish with a real world prediction:

\begin{conj}
 No being will derive a contradiction from the
 theory $\ZF+$``there is a Reinhardt cardinal''
in the next $\kappa$ years, where $\kappa$ is the least Reinhardt cardinal.
\end{conj}

\section*{Acknowledgments}
This work supported by  Deutsche Forschungsgemeinschaft (DFG, German 
Research
Foundation) under Germany's Excellence Strategy EXC 2044-390685587,
Mathematics M\"unster: Dynamics-Geometry-Structure.

This research was funded in part by the Austrian Science Fund (FWF)
[10.55776/Y1498]. For open access purposes, the author has applied a CC BY public copyright license to any author accepted manuscript version arising from this
submission.

The author thanks Gabriel Goldberg 
for 
bringing Woodin's questions to his attention, which we discuss 
here, and other related questions, and discussions on the topic. We also thank 
Goldberg for the key observation that $V_{\lambda+1}^\#$
was not needed for the main relative consistency fact.

Finally the author thanks the anonymous referee
for their work and many useful suggestions for improvements.

\bibliographystyle{unsrt}
\bibliography{bibliography_Vlp2_submission_2}

\end{document}